\documentclass[
final,nomarks
]{dmtcs-episciences}

\usepackage[utf8]{inputenc}
\usepackage{subfigure}

\usepackage[round]{natbib}

\usepackage{tikz} 
\usepackage{hyperref}
\usepackage{eqnarray,amsthm, amssymb, amsmath,verbatim,epsfig}
\usepackage{mathrsfs}
\usepackage{enumerate}

\usetikzlibrary{matrix,trees,arrows}
\usetikzlibrary{positioning}
\usetikzlibrary{fit}
\usetikzlibrary{patterns}

\newtheorem{theorem}{Theorem}
\newtheorem{lemma}{Lemma}

\newtheorem{corollary}{Corollary}

\newtheorem{conjecture}{Conjecture}

\newcommand{\tref}[1]{Theorem \ref{theorem:#1}}
\newcommand{\fref}[1]{Figure \ref{fig:#1}}
\newcommand{\coref}[1]{Corollary \ref{corollary:#1}}

\newcommand{\lift}{\mathrm{lift}}
\newcommand{\red}{\mathrm{red}}
\newcommand{\Sn}[1]{\mathcal{S}_{#1}}
\newcommand{\Des}[1]{\mathrm{Des}}

\newcommand{\Qmm}[1]{Q_{132}^{(#1)}(t,x)}
\newcommand{\Qm}[1]{Q_{123}^{(#1)}(t,x)}
\newcommand{\Qmmx}[2]{Q_{132}^{(#1)}(#2)}
\newcommand{\Qmx}[2]{Q_{123}^{(#1)}(#2)}
\newcommand{\Qmmn}[2]{Q_{#2,132}^{(#1)}(x)}
\newcommand{\Qmn}[2]{Q_{#2,123}^{(#1)}(x)}

\newcommand{\Qmz}[2]{Q_{123}^{(0,\binom{#1}{#2},0,0)}(t,x_0,x_1)}
\newcommand{\Qmzn}[3]{Q_{#3,123}^{(0,\binom{#1}{#2},0,0)}(x_0,x_1)}
\newcommand{\Qmznx}[4]{Q_{#3,123}^{(0,\binom{#1}{#2},0,0)}(#4)}
\newcommand{\Qmzx}[3]{Q_{123}^{(0,\binom{#1}{#2},0,0)}(t,#3)}

\newcommand{\MMP}{\mathrm{MMP}}
\newcommand{\mmp}{\mathrm{mmp}}

\newcommand{\fillll}[3]{\node at (#1-.5,#2-.5) {\begin{math}#3\end{math}}}
\newcommand{\filllll}[2]{\node at (#1-.5,#2-.5) {\begin{math}#2\end{math}}}
\newcommand{\filllllg}[2]{\node at (#1-.5,#2-.5) {\color{green!50!black}\begin{math}#2\end{math}}}
\newcommand{\fillcross}[2]{\draw[thick] (#1-1,#2-1)--(#1,#2);\draw[thick] (#1-1,#2)--(#1,#2-1)}
\newcommand{\fillgcross}[2]{\draw[thick,green!50!black] (#1-1,#2-1)--(#1,#2);\draw[thick,green!50!black] (#1-1,#2)--(#1,#2-1)}
\newcommand{\thn}[1]{\begin{math}#1^\textnormal{th}\end{math}}
\newcommand{\fillshade}[1]{\foreach \x/\y in {#1}{\path[fill,blue!20!white] (\x-1,\y-1) rectangle (\x,\y);}}

\newcommand{\shadetheboxes}[1]{
	\foreach \x/\y in {#1}
	\fill[pattern color = black!65, pattern=north east lines] (\x,\y) rectangle +(1,1);
}

\newcommand{\drawthegrid}[1]{
	\draw (0.01,0.01) grid (#1+0.99,#1+0.99);
}

\newcommand{\drawtheclpattern}[1]{
	\foreach \x/\y in {#1}
	\filldraw (\x,\y) circle (6pt);
}

\newcommand{\drawspecialbox}[1]{
	\foreach \x/\y/\z/\w/\A in {#1}
	{
		\fill[color = white!100, opacity=1, rounded corners = 1.5pt] (\x+0.125,\y+0.125) rectangle (\z-0.125,\w-0.125);
		\draw[color = black, rounded corners = 1.5pt] (\x+0.125,\y+0.125) rectangle (\z-0.125,\w-0.125);
		\fill[black] (\x/2+\z/2,\y/2+\w/2) node {\begin{math}\scriptstyle\A\end{math}};
	}
}

\newcommand{\mmpattern}[5]{									
	\raisebox{0.6ex}{
		\begin{tikzpicture}[scale=0.35, baseline=(current bounding box.center), #1]
		\useasboundingbox (0.0,-0.1) rectangle (#2+1.4,#2+1.1);
		\shadetheboxes{#4}
		\drawthegrid{#2}
		\drawspecialbox{#5}
		\drawtheclpattern{#3}
		\end{tikzpicture}}
}

\author{Dun Qiu
  \and Jeffrey Remmel}
\title{Quadrant marked mesh patterns in 123-avoiding permutations}
\affiliation{
  Department of Mathematics, University of California San Diego, La Jolla, USA}
\received{2017-5-2}
\revised{2017-10-3}
\accepted{2018-6-26}
\begin{document}
\publicationdetails{19}{2018}{2}{12}{3297}
\maketitle
\begin{abstract}
Given a permutation \begin{math}\sigma = \sigma_1 \ldots \sigma_n\end{math} in the symmetric group 
\begin{math}\mathcal{S}_{n}\end{math}, we say that \begin{math}\sigma_i\end{math} matches the quadrant marked mesh pattern 
\begin{math}\mathrm{MMP}(a,b,c,d)\end{math} in \begin{math}\sigma\end{math} if there are at least 
\begin{math}a\end{math} points to the right of \begin{math}\sigma_i\end{math} in \begin{math}\sigma\end{math} which are greater than 
\begin{math}\sigma_i\end{math}, at least \begin{math}b\end{math} points to the left of \begin{math}\sigma_i\end{math} in \begin{math}\sigma\end{math} which 
are greater than \begin{math}\sigma_i\end{math},  at least \begin{math}c\end{math} points to the left of 
\begin{math}\sigma_i\end{math} in \begin{math}\sigma\end{math} which are smaller  than \begin{math}\sigma_i\end{math}, and 
at least \begin{math}d\end{math} points to the right of \begin{math}\sigma_i\end{math} in \begin{math}\sigma\end{math} which 
are smaller than \begin{math}\sigma_i\end{math}.
Kitaev, Remmel, and Tiefenbruck 
systematically studied the distribution of the number of matches 
of \begin{math}\mathrm{MMP}(a,b,c,d)\end{math} in 132-avoiding permutations. The operation 
of reverse and complement on permutations allow one to translate 
their results to find the distribution of the number of 
\begin{math}\mathrm{MMP}(a,b,c,d)\end{math} matches in 231-avoiding, 213-avoiding, and 312-avoiding  permutations.
In this paper, 
we study  the distribution of the number of matches 
of \begin{math}\mathrm{MMP}(a,b,c,d)\end{math} in 123-avoiding permutations.  
We provide explicit recurrence relations to enumerate our objects which 
can be used to give closed forms for the generating functions associated 
with such distributions. In many  cases, we provide combinatorial explanations of the coefficients that appear in our generating functions.
\end{abstract}

\section{Introduction}

Given a sequence \begin{math}w = w_1 \ldots w_n\end{math} of distinct integers,
let \begin{math}\red[w]\end{math} be the permutation founded by replacing the
\begin{math}i\end{math}-th smallest integer that appears in \begin{math}\sigma\end{math} by \begin{math}i\end{math}.  For
example, if \begin{math}\sigma = 2754\end{math}, then \begin{math}\red[\sigma] = 1432\end{math}.  Given a
permutation \begin{math}\tau=\tau_1 \ldots \tau_j\end{math} in the symmetric group \begin{math}S_j\end{math}, we say that the pattern \begin{math}\tau\end{math} \emph{occurs} in \begin{math}\sigma = \sigma_1 \ldots \sigma_n \in \Sn{n}\end{math} provided   there exist \begin{math}1 \leq i_1 < \cdots < i_j \leq n\end{math} such that 
\begin{math}\red[\sigma_{i_1} \ldots \sigma_{i_j}] = \tau\end{math}.   We say 
that a permutation \begin{math}\sigma\end{math} \emph{avoids} the pattern \begin{math}\tau\end{math} if \begin{math}\tau\end{math} does not 
occur in \begin{math}\sigma\end{math}. Let \begin{math}\Sn{n}(\tau)\end{math} denote the set of permutations in \begin{math}\Sn{n}\end{math} 
which avoid \begin{math}\tau\end{math}. In the theory of permutation patterns,   \begin{math}\tau\end{math} is called a \emph{classical pattern}. See \cite{kit} for a comprehensive introduction to 
patterns in permutations.

The main goal of this paper is to study the distribution of 
quadrant marked mesh patterns in 123-avoiding permutations. 
The notion of mesh patterns was introduced by \cite{BrCl} to provide explicit expansions for certain permutation statistics as, possibly infinite, linear combinations of (classical) permutation patterns.  This notion was further studied in \cite{AKV,HilJonSigVid,kitlie,kitrem,Ulf}.  \cite{kitrem} initiated the systematic study of distribution of quadrant marked mesh patterns on permutations. The study was extended to 132-avoiding permutations by \cite{KRT1,KRT2,KRT3}.  \cite{kitrem2,kitrem3} also studied the distribution of quadrant marked 
mesh patterns in up-down and down-up permutations. 

Let \begin{math}\sigma = \sigma_1 \ldots \sigma_n\end{math} be a permutation written in one-line notation. We will consider the 
\emph{graph} of \begin{math}\sigma\end{math}, \begin{math}G(\sigma)\end{math}, to be the set of points \begin{math}(i,\sigma_i)\end{math} for 
\begin{math}i =1, \ldots, n\end{math}.  For example, the graph of the permutation 
\begin{math}\sigma = 471569283\end{math} is pictured in Figure 
\ref{fig:basic}.  Then if we draw a coordinate system centered at a 
point \begin{math}(i,\sigma_i)\end{math}, we will be interested in  the points that 
lie in the four quadrants I, II, III, and IV of that 
coordinate system as pictured 
in Figure \ref{fig:basic}.  For any \begin{math}a,b,c,d \in  
\mathbb{N}\end{math}, where \begin{math}\mathbb{N} = \{0,1,2, \ldots \}\end{math} is the set of 
natural numbers, and any \begin{math}\sigma = \sigma_1 \ldots \sigma_n \in \Sn{n}\end{math}, 
we say that \begin{math}\sigma_i\end{math} matches the 
quadrant marked mesh pattern \begin{math}\MMP(a,b,c,d)\end{math} in \begin{math}\sigma\end{math} if in \begin{math}G(\sigma)\end{math}, there are at least \begin{math}a\end{math} points in quadrant I, 
at least \begin{math}b\end{math} points in quadrant II, at least \begin{math}c\end{math} points in quadrant 
III, and at least \begin{math}d\end{math} points in quadrant IV relative to the coordinate system which has the point \begin{math}(i,\sigma_i)\end{math} as its origin.  
For example, 
if \begin{math}\sigma = 471569283\end{math}, the point \begin{math}\sigma_4 =5\end{math}  matches 
the marked mesh pattern \begin{math}\MMP(2,1,2,1)\end{math} since, in \begin{math}G(\sigma)\end{math} relative 
to the coordinate system with the origin at \begin{math}(4,5)\end{math},  
there are 3 points in quadrant I, 
1 point in quadrant II, 2 points in quadrant III, and 2 points in 
quadrant IV.  Note that if a coordinate 
in \begin{math}\MMP(a,b,c,d)\end{math} is 0, then there is no condition imposed 
on the points in the corresponding quadrant.  Another way to state this definition
is to say that \begin{math}\sigma_i\end{math} matches the marked mesh pattern 
\begin{math}\MMP(a,b,c,d)\end{math} in \begin{math}\sigma\end{math} if there are at least 
\begin{math}a\end{math} points to the right of \begin{math}\sigma_i\end{math} in \begin{math}\sigma\end{math} which are greater than 
\begin{math}\sigma_i\end{math}, at least \begin{math}b\end{math} points to the left of \begin{math}\sigma_i\end{math} in \begin{math}\sigma\end{math} which 
are greater than \begin{math}\sigma_i\end{math},  at least \begin{math}c\end{math} points to the left of 
\begin{math}\sigma_i\end{math} in \begin{math}\sigma\end{math} which are smaller  than \begin{math}\sigma_i\end{math}, and 
at least \begin{math}d\end{math} points to the right of \begin{math}\sigma_i\end{math} in \begin{math}\sigma\end{math} which 
are smaller than \begin{math}\sigma_i\end{math}.

In addition, we shall consider the patterns  \begin{math}\MMP(a,b,c,d)\end{math} where 
\begin{math}a,b,c,d \in \mathbb{N} \cup \{\emptyset\}\end{math}. Here when 
a coordinate of \begin{math}\MMP(a,b,c,d)\end{math} is the empty set, then for \begin{math}\sigma_i\end{math} to match  
\begin{math}\MMP(a,b,c,d)\end{math} in \begin{math}\sigma = \sigma_1 \ldots \sigma_n \in \Sn{n}\end{math}, 
it must be the case that there are no points in \begin{math}G(\sigma)\end{math} relative 
to the coordinate system with the origin at \begin{math}(i,\sigma_i)\end{math} in the corresponding 
quadrant. For example, if \begin{math}\sigma = 471569283\end{math}, the point 
\begin{math}\sigma_3 =1\end{math} matches 
the marked mesh pattern \begin{math}\MMP(4,2,\emptyset,\emptyset)\end{math} since,  
in \begin{math}G(\sigma)\end{math} relative 
to the coordinate system with the origin at \begin{math}(3,1)\end{math}, 
there are 6 points in quadrant I, 
2 points in quadrant II, no  points in quadrants III  and IV.   We let 
\begin{math}\mmp^{(a,b,c,d)}(\sigma)\end{math} denote the number of \begin{math}i\end{math} such that 
\begin{math}\sigma_i\end{math} matches \begin{math}\MMP(a,b,c,d)\end{math} in \begin{math}\sigma\end{math}. For example, \begin{math}\mmp^{(2,2,0,0)}(\sigma)=2\end{math} for \begin{math}\sigma = 471569283\end{math}, since \begin{math}\sigma_3=1\end{math} and \begin{math}\sigma_7=2\end{math} match \begin{math}\MMP(2,2,0,0)\end{math} in \begin{math}\sigma\end{math}.

\begin{figure}[ht]
	\centerline{\scalebox{.66}{\includegraphics{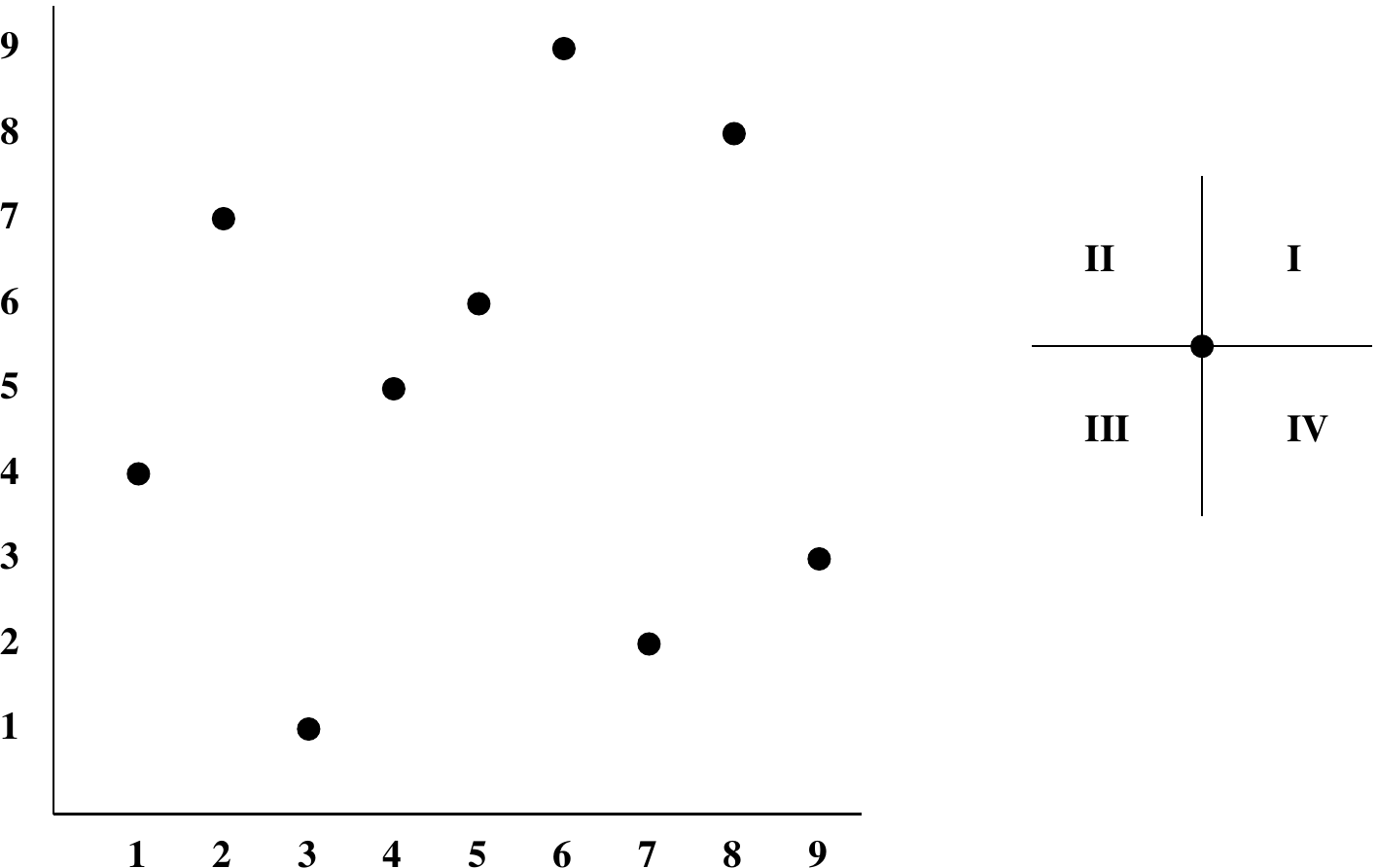}}}
	\caption{The graph of \(\sigma = 471569283\)}
	\label{fig:basic}
\end{figure}

Next we give some examples of how the (two-dimensional) notation of  \cite{Ulf} for marked mesh patterns corresponds to our (one-line) notation for quadrant marked mesh patterns. For example,

\[
\MMP(0,0,k,0)=\mmpattern{scale=2.3}{1}{1/1}{}{0/0/1/1/k}\hspace{-0.25cm},\  \MMP(k,0,0,0)=\mmpattern{scale=2.3}{1}{1/1}{}{1/1/2/2/k}\hspace{-0.25cm},
\]

\[
\MMP(0,a,b,c)=\mmpattern{scale=2.3}{1}{1/1}{}{0/1/1/2/a} \hspace{-2.02cm} \mmpattern{scale=2.3}{1}{1/1}{}{0/0/1/1/b} \hspace{-2.02cm} \mmpattern{scale=2.3}{1}{1/1}{}{1/0/2/1/c} \ \mbox{ and }\ \ \ \MMP(0,0,\emptyset,k)=\mmpattern{scale=2.3}{1}{1/1}{0/0}{1/0/2/1/k}\hspace{-0.25cm}.
\]

Given a permutation \begin{math}\tau = \tau_1 \ldots \tau_j \in S_j\end{math}, it is a natural question to study the 
distribution of quadrant marked mesh patterns in \begin{math}\Sn{n}(\tau)\end{math}. That is, one wants to 
study generating function of the form 
\begin{equation} \label{Rabcd}
Q_{\tau}^{(a,b,c,d)}(t,x) = 1 + \sum_{n\geq 1} t^n  Q_{n,\tau}^{(a,b,c,d)}(x)
\end{equation}
where for  any \begin{math}a,b,c,d \in \{\emptyset\} \cup \mathbb{N}\end{math}, 
\begin{equation} \label{Rabcdn}
Q_{n,\tau}^{(a,b,c,d)}(x) = \sum_{\sigma \in \Sn{n}(\tau)} x^{\mmp^{(a,b,c,d)}(\sigma)}.
\end{equation}
For any \begin{math}a,b,c,d\end{math} we let \begin{math}Q_{n,\tau}^{(a,b,c,d)}(x)|_{x^k}\end{math} denote 
the coefficient of \begin{math}x^k\end{math} in \begin{math}Q_{n,\tau}^{(a,b,c,d)}(x)\end{math}. 
Given a permutation \begin{math}\sigma = \sigma_1 \sigma_2 \ldots \sigma_n \in \Sn{n}\end{math}, we let the reverse of 
\begin{math}\sigma\end{math}, \begin{math}\sigma^r\end{math}, be defined by \begin{math}\sigma^r = \sigma_n \ldots \sigma_2 \sigma_1\end{math}, and the 
complement of \begin{math}\sigma\end{math}, \begin{math}\sigma^c\end{math}, be defined by \begin{math}\sigma^c = (n+1 -\sigma_1) (n+1-\sigma_2) \ldots  
(n+1 -\sigma_n)\end{math}.  It is easy to see that the family of generating 
functions \begin{math}Q_{\tau^r}^{(a,b,c,d)}(t,x)\end{math}, \begin{math}Q_{\tau^c}^{(a,b,c,d)}(t,x)\end{math}, and 
\begin{math}Q_{(\tau^r)^c}^{(a,b,c,d)}(t,x)\end{math} can be obtained from the family of generating 
functions \begin{math}Q_{\tau}^{(a,b,c,d)}(t,x)\end{math}.

\cite{KRT1,KRT2,KRT3} systematically studied the generating functions 
\begin{math}Q_{132}^{(a,b,c,d)}(t,x)\end{math}.   Since \begin{math}\Sn{n}(132)\end{math} is closed under inverses, there 
is a natural symmetry on these generating functions. That is, we have the following 
lemma. 
\begin{lemma}\label{sym}  (\cite{KRT1})
	For any \begin{math}a,b,c,d \in \{\emptyset\} \cup \mathbb{N}\end{math}, 
	\begin{equation}
	Q_{n,132}^{(a,b,c,d)}(x) = Q_{n,132}^{(a,d,c,b)}(x). 
	\end{equation}
\end{lemma}

In \cite{KRT1}, Kitaev, Remmel and Tiefenbrick  proved the following.

\begin{theorem}\label{thm:Qk000} (\cite[Theorem 4]{KRT1})
	\begin{equation}\label{eq:Q0000}
	Q_{132}^{(0,0,0,0)}(t,x) =  C(xt) = \frac{1-\sqrt{1-4xt}}{2xt}
	\end{equation}
	and, for \begin{math}k \geq 1\end{math}, 
	\begin{equation}\label{Qk000}
	Q_{132}^{(k,0,0,0)}(t,x) = \frac{1}{1-tQ_{132}^{(k-1,0,0,0)}(t,x)}.
	\end{equation}
\end{theorem}

\begin{theorem}\label{thm:Qk0e0} (\cite[Theorem 15]{KRT1})
	\begin{equation}\label{00e0gf--}
	Q_{132}^{(0,0,\emptyset,0)}(t,x) = 
	\frac{(1+t-tx)-\sqrt{(1+t-tx)^2 -4t}}{2t}.
	\end{equation}
	For \begin{math}k \geq 1\end{math}, 
	\begin{equation}\label{k0e0gf}
	Q_{132}^{(k,0,\emptyset,0)}(t,x) = 
	\frac{1}{1-t Q_{132}^{(k-1,0,\emptyset,0)}(t,x)}.
	\end{equation} 
\end{theorem}

\begin{theorem}\label{thm:Q00k0} (\cite[Theorem 8]{KRT1}) 
	For \begin{math}k \geq 1\end{math}, 
	\begin{align}\label{gf00k0}
	Q_{132}^{(0,0,k,0)}(t,x)&=\frac{1+(tx-t)(\sum_{j=0}^{k-1}C_jt^j) - 
		\sqrt{(1+(tx-t)(\sum_{j=0}^{k-1}C_jt^j))^2 -4tx}}{2tx}\nonumber\\
	&=\frac{2}{1+(tx-t)(\sum_{j=0}^{k-1}C_jt^j) + \sqrt{(1+(tx-t)(\sum_{j=0}^{k-1}C_jt^j))^2 -4tx}}.
	\end{align}
\end{theorem}

By Lemma \ref{sym}, \begin{math}Q_{132}^{(0,k,0,0)}(t,x) =Q_{132}^{(0,0,0,k)}(t,x)\end{math} so the 
remaining two cases of \begin{math}Q_{132}^{(a,b,c,d)}(t,x)\end{math} where \begin{math}a,b,c,d \in \mathbb{N}\end{math} and 
exactly one of \begin{math}a,b,c,d\end{math} is not zero is covered by the following theorem. 

\begin{theorem}\label{thm:Q0k00} (\cite[Theorem 12]{KRT1})
	\begin{equation}\label{Q0100}
	Q_{132}^{(0,1,0,0)}(t,x) = \frac{1}{1-tC(tx)}.
	\end{equation}
	For \begin{math}k > 1\end{math}, 
	\begin{equation}\label{Q0100-}
	Q_{132}^{(0,k,0,0)}(t,x) = \frac{1+t\sum_{j=0}^{k-2} C_j t^j
		(Q_{132}^{(0,k-1-j,0,0)}(t,x) -C(tx))}{1-tC(tx)}
	\end{equation}
	and 
	\begin{equation}\label{x=0Q0100-}
	Q_{132}^{(0,k,0,0)}(t,0) = \frac{1+t\sum_{j=0}^{k-2} C_j t^j
		(Q_{132}^{(0,k-1-j,0,0)}(t,0) -1)}{1-t}.
	\end{equation}
\end{theorem}

In \cite{KRT2}, Kitaev, Remmel, and Tiefenbruck used the results above to cover the cases 
\begin{math}Q_{132}^{(a,b,c,d)}(t,x)\end{math} where \begin{math}a,b,c,d \in \mathbb{N}\end{math} and 
exactly two of \begin{math}a,b,c,d\end{math} are not zero. For example, they proved the following.

\begin{theorem}\label{thm:Qk0l0} For all \begin{math}k, \ell \geq 1\end{math}, 
	\begin{equation}\label{k0l0gf}
	Q_{132}^{(k,0,\ell,0)}(t,x) = 
	\frac{1}{1-t Q_{132}^{(k-1,0,\ell,0)}(t,x)}.
	\end{equation}
\end{theorem}

\begin{theorem} \label{thm:k00l}
	For all \begin{math}k, \ell \geq 1\end{math}, 
	\begin{multline}\label{Qk00lgf-}
	Q_{132}^{(k,0,0,\ell)}(t,x) = \\
	\frac{C_\ell t^\ell + \sum_{j=0}^{\ell -1} C_j t^j (1 -tQ_{132}^{(k-1,0,0,0)}(t,x)
		+t(Q_{132}^{(k-1,0,0,\ell-j)}(t,x)-\sum_{s=0}^{\ell -j -1}C_s t^s))}{1-tQ_{132}^{(k-1,0,0,0)}(t,x)}.
	\end{multline}
\end{theorem}

Finally, in \cite{KRT3}, Kitaev, Remmel, and Tiefenbruck used these results to 
find generating functions to obtain similar recursions for \begin{math}Q_{132}^{(a,b,c,d)}(t,x)\end{math} 
for arbitrary \begin{math}a,b,c,d \in \mathbb{N}\end{math}.

The situation for the generating functions \begin{math}Q_{123}^{(a,b,c,d)}(t,x)\end{math} is different. First of all it 
is easy to see that \begin{math}\Sn{n}(123)\end{math} is closed under the operation reverse-complement. Thus we 
have the following lemma. 

\begin{lemma}\label{sym2} 
	For any \begin{math}a,b,c,d \in \{\emptyset\} \cup \mathbb{N}\end{math}, 
	\begin{equation}
	Q_{n,123}^{(a,b,c,d)}(x) = Q_{n,123}^{(c,d,a,b)}(x). 
	\end{equation}
\end{lemma}

Next it is obvious that if there is a \begin{math}\sigma_i\end{math} in \begin{math}\sigma =\sigma_1 \ldots \sigma_n \in \Sn{n}\end{math} such 
that \begin{math}\sigma_i\end{math} matches \begin{math}\MMP(a,b,c,d)\end{math} where \begin{math}a,c \geq 1\end{math}, then \begin{math}\sigma\end{math} contains an occurrence of \begin{math}123\end{math}. 
Thus there are no permutations \begin{math}\sigma \in \Sn{n}(123)\end{math} that can 
match a quadrant marked mesh pattern \begin{math}\MMP(a,b,c,d)\end{math} where \begin{math}a,c \geq 1\end{math}.  Thus if 
\begin{math}a \geq 1\end{math}, then \begin{math}Q_{123}^{(a, b, 0, d)}(t,x) = Q_{123}^{(a, b, \emptyset, d)}(t,x)\end{math}.
Our first major result is that for all \begin{math}a,b,d \in \mathbb{N}\end{math} such that \begin{math}a > 0\end{math},  
\begin{equation}\label{reduce1}
Q_{123}^{(a,b, \emptyset, d)}(t,x) = Q_{132}^{(a,b, \emptyset, d)}(t,x).
\end{equation}
We will prove this result by using a bijection of \cite{Kr} between \begin{math}\Sn{n}(132)\end{math} and \begin{math}\mathcal{D}_n\end{math}, the set of Dyck paths of length \begin{math}2n\end{math}, and a bijection of \cite{EliDeu} between \begin{math}\Sn{n}(123)\end{math} and \begin{math}\mathcal{D}_n\end{math}. It is easier to 
compute the generating functions of the form \begin{math}Q_{132}^{(a, b, \emptyset, d)}(t,x)\end{math}, so we will use them to compute \begin{math}Q_{123}^{(a,b, 0, d)}(t,x)\end{math}.  The only generating functions 
of the form \begin{math}Q_{132}^{(a,b, \emptyset, d)}(t,x)\end{math} where \begin{math}a > 0\end{math} that were computed by 
\cite{KRT1,KRT2,KRT3} were the generating functions of the form 
\begin{math}Q_{132}^{(a, 0, \emptyset, 0)}(t,x)\end{math}  given in Theorem \ref{thm:Qk0e0} above. However 
their techniques can be used to compute \begin{math}Q_{123}^{(a, b, 0, d)}(t,x)\end{math} when \begin{math}a > 0\end{math} for 
arbitrary \begin{math}b\end{math} and \begin{math}d\end{math}. By Lemma \ref{sym2}, 
\begin{math}Q_{123}^{(a,b, 0, d)}(t,x) =  Q_{123}^{(0, d, a, b)}(t,x)\end{math} so 
that such computations will cover all the cases of \begin{math}Q_{123}^{(a,b,c,d)}(t,x)\end{math} where 
exactly one of \begin{math}a\end{math} and \begin{math}c\end{math} equals zero. Thus to complete our analysis of \begin{math}Q_{123}^{(a,b,c,d)}(t,x)\end{math} when \begin{math}a,b,c,d \in \mathbb{N}\end{math}, we need only compute generating functions of the form 
\begin{math}Q_{123}^{(0,b,0,d)}(t,x)\end{math} which we will compute by other methods.

As it was pointed out in \cite{KRT1}, \emph{avoidance} of a marked mesh pattern without quadrants containing the empty set can always be expressed in terms of multi-avoidance of (possibly many) classical patterns. Thus, among our results we will re-derive several known facts in permutation patterns theory as 
well as several new results. However, our main goals are more ambitious aimed at finding  distributions in question.

The outline of this paper is as follows. In Section 2, we shall 
review the bijections of \cite{Kr} and \cite{EliDeu}. 
In Section 3, we shall prove (\ref{reduce1}).   In Section 4, we shall prove that 
\begin{equation}
\Qmn{k,\ell,0,m}{n}\big\vert_{x^0}=\Qmmn{k,\ell,0,m}{n}\big\vert_{x^0}
\end{equation}
and
\begin{equation}
\Qmn{k,\ell,0,m}{n}\big\vert_{x^1}=\Qmmn{k,\ell,0,m}{n}\big\vert_{x^1},
\end{equation}
so that as far as constant terms and the degree \begin{math}1\end{math} terms that occur in the polynomials of the form 
\begin{math}Q_{n,123}^{(k,\ell,0,m)}\end{math}, they  reduce to constant terms and the degree \begin{math}1\end{math} terms that appear in 
polynomials of the form \begin{math}Q_{n,132}^{(k,\ell,0,m)}\end{math} which were analyzed in 
\cite{KRT1,KRT2,KRT3}. we shall also prove some general results 
about the coefficients of the highest power of \begin{math}x\end{math} that occur in the polynomials 
\begin{math}Q_{n,123}^{(a,b,c,d)}(x)\end{math}. In Section 5, we shall show how to compute generating 
functions of the form \begin{math}Q_{123}^{(k,\ell,0,m)}(x,t) = Q_{132}^{(k,\ell,\emptyset,m)}(x,t)\end{math}. In Section 6, 
we will show how to compute generating functions of the form \begin{math}Q_{123}^{(\emptyset,k,\emptyset,\ell)}(x,t)\end{math}. Finally, in Section 7, 
we will show how to compute generating functions of the form \begin{math}Q_{123}^{(0,k,0,0)}(x,t)\end{math} and 
\begin{math}Q_{123}^{(0,k,0,\ell)}(x,t)\end{math}.

\section{Bijections from \(\Sn{n}(132)\) and \(\Sn{n}(123)\) to Dyck paths on an \(n\times n\) Lattice}

Given an \begin{math}n\times n\end{math} square, we will label the coordinates of the columns from left to right with \begin{math}0,1, \ldots, n\end{math} 
and the coordinates of the rows from top to bottom with \begin{math}0,1, \ldots, n\end{math}.  A Dyck path is a path 
made up of unit down-steps \begin{math}D\end{math} and unit right-steps \begin{math}R\end{math} which starts at \begin{math}(0,0)\end{math} and 
ends at \begin{math}(n,n)\end{math} and stays on or below the diagonal \begin{math}x =y\end{math}. The set of Dyck paths on an \(n\times n\) lattice is denoted by 
\begin{math}\mathcal{D}_n\end{math}.
Given a Dyck path \begin{math}P\end{math}, 
we let 
\begin{displaymath}
	\mathrm{Ret}(P) = \{i \geq 1: P \ \mbox{goes through the point} \ (i,i)\}
\end{displaymath}
be the return positions of \begin{math}P\end{math}, and we let \begin{math}\mathrm{ret}(P) = |\mathrm{Ret}(P)|\end{math} be the number of return positions of \begin{math}P\end{math}.  For example, for the Dyck path 
\begin{displaymath}
P =DDRDDRRRDDRDRDRRDR
\end{displaymath}
shown on the right in  \fref{132Dn}, 
\begin{math}\mathrm{Ret}(P) = \{4,8,9\}\end{math} and \begin{math}\mathrm{ret}(P) =3\end{math}.

It is well known that for all \begin{math}n \geq 1\end{math}, 
\begin{math}|\Sn{n}(132)|=|\Sn{n}(123)|=|\mathcal{D}_n|=C_n\end{math}, where \begin{math}C_n = \frac{1}{n+1}\binom{2n}{n}\end{math} is the 
\begin{math}n^{\mathrm{th}}\end{math} Catalan number. \cite{Kr} gave a bijection between \begin{math}\Sn{n}(132)\end{math} and \begin{math}\mathcal{D}_n\end{math}. Later, \cite{EliDeu} gave a bijection between \begin{math}\Sn{n}(123)\end{math} and \begin{math}\mathcal{D}_n\end{math}. The main goal of this section is to
review these two bijections because the recursions that we can derive from 
these bijections will help us develop recursions that allow us to compute 
generating functions of the form \begin{math}Q_{123}^{(a,b,c,d)}(x,t)\end{math}. 

\subsection{The bijection \(\Phi:\Sn{n}(132) \rightarrow \mathcal{D}_n\)}

In this subsection, we describe 
the bijection of \cite{Kr} between \begin{math}\Sn{n}(132)\end{math} and \begin{math}\mathcal{D}_n\end{math}.
Given any permutation \begin{math}\sigma = \sigma_1 \ldots \sigma_n \in\Sn{n}(132)\end{math}, we write it on an \begin{math}n\times n\end{math} table 
by placing \begin{math}\sigma_i\end{math} in the \begin{math}i^{\mathrm{th}}\end{math} column and \begin{math}\sigma_i^{\mathrm{th}}\end{math} row, reading from bottom to top. Then, we 
shade the cells to the north-east of the cell that contains \begin{math}\sigma_i\end{math}. Then the path \begin{math}\Phi(\sigma)\end{math} 
is the path that goes along the south-west boundary of the shaded cells. For example, this 
process is pictured in \fref{132Dn} in the case where  \begin{math}\sigma=867943251\in\Sn{9}(132)\end{math}. 
In this case, \begin{math}\Phi(\sigma)=  DDRDDRRRDDRDRDRRDR\end{math}.

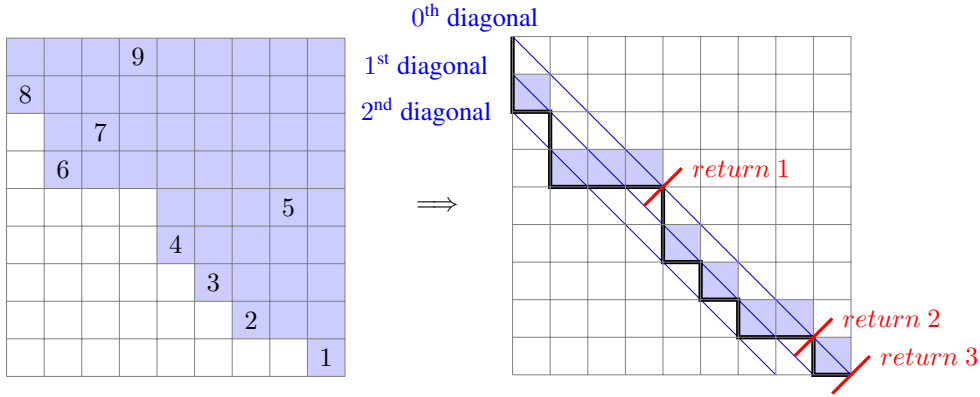
\begin{figure}[ht]
	\centering
	\vspace{-1mm}
	\begin{tikzpicture}[scale =.5]
	\path[fill,blue!20!white] (0,7) -- (1,7)--(1,5)-- (4,5)--(4,3)--(5,3)--(5,2)--(6,2)--(6,1)--(8,1)--(8,0)--(9,0)--(9,9)--(0,9);
	\draw[help lines] (0,0) grid (9,9);
	\filllll{1}{8};\filllll{2}{6};
	\filllll{3}{7};\filllll{4}{9};
	\filllll{5}{4};\filllll{6}{3};
	\filllll{7}{2};\filllll{8}{5};
	\filllll{9}{1};
	\path (0,-.5);
	\end{tikzpicture}	
	\begin{tikzpicture}[scale =.5]
	\fillshade{1/8,2/6,3/6,4/6,5/4,6/3,7/2,8/2,9/1}
	\draw[blue] (0,9)--(9,0);
	\node[blue] at (-1,9.5) {\begin{math}0^{\textnormal{th}}\end{math} diagonal};
	\draw[blue] (0,8)--(8,0);
	\node[blue] at (-2.3,8.2) {\begin{math}1^{\textnormal{st}}\end{math} diagonal};
	\draw[blue] (0,7)--(7,0);
	\node[blue] at (-2.3,7) {\begin{math}2^{\textnormal{nd}}\end{math} diagonal};
	\draw (-2,4.5) node {\begin{math}\Longrightarrow\end{math}};
	\path (-4,4.5);
	\draw[ultra thick] (0,9)--(0,7) -- (1,7)--(1,5)-- (4,5)--(4,3)--(5,3)--(5,2)--(6,2)--(6,1)--(8,1)--(8,0)--(9,0);
	\draw[help lines] (0,0) grid (9,9);	
	\draw[red, very thick] (3.5,4.5)--(4.5,5.5) node[right] {\begin{math}return\ 1\end{math}};
	\draw[red, very thick] (7.5,.5)--(8.5,1.5) node[right] {\begin{math}return\ 2\end{math}};
	\draw[red, very thick] (8.5,-.5)--(9.5,.5) node[right] {\begin{math}return\ 3\end{math}};
	\end{tikzpicture}
	\caption{\(\Sn{n}(132)\) to \(\mathcal{D}_n\)}
	\label{fig:132Dn}
\end{figure}

Given \begin{math}\sigma = \sigma_1 \ldots \sigma_n\end{math}, we say that \begin{math}\sigma_j\end{math} is a left-to-right mininum of \begin{math}\sigma\end{math} if 
\begin{math}\sigma_i > \sigma_j\end{math} for all \begin{math}i < j\end{math}. It is easy to see that the left-to-right minima 
of \begin{math}\sigma\end{math} correspond to peaks of the path \begin{math}\Phi(\sigma)\end{math}, i.e.\ they occupy cells along the inside 
boundary of the \begin{math}\Phi(\sigma)\end{math} that correspond to a down-step \begin{math}D\end{math} immediately followed by a 
right-step \begin{math}R\end{math}.  We call such cells, the outer corners of the path. 
Thus we shall often refer to the left-to-right minima 
of the \begin{math}\sigma\end{math} as the set of peaks of \begin{math}\sigma\end{math}, and \begin{math}\sigma_i\end{math}'s which are not left-to-right 
minima as the non-peaks of \begin{math}\sigma\end{math}. For example, for the permutation \begin{math}\sigma\end{math} pictured in 
\fref{132Dn}, there are \begin{math}6\end{math} peaks,  \begin{math}\{8,6,4,3,2,1\}\end{math}, and \begin{math}3\end{math} non-peaks, \begin{math}\{7,9,5\}\end{math}. 
The horizontal segments of the path \begin{math}\Phi(\sigma)\end{math} are the maximal consecutive 
sequences of \begin{math}R\end{math}'s in \begin{math}\Phi(\sigma)\end{math}. For example, in \fref{132Dn}, the lengths of 
the horizontal segments, reading from top to bottom, are \begin{math}1,3,1,1,2,1\end{math}.  We will be interested in the set of numbers 
that lie to the north of each horizontal segments in \begin{math}\Phi(\sigma)\end{math}.  For instance, 
in our example,  \begin{math}\{8\}\end{math} is the set associated with the first horizontal segment of  
\begin{math}\Phi(\sigma)\end{math}, \begin{math}\{6,7,9\}\end{math} is the set of numbers associated with the second horizontal segment of  
\begin{math}\Phi(\sigma)\end{math}, etc..  Because \begin{math}\sigma\end{math} is a 132-avoiding permutation, it follows 
that set of numbers above a horizontal segment must occur in increasing order. 
That is, since the cell immediately above the first right-step of the horizontal segment 
must be occupied with the least element in the set associated to the horizontal segment, 
then the remaining numbers must appear in increasing order if we are to avoid 132.

We shall also label the diagonals that go through corners of squares 
that are parallel to and below the main diagonal with \begin{math}0, 1, 2, \ldots \end{math} starting at the main 
diagonal.  In this way, each peak of the permutation corresponds to a diagonal. 
In the example in \fref{132Dn}, we have \begin{math}1\end{math} peak on the \begin{math}0^{\textnormal{th}}\end{math} diagonal, \begin{math}4\end{math} peaks on the \begin{math}1^{\textnormal{st}}\end{math} diagonal and \begin{math}1\end{math} peak on the \begin{math}2^{\textnormal{nd}}\end{math} diagonal. 

The map \begin{math}\Phi^{-1}\end{math} is easy to describe.  That is, given a Dyck path \begin{math}P\end{math}, we 
first mark every cell corresponding to a peak of the path with a ``{\color{green!50!black}\begin{math}\times\end{math}}". Then we look at the rows and columns which do not have a cross. Starting from the left-most column, 
that does not contain a cross, we put a cross in the lowest possible row without a cross that lies 
above the path. In this ways we will construct a permutation \begin{math}\sigma = \Phi^{-1}(P)\end{math}.  
This process is pictured in \fref{Dn132}. 

\begin{figure}[ht]
	\centering
	\vspace{-1mm}
	\begin{tikzpicture}[scale =.5]
	\path[fill,blue!20!white] (0,7) -- (1,7)--(1,5)-- (4,5)--(4,3)--(5,3)--(5,2)--(6,2)--(6,1)--(8,1)--(8,0)--(9,0)--(9,9)--(0,9);
	\draw[help lines] (0,0) grid (9,9);
	\draw[very thick,blue] (0,8.5)--(3.5,8.5)--(3.5,5);
	\draw[very thick,blue] (1,6.5)--(2.5,6.5)--(2.5,5);
	\draw[very thick,blue] (4,4.5)--(7.5,4.5)--(7.5,1);
	\fillgcross{1}{8};\fillgcross{2}{6};
	\fillcross{3}{7};\fillcross{4}{9};
	\fillgcross{5}{4};\fillgcross{6}{3};
	\fillgcross{7}{2};\fillcross{8}{5};
	\fillgcross{9}{1};	
	\end{tikzpicture}	
	\begin{tikzpicture}[scale =.5]
	\fillshade{1/8,2/6,3/6,4/6,5/4,6/3,7/2,9/1,8/2}
	\draw[help lines] (0,9)--(9,0);
	\draw (-2,4.5) node {\begin{math}\Longrightarrow\end{math}};
	\path (-4,4.5);
	\draw[ultra thick] (0,9)--(0,7) -- (1,7)--(1,5)-- (4,5)--(4,3)--(5,3)--(5,2)--(6,2)--(6,1)--(8,1)--(8,0)--(9,0);
	\draw[help lines] (0,0) grid (9,9);
	\filllll{1}{8};\filllll{2}{6};
	\filllll{3}{7};\filllll{4}{9};
	\filllll{5}{4};\filllll{6}{3};
	\filllll{7}{2};\filllll{8}{5};
	\filllll{9}{1};
	\end{tikzpicture}
	
	\caption{\(\mathcal{D}_n\) to \(\Sn{n}(132)\)}
	\label{fig:Dn132}
\end{figure}
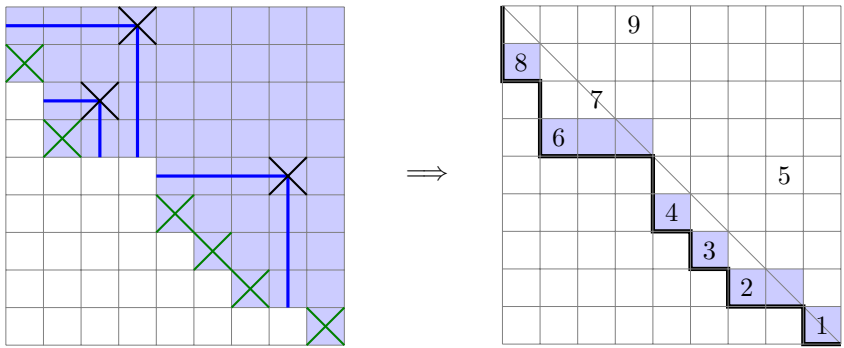

Details that  \begin{math}\Phi:\Sn{n}(132) \rightarrow D_n\end{math} is a bijection can be found in \cite{Kr}. 
However, given that \begin{math}\Phi\end{math} is a bijection, the following properties are easy to prove.
\begin{lemma}\label{p1}
	Given any Dyck path \begin{math}P\end{math}, let \begin{math}\sigma=\sigma_{132}(P)=\Phi^{-1}(P)\end{math}. Then the following hold.
	\begin{enumerate}[(1)]
		\item For each horizontal segment \begin{math}H\end{math} of \begin{math}P\end{math}, the set of numbers associated 
		to \begin{math}H\end{math} form a consecutive increasing sequence in \begin{math}\sigma\end{math} and the least 
		number of the sequence sits immediately above the first right-step of \begin{math}H\end{math}. Hence 
		the only decreases in \begin{math}\sigma\end{math} occur between two different horizontal segments of \begin{math}P\end{math}. 
		
		\item The number \begin{math}n\end{math} is in the column of last right-step before the first return.

		\item Suppose that \begin{math}\sigma_i\end{math} is a peak of \begin{math}\sigma\end{math} and the cell containing \begin{math}\sigma_i\end{math} is on the 
		\begin{math}k^{\mathrm{th}}\end{math}-diagonal. Then there are \begin{math}k\end{math} elements in the graph 
		\begin{math}G(\sigma)\end{math} in the first quadrant relative 
		to coordinate system centered at \begin{math}(i,\sigma_i)\end{math}. 
	\end{enumerate}
\end{lemma}
\begin{proof}
	
	(1) easily follows from our description of the bijections \begin{math}\Phi\end{math} and \begin{math}\Phi^{-1}\end{math}. 
	
	For (2), we consider two cases.  First if \begin{math}1 \in Return(P)\end{math}, then \begin{math}P\end{math} must start 
	out \begin{math}DR \ldots \end{math} so that the first outer corner of \begin{math}P\end{math} is in row \begin{math}n\end{math} reading 
	from bottom to top, which must be occupied by \begin{math}n\end{math} so that \begin{math}n\end{math} is 
	in the column of the last right-step before the first return. If 
	\begin{math}i > 1\end{math} is the least element of \begin{math}Return(P)\end{math}, then there are \begin{math}i\end{math} right-steps in 
	the first \begin{math}2i\end{math} steps of \begin{math}P\end{math}.  The outer corners in the first \begin{math}2i\end{math} steps 
	of \begin{math}P\end{math} must all be occupied by numbers greater than \begin{math}n-i\end{math}.  Thus we can only place the 
	numbers \begin{math}n, \ldots ,n-i+1\end{math} in the columns above the horizontal segments that occur 
	in the first \begin{math}2i\end{math} steps of \begin{math}P\end{math}. After we place numbers in the outer corners of 
	the first \begin{math}2i\end{math} steps, we always place \begin{math}\times\end{math}'s  in the lowest row that is above 
	the path starting from the left-most column. This means that we will place 
	\begin{math}\times\end{math}'s in the rows \begin{math}n-1, \ldots, n-i+1\end{math}, 
	before we place a \begin{math}\times\end{math} in row \begin{math}n\end{math}, reading from bottom to top. It follows 
	that the position of the \begin{math}\times\end{math} in row \begin{math}n\end{math} is in column \begin{math}i\end{math}.

	For (3), suppose that \begin{math}\sigma_i\end{math} is a peak of \begin{math}\sigma\end{math} and \begin{math}\sigma_i\end{math} is in the \begin{math}k^{\mathrm{th}}\end{math}-diagonal.
	This means that the right-step that sits directly below \begin{math}\sigma_i\end{math} in \begin{math}P\end{math} is the  
	\begin{math}i^{\mathrm{th}}\end{math} right-step in \begin{math}P\end{math} and is preceded by \begin{math}i+k\end{math} down-steps. 
	Hence there are \begin{math}i+k-1\end{math} rows above \begin{math}\sigma_i\end{math} in the graphs of \begin{math}\sigma\end{math}. There are \begin{math}i-1\end{math} elements 
	that are associated with the horizontal segments to the left of \begin{math}\sigma_i\end{math} which means by the time that we get to \begin{math}\sigma_i\end{math} in the construction of \begin{math}\sigma_{132}(P)\end{math} from \begin{math}P\end{math}, there are \begin{math}i-1\end{math} elements to 
	the left of \begin{math}\sigma_i\end{math} in \begin{math}\sigma\end{math} which are larger than \begin{math}\sigma_i\end{math}. Hence there must be exactly \begin{math}k\end{math} elements 
	to the right of \begin{math}\sigma_i\end{math} in \begin{math}\sigma\end{math} which are larger than \begin{math}\sigma_i\end{math}. 
\end{proof}

\subsection{The bijection \(\Psi:\Sn{n}(123) \rightarrow \mathcal{D}_n\)}

In this section, we will describe the bijection \begin{math}\Psi:\Sn{n}(123) \rightarrow \mathcal{D}_n\end{math}  given 
by \cite{EliDeu}.  
Given any permutation \begin{math}\sigma \in\Sn{n}(123)\end{math}, \begin{math}\Psi(\sigma)\end{math} is constructed exactly as in 
the previous section. \fref{123Dn} shows an example of this map, from \begin{math}\sigma=869743251\in\Sn{9}(123)\end{math} to the Dyck path \emph{DDRDDRRRDDRDRDRRDR}.

\begin{figure}[ht]
	\centering
	\vspace{-1mm}
	\begin{tikzpicture}[scale =.5]
	\path[fill,blue!20!white] (0,7) -- (1,7)--(1,5)-- (4,5)--(4,3)--(5,3)--(5,2)--(6,2)--(6,1)--(8,1)--(8,0)--(9,0)--(9,9)--(0,9);
	\draw[help lines] (0,0) grid (9,9);
	\filllll{1}{8};\filllll{2}{6};
	\filllll{3}{9};\filllll{4}{7};
	\filllll{5}{4};\filllll{6}{3};
	\filllll{7}{2};\filllll{8}{5};
	\filllll{9}{1};
	\end{tikzpicture}	
	\begin{tikzpicture}[scale =.5]
	\fillshade{1/8,2/6,3/6,4/6,5/4,6/3,7/2,9/1,8/2}
	\draw[help lines] (0,9)--(9,0);
	\draw (-2,4.5) node {\begin{math}\Longrightarrow\end{math}};
	\path (-4,4.5);
	\draw[ultra thick] (0,9)--(0,7) -- (1,7)--(1,5)-- (4,5)--(4,3)--(5,3)--(5,2)--(6,2)--(6,1)--(8,1)--(8,0)--(9,0);
	\draw[help lines] (0,0) grid (9,9);
	\end{tikzpicture}
	
	\caption{\(\Sn{n}(123)\) to \(\mathcal{D}_n\)}
	\label{fig:123Dn}
\end{figure}
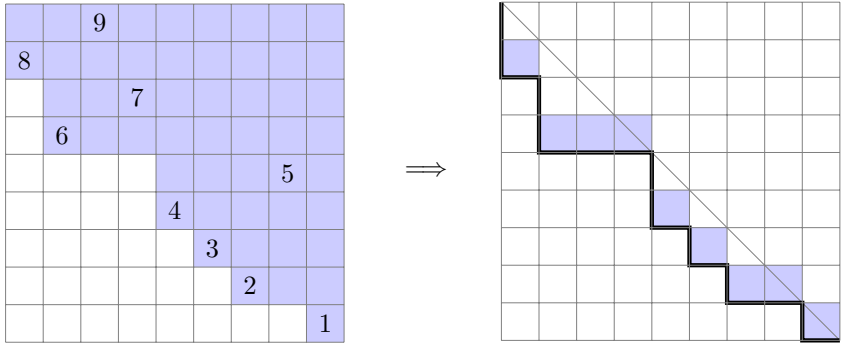

Given any Dyck path \begin{math}P\end{math}, we construct \begin{math}\Psi^{-1}(P) = \sigma_{123}(P)\end{math} as follows. 
First we place an ``{\color{green!50!black}\begin{math}\times\end{math}}'' in every outer corner of \begin{math}P\end{math}. Then we consider 
the rows and columns which do not have a {\color{green!50!black}\begin{math}\times\end{math}}. Processing the rows from top to bottom 
and the columns from left to right, we place an \begin{math}\times\end{math} in the 
\begin{math}i^{\mathrm{th}}\end{math} empty row and \begin{math}i^{\mathrm{th}}\end{math} empty column. This process is 
pictured in \fref{Dn123}. The details that \begin{math}\Psi\end{math} is a bijection 
between \begin{math}\Sn{n}(123)\end{math} and \begin{math}\mathcal{D}_n\end{math} can be found in \cite{EliDeu}.

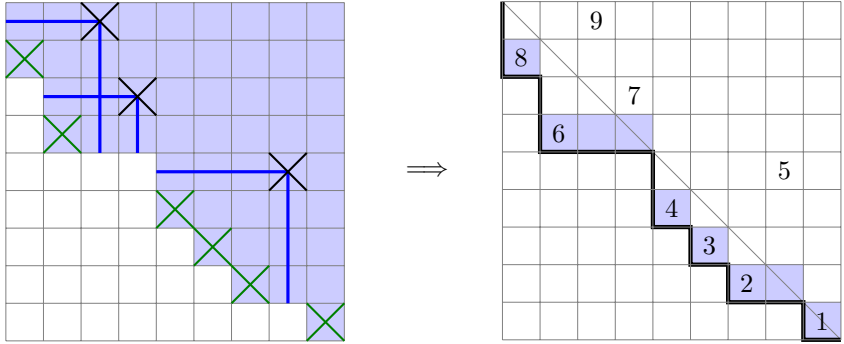
\begin{figure}[ht]
	\centering
	\vspace{-1mm}
	\begin{tikzpicture}[scale =.5]
	\path[fill,blue!20!white] (0,7) -- (1,7)--(1,5)-- (4,5)--(4,3)--(5,3)--(5,2)--(6,2)--(6,1)--(8,1)--(8,0)--(9,0)--(9,9)--(0,9);
	\draw[help lines] (0,0) grid (9,9);
	\draw[very thick,blue] (0,8.5)--(2.5,8.5)--(2.5,5);
	\draw[very thick,blue] (1,6.5)--(3.5,6.5)--(3.5,5);
	\draw[very thick,blue] (4,4.5)--(7.5,4.5)--(7.5,1);
	\fillgcross{1}{8};\fillgcross{2}{6};
	\fillcross{3}{9};\fillcross{4}{7};
	\fillgcross{5}{4};\fillgcross{6}{3};
	\fillgcross{7}{2};\fillcross{8}{5};
	\fillgcross{9}{1};	
	\end{tikzpicture}	
	\begin{tikzpicture}[scale =.5]
	\fillshade{1/8,2/6,3/6,4/6,5/4,6/3,7/2,9/1,8/2}
	\draw[help lines] (0,9)--(9,0);
	\draw (-2,4.5) node {\begin{math}\Longrightarrow\end{math}};
	\path (-4,4.5);
	\draw[ultra thick] (0,9)--(0,7) -- (1,7)--(1,5)-- (4,5)--(4,3)--(5,3)--(5,2)--(6,2)--(6,1)--(8,1)--(8,0)--(9,0);
	\draw[help lines] (0,0) grid (9,9);
	\filllll{1}{8};\filllll{2}{6};
	\filllll{3}{9};\filllll{4}{7};
	\filllll{5}{4};\filllll{6}{3};
	\filllll{7}{2};\filllll{8}{5};
	\filllll{9}{1};
	\end{tikzpicture}
	
	\caption{\(\mathcal{D}_n\) to \(\Sn{n}(123)\)}
	\label{fig:Dn123}
\end{figure}

We then have the following lemma about the properties of this map.
\begin{lemma}\label{p2} Let \begin{math}P \in \mathcal{D}_n\end{math} and \begin{math}\sigma = \sigma_{123}(P) = \Psi^{-1}(P)\end{math}. 
	Then the following hold. 
	\begin{enumerate}[(1)]
		\item For each horizontal segment \begin{math}H\end{math} of \begin{math}P\end{math}, the least element of the 
		set of numbers associated to \begin{math}H\end{math} sits directly above the first right-step of \begin{math}H\end{math} and 
		the remaining numbers of the set form a consecutive decreasing sequence in \begin{math}\sigma\end{math}. 
		
		\item \begin{math}\sigma\end{math} can be decomposed into two decreasing subsequences, the first decreasing 
		subsequence corresponds to the peaks of \begin{math}\sigma\end{math} and the second decreasing subsequence 
		corresponds to the non-peaks of \begin{math}\sigma\end{math}. 
		
		\item Suppose that \begin{math}\sigma_i\end{math} is a peak of \begin{math}\sigma\end{math} and the cell containing \begin{math}\sigma_i\end{math} is on the 
		\begin{math}k^{\mathrm{th}}\end{math}-diagonal. Then there are \begin{math}k\end{math} elements in the graph 
		\begin{math}G(\sigma)\end{math} in the first quadrant relative 
		to coordinate system centered at \begin{math}(i,\sigma_i)\end{math}.
	\end{enumerate}
\end{lemma}
\begin{proof}
	It is easy to see that parts (1) and (2) follow from the construction of \begin{math}\Psi^{-1}\end{math}. The proof 
	of part (3) is the same as the proof of part (3) of Lemma \ref{p1}. 
\end{proof}

\section{General results about \(\Qm{a,b,c,d}\) and \(\Qmm{a,b,c,d}\)}

In this section, we shall prove several general results about the generating functions
\begin{math}\Qm{a,b,c,d}\end{math} and \begin{math}\Qmm{a,b,c,d}\end{math}. 

First suppose that \begin{math}k >0\end{math}.  Then since in a 123-avoiding permutation 
\begin{math}\sigma = \sigma_1 \ldots \sigma_n \in \Sn{n}(123)\end{math}, no \begin{math}\sigma_i\end{math} can have elements 
in the first and third quadrants in \begin{math}G(\sigma)\end{math} relative to the coordinate system 
centered at \begin{math}(i,\sigma_i)\end{math}, it follows that \begin{math}\sigma_i\end{math} matches 
\begin{math}\MMP(k,\ell,0,m)\end{math} in \begin{math}\sigma\end{math} if and only if it matches \begin{math}\MMP(k,\ell,\emptyset,m)\end{math} in \begin{math}\sigma\end{math}. 
Thus 
\begin{equation}
\Qm{k,\ell,0,m}=\Qm{k,\ell,\emptyset,m} \ \mbox{for all \begin{math}k > 0\end{math} and \begin{math}\ell,m \geq 0\end{math}.}
\end{equation}
Similarly, one can prove that 
\begin{equation}
\Qm{0,\ell,k,m}=\Qm{\emptyset,\ell,k,m} \ \mbox{for all \begin{math}k > 0\end{math} and \begin{math}\ell,m \geq 0\end{math}.}
\end{equation}

Next suppose that \begin{math}P\end{math} is a Dyck path in \begin{math}\mathcal{D}_n\end{math} and consider the 
differences between \begin{math}\sigma = \Phi^{-1}(P)\end{math} and \begin{math}\tau =\Psi^{-1}(P)\end{math}. Clearly, the elements 
corresponding to the outer corners of \begin{math}P\end{math} are the same in both \begin{math}\sigma\end{math} and \begin{math}\tau\end{math}, thus \begin{math}\sigma\end{math} and \begin{math}\tau\end{math} have the same peaks. The only 
difference is how to order the non-peaks. Note that, by construction, 
the non-peaks in \begin{math}\sigma\end{math} and \begin{math}\tau\end{math} cannot match a quadrant marked mesh 
pattern of the form \begin{math}\MMP(k,\ell,\emptyset,m)\end{math}. That is, 
a non-peak \begin{math}\sigma_i\end{math} of \begin{math}\sigma\end{math} must have at least one element occurring in 
the third quadrant of \begin{math}G(\sigma)\end{math} relative to the coordinate system centered 
at \begin{math}(i,\sigma_i)\end{math}, namely, the least element of the set associated 
with the horizontal segment \begin{math}H\end{math} whose associated set contains \begin{math}\sigma_i\end{math}. A similar 
statement holds for \begin{math}\tau\end{math}.  Now suppose that the number \begin{math}\sigma_j\end{math} is 
a peak of \begin{math}\sigma\end{math}.  Thus \begin{math}\sigma_j\end{math} sits directly above the first right-step of 
some horizontal segment \begin{math}H\end{math} of \begin{math}P\end{math} in the graph of \begin{math}\sigma\end{math}. By Lemma \ref{p1}, if the cell 
containing \begin{math}\sigma_i\end{math} is in the \begin{math}r^{\mathrm{th}}\end{math}-diagonal, then in \begin{math}G(\sigma)\end{math}, 
there are exactly \begin{math}r\end{math}-elements in the first quadrant relative to the coordinate 
system centered at \begin{math}(j,\sigma_j)\end{math}. It is easy to see that 
the number of elements in the second quadrant in \begin{math}G(\sigma)\end{math} 
relative to coordinate system centered \begin{math}(j,\sigma_j)\end{math} is \begin{math}s=j-1\end{math} where 
\begin{math}s\end{math} is the sum of lengths of the horizontal 
segments to the left of \begin{math}H\end{math} and, hence, the number of elements in the fourth quadrant in \begin{math}G(\sigma)\end{math} 
relative to coordinate system centered \begin{math}(j,\sigma_j)\end{math} is equal to \begin{math}n-k-s-1=n-k-j\end{math}. However, by 
Lemma \ref{p2}, the exact 
same statement holds for \begin{math}\sigma_j\end{math} in the graph \begin{math}G(\tau)\end{math} relative to the coordinate 
system center at \begin{math}(j,\sigma_j)\end{math}. It follows that for any \begin{math}k,\ell,m \geq 0\end{math}, 
\begin{math}\sigma_j\end{math} matches \begin{math}\MMP(k,\ell,\emptyset,m)\end{math} in \begin{math}\sigma\end{math} if and only if 
\begin{math}\sigma_j\end{math} matches \begin{math}\MMP(k,\ell,\emptyset,m)\end{math} in \begin{math}\tau\end{math}. 
For example, \fref{Equi1} illustrates this correspondence. It follows 
that the map \begin{math}\Psi \circ \Phi^{-1}:\Sn{n}(132) \rightarrow \Sn{n}(123)\end{math} shows 
that for all \begin{math}k > 0\end{math} and \begin{math}\ell,m \geq 0\end{math}, 
\begin{equation}
Q_{n,132}^{(k,\ell, \emptyset,m)}(x) = Q_{n,123}^{(k,\ell, \emptyset,m)}(x).
\end{equation}

\begin{figure}[ht]
	\centering
	\vspace{-1mm}
	\begin{tikzpicture}[scale =.5]
	\path[fill,blue!20!white] (0,7) -- (1,7)--(1,5)-- (4,5)--(4,3)--(5,3)--(5,2)--(6,2)--(6,1)--(8,1)--(8,0)--(9,0)--(9,9)--(0,9);
	\fill[red!50!white] (1,0) rectangle (2,9);
	\fill[red!50!white] (0,5) rectangle (9,6);
	\draw[help lines] (0,0) grid (9,9);
	\draw[ultra thick] (0,9)--(0,7) -- (1,7)--(1,5)-- (4,5)--(4,3)--(5,3)--(5,2)--(6,2)--(6,1)--(8,1)--(8,0)--(9,0);
	\filllll{1}{8};\filllll{2}{6};
	\filllll{3}{7};\filllll{4}{9};
	\filllll{5}{4};\filllll{6}{3};
	\filllll{7}{2};\filllll{8}{5};
	\filllll{9}{1};
	\end{tikzpicture}
	\begin{tikzpicture}[scale =.5]
	\path[fill,blue!20!white] (0,7) -- (1,7)--(1,5)-- (4,5)--(4,3)--(5,3)--(5,2)--(6,2)--(6,1)--(8,1)--(8,0)--(9,0)--(9,9)--(0,9);
	\fill[red!50!white] (1,0) rectangle (2,9);
	\fill[red!50!white] (0,5) rectangle (9,6);
	\draw[help lines] (0,9)--(9,0);
	\draw (-2,4.5) node {\begin{math}\Longrightarrow\end{math}};
	\path (-4,4.5);
	\draw[ultra thick] (0,9)--(0,7) -- (1,7)--(1,5)-- (4,5)--(4,3)--(5,3)--(5,2)--(6,2)--(6,1)--(8,1)--(8,0)--(9,0);
	\draw[help lines] (0,0) grid (9,9);
	\filllll{1}{8};\filllll{2}{6};
	\filllll{3}{9};\filllll{4}{7};
	\filllll{5}{4};\filllll{6}{3};
	\filllll{7}{2};\filllll{8}{5};
	\filllll{9}{1};
	\end{tikzpicture}
	
	\caption{\(\Sn{n}(132)\) to \(\Sn{n}(123)\) preserves \(\MMP(k,\ell,\emptyset,m)\)}
	\label{fig:Equi1}
\end{figure}
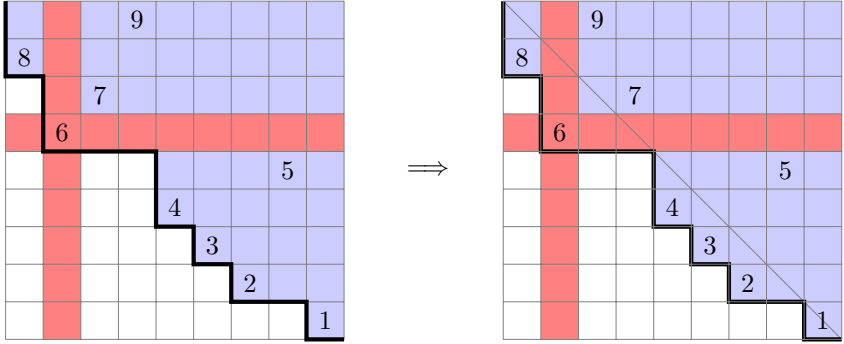

Combining the remarks above with Lemma \ref{sym2}, we have the following theorem.
\begin{theorem}\label{corollary:1} For any \begin{math}k > 0\end{math} and \begin{math}\ell, m \geq 0\end{math}, 
	\begin{eqnarray} 
	\Qm{k,\ell,0,m}&=&\Qm{k,\ell,\emptyset,m}=\Qmm{k,\ell,\emptyset,m} \\ \nonumber 
	&=& \Qm{0,m,k,\ell} = \Qm{\emptyset,m,k,\ell}.
	\end{eqnarray}
\end{theorem}

It follows that the only generating 
functions of the form \begin{math}\Qm{a,b,c,d}\end{math} that cannot be reduced to generating functions 
of the form \begin{math}\Qmm{a,b,c,d}\end{math} are generating functions of the form 
\begin{math}\Qm{0,b,0,d}\end{math}. In the series of papers of \cite{KRT1,KRT2,KRT3}, the only generating 
functions of the form \begin{math}\Qmm{k,\ell,\emptyset,m}\end{math} that were computed were the generating 
functions of the form \begin{math}\Qmm{k,0,\emptyset,0}\end{math} given in Theorem \ref{thm:Qk0e0}. 
Our main interest in this paper is to compute generating functions of 
the form \begin{math}\Qm{a,b,c,d}\end{math} for \begin{math}a,b,c,d \in \mathbb{N}\end{math}. 
Thus we will show how to compute generating functions of the form 
\begin{math}\Qmm{k,\ell,\emptyset,m}\end{math} for \begin{math}k,\ell, m \in \mathbb{N}\end{math} and of the form 
\begin{math}\Qmm{0,k,0,\ell}\end{math} for \begin{math}k,\ell \in \mathbb{N}\end{math}. 

\section{The coefficients of \(x^0\), \(x^1\) and the highest power of \(x\) in polynomials \(\Qmn{a,b,c,d}{n}\)}

Before we compute the generating functions, we shall 
prove some general results about the constant terms and the coefficients 
of the highest power of \begin{math}x\end{math} in the polynomials 
\begin{math}\Qmn{a,b,c,d}{n}\end{math} in this section.

\subsection{The coefficients of \(x^0\) and \(x^1\) in polynomials \(\Qmn{k,\ell,\emptyset,m}{n}\)}

Since the coefficients of \begin{math}x^k\end{math}  in polynomials of the form 
\begin{math}\Qmn{k,\ell,0,m}{n}\end{math} and \begin{math}\Qmn{0,m,k,\ell}{n}\end{math} can be found from the 
coefficients of \begin{math}x^k\end{math} 
in polynomials of the form \begin{math}\Qmmn{k,\ell,\emptyset,m}{n}\end{math}, we start out 
with an observation about the coefficients of \begin{math}x^0\end{math} and \begin{math}x^1\end{math} in polynomials of the 
form \begin{math}\Qmmn{k,\ell,\emptyset,m}{n}\end{math}.
\begin{theorem}\label{theorem:3}
	\begin{equation}\label{eq3.1}
	\Qmmn{k,\ell,\emptyset,m}{n}\big\vert_{x^0}=\Qmmn{k,\ell,0,m}{n}\big\vert_{x^0}
	\end{equation}
	and
	\begin{equation}\label{eq3.2}
	\Qmmn{k,\ell,\emptyset,m}{n}\big\vert_{x^1}=\Qmmn{k,\ell,0,m}{n}\big\vert_{x^1}.
	\end{equation}
\end{theorem}

\begin{proof}
	For (\ref{eq3.1}), note that any permutation in \begin{math}\Sn{n}(132)\end{math} avoiding the pattern \begin{math}\MMP(k,\ell,0,m)\end{math} must also avoid the pattern \begin{math}\MMP(k,\ell,\emptyset,m)\end{math}. Thus to prove (\ref{eq3.1}), we need to show 
	that any permutation in \begin{math}\Sn{n}(132)\end{math} avoiding the pattern \begin{math}\MMP(k,\ell,\emptyset,m)\end{math} must also avoid the pattern \begin{math}\MMP(k,\ell,0,m)\end{math}. We know that only the peaks of \begin{math}\sigma\end{math} can 
	match patterns of the from \begin{math}\MMP(k,\ell,\emptyset,m)\end{math}. Thus we must show that if 
	the peaks of \begin{math}\sigma\end{math} do not match \begin{math}\MMP(k,\ell,0,m)\end{math}, then the non-peaks of \begin{math}\sigma\end{math} do not match \begin{math}\MMP(k,\ell,0,m)\end{math} either.
	
	To show this, we appeal to part (a) of Lemma \ref{p1}. That is, we know that 
	on each horizontal segment \begin{math}H\end{math} of \begin{math}\Phi(\sigma)\end{math}, the elements in the columns 
	above \begin{math}H\end{math} form a consecutively increasing sequence in \begin{math}\sigma\end{math}. But it is 
	easy to see that if \begin{math}\sigma_i < \sigma_{i+1}\end{math}, then in the graph of \begin{math}G(\sigma)\end{math}, 
	the number of elements in quadrant \begin{math}A\end{math} relative to the coordinate system centered 
	at \begin{math}(i,\sigma_i)\end{math} is greater than or equal to the number of 
	elements in quadrant \begin{math}A\end{math} relative to the coordinate system centered 
	at \begin{math}(i+1,\sigma_{i+1})\end{math} for \begin{math}A \in \{I,II,IV\}\end{math}.  Thus if the peak corresponding 
	to the horizontal segment \begin{math}H\end{math} does not match \begin{math}\MMP(k,\ell,0,m)\end{math}, then no other 
	element associated with \begin{math}H\end{math} can match \begin{math}\MMP(k,\ell,0,m)\end{math}.
	For example, \fref{Equi2} illustrates this observation for the horizontal segment 
	corresponding to the set \begin{math}\{6,7,9\}\end{math}.
	Thus we have proved that if the peaks of \begin{math}\sigma\end{math} do not match \begin{math}\MMP(k,\ell,0,m)\end{math}, then the non-peaks of \begin{math}\sigma\end{math} do not match \begin{math}\MMP(k,\ell,0,m)\end{math} either.
	
	\begin{figure}[ht]
		\centering
		\vspace{-1mm}
		\begin{tikzpicture}[scale =.45]
		\path[fill,blue!20!white] (0,7) -- (1,7)--(1,5)-- (4,5)--(4,3)--(5,3)--(5,2)--(6,2)--(6,1)--(8,1)--(8,0)--(9,0)--(9,9)--(0,9);
		\fill[red!50!white] (1,0) rectangle (2,9);
		\fill[red!50!white] (0,5) rectangle (9,6);
		\draw[help lines] (0,0) grid (9,9);
		\draw[ultra thick] (0,9)--(0,7) -- (1,7)--(1,5)-- (4,5)--(4,3)--(5,3)--(5,2)--(6,2)--(6,1)--(8,1)--(8,0)--(9,0);
		\filllll{1}{8};\filllll{2}{6};
		\filllll{3}{7};\filllll{4}{9};
		\filllll{5}{4};\filllll{6}{3};
		\filllll{7}{2};\filllll{8}{5};
		\filllll{9}{1};
		\path (10,0);
		\end{tikzpicture}
		\begin{tikzpicture}[scale =.45]
		\path[fill,blue!20!white] (0,7) -- (1,7)--(1,5)-- (4,5)--(4,3)--(5,3)--(5,2)--(6,2)--(6,1)--(8,1)--(8,0)--(9,0)--(9,9)--(0,9);
		\fill[red!50!white] (2,0) rectangle (3,9);
		\fill[red!50!white] (0,6) rectangle (9,7);
		\draw[help lines] (0,0) grid (9,9);
		\draw[ultra thick] (0,9)--(0,7) -- (1,7)--(1,5)-- (4,5)--(4,3)--(5,3)--(5,2)--(6,2)--(6,1)--(8,1)--(8,0)--(9,0);
		\filllll{1}{8};\filllll{2}{6};
		\filllll{3}{7};\filllll{4}{9};
		\filllll{5}{4};\filllll{6}{3};
		\filllll{7}{2};\filllll{8}{5};
		\filllll{9}{1};
		\path (10,0);
		\end{tikzpicture}
		\begin{tikzpicture}[scale =.45]
		\path[fill,blue!20!white] (0,7) -- (1,7)--(1,5)-- (4,5)--(4,3)--(5,3)--(5,2)--(6,2)--(6,1)--(8,1)--(8,0)--(9,0)--(9,9)--(0,9);
		\fill[red!50!white] (3,0) rectangle (4,9);
		\fill[red!50!white] (0,8) rectangle (9,9);
		\draw[help lines] (0,0) grid (9,9);
		\draw[ultra thick] (0,9)--(0,7) -- (1,7)--(1,5)-- (4,5)--(4,3)--(5,3)--(5,2)--(6,2)--(6,1)--(8,1)--(8,0)--(9,0);
		\filllll{1}{8};\filllll{2}{6};
		\filllll{3}{7};\filllll{4}{9};
		\filllll{5}{4};\filllll{6}{3};
		\filllll{7}{2};\filllll{8}{5};
		\filllll{9}{1};
		\end{tikzpicture}
		\caption{\(\MMP(k,\ell,0,m)\)-mch for the peak and non-peaks on a horizontal segment}
		\label{fig:Equi2}
	\end{figure}
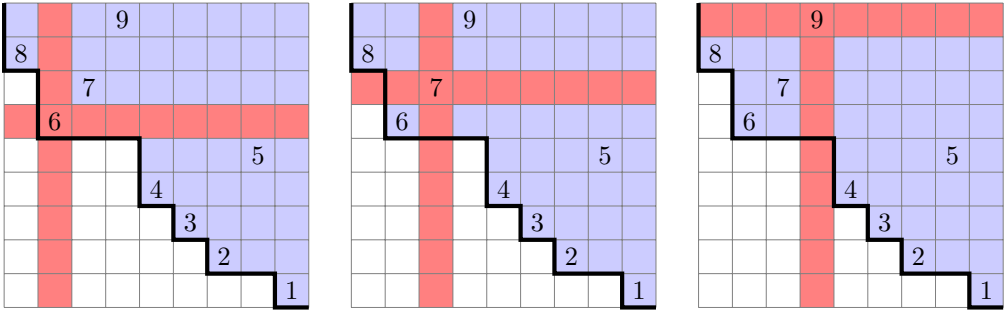
	
	To prove (\ref{eq3.2}), 
	suppose that \begin{math}\sigma = \sigma_1 \ldots \sigma_n \in \Sn{n}(132)\end{math} is such that there is exactly one \begin{math}\sigma_i\end{math} 
	that  match 
	\begin{math}\MMP(k,\ell,0,m)\end{math}. We claim that \begin{math}\sigma_i\end{math} must be 
	a peak. That is, by our argument above, if \begin{math}\sigma_i\end{math} sits above a horizontal segment \begin{math}H\end{math} of \begin{math}\Phi(\sigma)\end{math}, then the peak corresponding 
	to \begin{math}H\end{math} must match \begin{math}\MMP(k,\ell,0,m)\end{math} as long as \begin{math}\sigma_i\end{math} matches \begin{math}\MMP(k,\ell,0,m)\end{math}. Thus if \begin{math}\sigma_i\end{math} is the only element 
	of \begin{math}\sigma\end{math} that matches \begin{math}\MMP(k,\ell,0,m)\end{math}, then it must be a peak and hence it also matches 
	\begin{math}\MMP(k,\ell,\emptyset,m)\end{math}. Clearly, there cannot be two elements of \begin{math}\sigma\end{math} that 
	match \begin{math}\MMP(k,\ell,\emptyset,m)\end{math} as
	this would imply that these two elements 
	of \begin{math}\sigma\end{math} would match \begin{math}\MMP(k,\ell,0,m)\end{math}. Thus (\ref{eq3.2}) follows. 
\end{proof}

Thus we have the following corollary.
\begin{corollary}\label{corollary:2}
\begin{equation}
\Qmx{k,\ell,\emptyset,m}{t,0}=\Qmmx{k,\ell,\emptyset,m}{t,0}=\Qmmx{k,\ell,0,m}{t,0}
\end{equation}
	and
\begin{equation}
\Qmx{k,\ell,\emptyset,m}{t,x}\big\vert_{x^1}=\Qmmx{k,\ell,\emptyset,m}{t,x}\big\vert_{x^1}=
\Qmmx{k,\ell,0,m}{t,x}\big\vert_{x^1}.
\end{equation}
\end{corollary}

We note that \cite{KRT1,KRT2,KRT3} have many results on special cases of the coefficients of \begin{math}x^0\end{math} and \begin{math}x^1\end{math} in polynomials of the form 
\begin{math}\Qmmn{k,\ell,0,m}{n}\end{math}.

\subsection{The coefficients of the highest power of \(x\) that occurs in the polynomials 
	\(\Qmn{a,b,c,d}{n}\)}

By our results in the Section \begin{math}3\end{math}, to analyze the coefficients of the highest power of \begin{math}x\end{math} that occurs in the polynomials \begin{math}\Qmn{a,b,c,d}{n}\end{math}, we need only consider two cases. Namely, 
we need to analyze the coefficients of the highest power of \begin{math}x\end{math} that occurs in 
polynomials of the form \begin{math}\Qmn{0,k,0,\ell}{n}\end{math} and we need to analyze 
the coefficients of the highest power of \begin{math}x\end{math} that occurs in 
polynomials of the form \begin{math}\Qmmn{k,\ell,\emptyset,m}{n}\end{math}.

We shall start with analyzing the coefficients of the highest power of \begin{math}x\end{math} in polynomials 
of the form \begin{math}\Qmn{0,k,0,\ell}{n}\end{math}.  Clearly, in any permutation 
\begin{math}\sigma \in \Sn{n}(123)\end{math}, none of the numbers \begin{math}1, \ldots,\ell\end{math} or \begin{math}n,n-1, \ldots, n-k+1\end{math} 
can match \begin{math}\MMP(0,k,0,\ell)\end{math}.  It follows that the highest possible power of 
\begin{math}x\end{math} that can occur in \begin{math}\Qmn{0,k,0,\ell}{n}\end{math} is \begin{math}x^{n-k-\ell}\end{math} and its coefficient 
can be non-zero only if \begin{math}n \geq k+\ell+1\end{math}. Moreover, if \begin{math}\sigma_i\end{math} matches 
\begin{math}\MMP(0,k,0,\ell)\end{math} in \begin{math}\sigma\end{math}, then \begin{math}i \in \{k+1, \ldots, n-\ell\}\end{math}. It follows 
that if \begin{math}\mmp^{(0,k,0,\ell)}(\sigma) = n-k-\ell\end{math}, then 
\begin{enumerate}[(a)]
\item \begin{math}n-k+1,n-k+2, \ldots n\end{math} 
must be in positions \begin{math}1, \ldots, k\end{math}, 
\item \begin{math}\ell+1, \ldots, n-k\end{math} must be in positions 
\begin{math}k+1, \ldots, n-\ell\end{math}, and 
\item \begin{math}1, \ldots, \ell\end{math} must be in positions \begin{math}n-\ell+1, \ldots, n\end{math}. 
\end{enumerate}

These observations lead to the following theorem.
\begin{theorem}\label{theorem:4} If \begin{math}n \geq k + \ell + 1\end{math}, then 
\begin{equation}
\Qmn{0,k,0,\ell}{n}\big\vert_{x^{n-k-\ell}}=\Qmmn{0,k,0,\ell}{n}\big\vert_{x^{n-k-\ell}}=C_k C_{n-k-\ell}C_\ell.
\end{equation}
\end{theorem}
\begin{proof}
	Suppose \begin{math}n \geq k + \ell + 1\end{math}. 
	
	To have a \begin{math}\sigma \in \Sn{n}(123)\end{math} where \begin{math}\mmp^{(0,k,0,\ell)}(\sigma) = n-k-\ell\end{math}, 
	we need only have \begin{math}\sigma_1 \ldots \sigma_k\end{math} be any rearrangement of 
	\begin{math}n-k+1, \ldots, n\end{math}, which reduces to an element of \begin{math}\Sn{k}(123)\end{math} 
	which we can choose in \begin{math}C_k\end{math} ways, \begin{math}\sigma_{k+1} \ldots \sigma_{n-\ell}\end{math} be any rearrangement 
	of \begin{math}\ell+1, \ldots, n-k\end{math}, which reduces to an element of \begin{math}\Sn{n-k-\ell}(123)\end{math} which we can choose 
	in \begin{math}C_{n-k-\ell}\end{math} ways, and \begin{math}\sigma_{n-\ell+1} \ldots \sigma_n\end{math} be any rearrangement of 
	\begin{math}1, \ldots, \ell\end{math}, which is in \begin{math}\Sn{\ell}(123)\end{math} which we can choose in \begin{math}C_\ell\end{math} ways. 
	
\end{proof}

In the special case where \begin{math}\ell=0\end{math}, we have the following corollary. 
\begin{corollary}\label{corollary:4} If \begin{math}n \geq k+1\end{math}, then 
\begin{equation}
\Qmn{0,k,0,0}{n}\big\vert_{x^{n-k}}=\Qmmn{0,k,0,0}{n}\big\vert_{x^{n-k}}=C_k C_{n-k}.
\end{equation}
\end{corollary}

If we are considering the pattern \begin{math}\MMP(0,k,\emptyset,\ell)\end{math}, we can do a similar 
analysis.  The only difference is that for the numbers \begin{math}\ell+1, \ldots, n-k\end{math} to 
match \begin{math}\MMP(0,k,\emptyset,\ell)\end{math} in a 123-avoiding permutation \begin{math}\sigma\end{math}, they must all be peaks of \begin{math}\sigma\end{math} 
so that these numbers must occur in decreasing order. Thus we have the following 
theorem.
\begin{theorem}\label{theorem:04} For any \begin{math}n \geq k+\ell +1\end{math}, 
\begin{equation}
\Qmn{0,k,\emptyset,\ell}{n}\big\vert_{x^{n-k-\ell}}=\Qmmn{0,k,\emptyset,\ell}{n}\big\vert_{x^{n-k-\ell}}=C_k C_\ell.
\end{equation}
\end{theorem}

In the special case where \begin{math}\ell =0\end{math}, we have the following corollary.
\begin{corollary}\label{corollary:04}
\begin{equation}
\Qmn{0,k,\emptyset,0}{n}\big\vert_{x^{n-k}}=\Qmmn{0,k,\emptyset,0}{n}\big\vert_{x^{n-k}}=C_k.
\end{equation}
\end{corollary}

Notice that the numbers that match the pattern \begin{math}\MMP(0,k,\emptyset,\ell)\end{math} are on the diagonal under 
the maps \begin{math}\Phi\end{math} and \begin{math}\Psi\end{math} which means that they also have nothing in their first quadrant. Thus we have the following corollary.
\begin{corollary}\label{corollary:05}
\begin{eqnarray}
\Qmn{\emptyset,k,\emptyset,\ell}{n}\big\vert_{x^{n-k-\ell}}&=&\Qmmn{\emptyset,k,\emptyset,\ell}{n}\big\vert_{x^{n-k-\ell}}=C_k C_\ell,\\
\Qmn{\emptyset,k,\emptyset,0}{n}\big\vert_{x^{n-k}}&=&\Qmmn{\emptyset,k,\emptyset,0}{n}\big\vert_{x^{n-k}}=C_k.
\end{eqnarray}
\end{corollary}

Next we continue our analysis of the coefficients of the highest 
power of \begin{math}x\end{math} that can occur in polynomials of the form 
\begin{math}\Qmn{k,\ell,\emptyset,m}{n}\end{math}. We start by considering the special case where 
\begin{math}m=0\end{math}. Again, the highest power of \begin{math}x\end{math} that can occur in 
\begin{math}\Qmn{k,\ell,\emptyset,0}{n} =\Qmmn{k,\ell,\emptyset,0}{n}\end{math} is \begin{math}x^{n-k-\ell}\end{math} if 
\begin{math}n \geq k + \ell +1\end{math}.

\begin{theorem}\label{theorem:004} For all \begin{math}n \geq k + \ell +1\end{math}, 
\begin{equation}
\Qmn{k,\ell,\emptyset,0}{n}\big\vert_{x^{n-k-\ell}}=\Qmmn{k,\ell,\emptyset,0}{n}\big\vert_{x^{n-k-\ell}}=\frac{k+1}{k+\ell+1}\binom{k+2\ell}{\ell}.
\end{equation}
\end{theorem}
\begin{proof}
	Given \begin{math}\sigma = \sigma_1 \ldots \sigma_n \in \Sn{n}(132)\end{math}, if \begin{math}\sigma_i\end{math} matches 
	\begin{math}\MMP(k,\ell,\emptyset,0)\end{math}, it must be the case that \begin{math}\sigma_i\end{math} is a peak and there are  \begin{math}k+\ell\end{math} numbers larger than \begin{math}\sigma_i\end{math} in \begin{math}\sigma\end{math}. Thus if 
	we want \begin{math}\mmp^{(k,\ell,\emptyset,0)}(\sigma) = n-k-1\end{math}, then 
	the numbers \begin{math}\{1,2,\ldots,n-k-\ell\}\end{math} must be peaks and appear between the \begin{math}\ell+1^{\textnormal{st}}\end{math} position and \thn{n-k} position. Moreover, there should be \begin{math}k+\ell\end{math} numbers in the first \begin{math}k+\ell\end{math} rows, of which \begin{math}\ell\end{math} numbers appear before the numbers \begin{math}\{1,2,\ldots,n-k-\ell\}\end{math} and \begin{math}k\end{math} numbers appear after the numbers \begin{math}\{1,2,\ldots,n-k-\ell\}\end{math}. In \fref{lastresult}, the position of the numbers \begin{math}\{1,2,\ldots,n-k-\ell\}\end{math} are marked red while the position of the \begin{math}k+\ell\end{math} numbers are marked blue. The numbers \begin{math}\{1,2,\ldots,n-k-\ell\}\end{math} must 
	appear in  decreasing order since they are all peaks. The numbers in the blue region must reduce to a \begin{math}132\end{math}-avoiding permutation \begin{math}\tau\end{math} of size \begin{math}k+\ell\end{math} with an additional restriction that the numbers in the last \begin{math}k\end{math} columns must be increasing. Thus, we must count the number of Dyck paths of 
	length \begin{math}2(k+\ell)\end{math} that end in \begin{math}k\end{math} right-steps which is also equal to the number of standard 
	tableaux of shape \begin{math}(\ell,k+\ell)\end{math} which is equal to \begin{math}\frac{k+1}{k+\ell+1}\binom{k+2\ell}{\ell}\end{math} by the hook 
	formula for the number of standard tableaux. This fact is also proved by \cite{F,Sh}.

	Thus we have \begin{math}\Qmmn{k,\ell,\emptyset,0}{n}\big\vert_{x^{n-k-\ell}}=\frac{k+1}{k+\ell+1}\binom{k+2\ell}{\ell}\end{math}.
	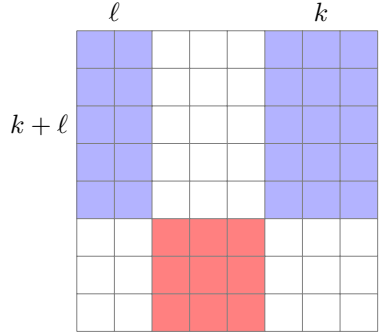
\begin{figure}[ht]
		\centering
		\vspace{-1mm}
		\begin{tikzpicture}[scale =.5]
		\fill[blue!30!white] (0,3) rectangle (2,8);
		\fill[blue!30!white] (5,3) rectangle (8,8);
		\fill[red!50!white] (5,3) rectangle (2,0);
		\draw[help lines] (0,0) grid (8,8);
		\node at (-1,5.5) {\begin{math}k+\ell\end{math}};
		\node at (1,8.5) {\begin{math}\ell\end{math}};
		\node at (6.5,8.5) {\begin{math}k\end{math}};
		\end{tikzpicture}
		\caption{Structure of \(\Qmmn{k,\ell,\emptyset,0}{n}\big\vert_{x^{n-k-\ell}}\)}
		\label{fig:lastresult}
	\end{figure}	
\end{proof}

\begin{theorem}\label{theorem:0004} For \begin{math}n \geq k+\ell+m+1\end{math} and \begin{math}k >0\end{math}, 
\begin{equation}
\Qmn{k,\ell,\emptyset,m}{n}\big\vert_{x^{n-k-\ell-m}}=\Qmmn{k,\ell,\emptyset,m}{n}\big\vert_{x^{n-k-\ell-m}}=\frac{(k+1)^2}{(k+\ell+1)(k+m+1)}\binom{k+2\ell}{\ell}\binom{k+2m}{m}.
\end{equation}
\end{theorem}
\begin{proof}
	Assume that \begin{math}n \geq k+\ell+m+1\end{math} and \begin{math}k > 0\end{math}. Then for \begin{math}\sigma_i\end{math} to match \begin{math}\MMP(k,\ell,\emptyset,m)\end{math} 
	in \begin{math}\sigma = \sigma_1 \ldots \sigma_n \in \Sn{n}(132)\end{math}, 
	\begin{math}\sigma_i\end{math} must be a peak of \begin{math}\sigma\end{math} and \begin{math}\sigma_i\end{math} must have \begin{math}m\end{math} numbers to its right in \begin{math}\sigma\end{math} which are 
	smaller than \begin{math}\sigma_i\end{math}, \begin{math}\ell\end{math} numbers to its left in \begin{math}\sigma\end{math} which are 
	larger than \begin{math}\sigma_i\end{math}, and \begin{math}k\end{math} numbers to its right in \begin{math}\sigma\end{math} which are larger than \begin{math}\sigma_i\end{math}. 
	It follows that the maximum power of \begin{math}x\end{math} that can appear in \begin{math}\Qmn{k,\ell,\emptyset,m}{n}\end{math} is 
	\begin{math}x^{n-k-\ell-m}\end{math}. Now if \begin{math}\mmp^{(k,\ell,\emptyset,m)}(\sigma) = n-k-\ell-m\end{math}, then 
	the numbers \begin{math}\{m+1,m+2,\ldots,n-k-\ell\}\end{math} must be peaks and appear between the \begin{math}\ell+1^{\textnormal{st}}\end{math} position and \thn{n-k-m} position in decreasing order. These positions are marked red in \fref{xka}. There should be \begin{math}k+\ell\end{math} numbers in the first \begin{math}k+\ell\end{math} rows, of which \begin{math}\ell\end{math} numbers appear before the numbers \begin{math}\{m+1,m+2,\ldots,n-k-\ell\}\end{math} and \begin{math}k\end{math} numbers appear after the numbers \begin{math}\{m+1,m+2,\ldots,n-k-\ell\}\end{math}; and there should be \begin{math}k+m\end{math} numbers in the last \begin{math}k+m\end{math} columns, of which \begin{math}m\end{math} numbers appear under the numbers \begin{math}\{m+1,m+2,\ldots,n-k-\ell\}\end{math} and \begin{math}k\end{math} numbers appear above the numbers \begin{math}\{m+1,m+2,\ldots,n-k-\ell\}\end{math}. In \fref{xka}, the position of these \begin{math}k+\ell+m\end{math} numbers are marked blue. The numbers in the blue region must reduce to a \begin{math}132\end{math}-avoiding permutation \begin{math}\tau\end{math} of size \begin{math}k+\ell\end{math} with an additional restriction that the numbers in the last \begin{math}k+m\end{math} columns and top \begin{math}k+\ell\end{math} rows (region \begin{math}A\end{math} in \fref{xka}) must be increasing. It is easy to see that under the map \begin{math}\Phi\end{math} such permutations correspond to  Dyck paths in the join 
	of the \begin{math}3\end{math} blue regions as pictured in \fref{xkb}. There have to be no peaks in region \begin{math}A\end{math} since the numbers in region \begin{math}A\end{math} are in an increasing order.
	It follows that 
	the coefficient of \begin{math}x^{n-k-\ell-m}\end{math} in \begin{math}\Qmmn{k,\ell,\emptyset,m}{n}\end{math} equals the number of Dyck paths \begin{math}U\end{math}
	of length \begin{math}2(k+\ell+m)\end{math} which pass through the points \begin{math}P\end{math}, \begin{math}Q\end{math} and \begin{math}R\end{math} in \fref{xkb}. 
	For each such path \begin{math}U\end{math}, we can uniquely associate two paths \begin{math}U_1\end{math} and \begin{math}U_2\end{math} where 
	\begin{math}U_1\end{math} starts at \begin{math}P\end{math} and goes to the point 
	\begin{math}Q\end{math}, and \begin{math}U_2\end{math} starts at \begin{math}Q\end{math} and goes to the point \begin{math}R\end{math}.  By our results in the previous theorem 
	the number of such \begin{math}U_1\end{math} is \begin{math}\frac{k+1}{k+\ell+1}\binom{k+2\ell}{\ell}\end{math} and the number of 
	such \begin{math}U_2\end{math} is \begin{math}\frac{k+1}{k+m+1} \binom{k+2m}{m}\end{math}. It follows that 
	\begin{math}\Qmmn{k,\ell,\emptyset,m}{n}\big\vert_{x^{n-k-\ell-m}}=
	\frac{(k+1)^2}{(k+\ell+1)(k+m+1)}\binom{k+2\ell}{\ell}\binom{k+2m}{m}\end{math}.
	\begin{figure}[ht]
		\centering
		\vspace{-1mm}
		\subfigure[\label{fig:xka}]{\begin{tikzpicture}[scale =.5]
		\fill[blue!30!white] (0,5) rectangle (1,8);
		\fill[blue!30!white] (3,5) rectangle (8,8);
		\fill[blue!30!white] (3,0) rectangle (8,3);
		\fill[red!50!white] (1,3) rectangle (3,5);
		\draw[help lines] (0,0) grid (8,8);
		\node at (9,6.5) {\begin{math}k+\ell\end{math}};
		\node at (.5,8.5) {\begin{math}\ell\end{math}};
		\node at (5.5,8.5) {\begin{math}k+m\end{math}};
		\node at (5.5,6.5) {\begin{math}A\end{math}};
		\node at (8.5,1.5) {\begin{math}m\end{math}};
		\path (-2,0);
		\end{tikzpicture}}
		\subfigure[\label{fig:xkb}]{\begin{tikzpicture}[scale =.5]
		\draw[ultra thick, fill=blue!30!white] (0,3) rectangle (1,6);
		\draw[ultra thick, fill=blue!30!white] (1,3) rectangle (6,6);
		\draw[ultra thick, fill=blue!30!white] (1,0) rectangle (6,3);
		\draw[help lines] (0,0) grid (6,6);
		\node at (7,4.5) {\begin{math}k+\ell\end{math}};
		\node at (.5,6.5) {\begin{math}\ell\end{math}};
		\node at (3.5,6.5) {\begin{math}k+m\end{math}};
		\node at (6.5,1.5) {\begin{math}m\end{math}};
		\node at (3.5,4.5) {\begin{math}A\end{math}};
		\node at (-.3,6.3) {\begin{math}P\end{math}};
		\node at (.6,2.5) {\begin{math}Q\end{math}};
		\node at (6.4,-.4) {\begin{math}R\end{math}};
		\path (-2,0);
		\draw (0,6)--(6,0);
		\draw[fill] (0,6) circle (.15);
		\draw[fill] (6,0) circle (.15);
		\draw[fill] (1,3) circle (.15);
		\path (-2,0);
		\end{tikzpicture}}
		\caption{Structure of \(\Qmmn{k,\ell,\emptyset,m}{n}\big\vert_{x^{n-k-\ell-m}}\) for \(k=2,l=1,m=3\)}
		\label{fig:lastresultt}
	\end{figure}
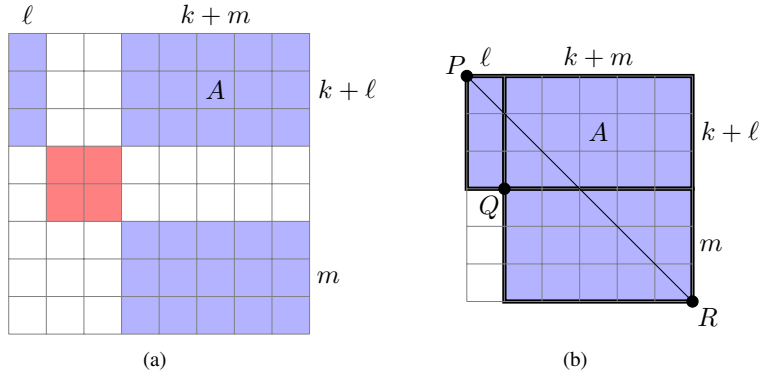	
	
	Since \begin{math}\Qmn{k,\ell,\emptyset,m}{n}=\Qmmn{k,\ell,\emptyset,m}{n}\end{math}, the theorem follows. 
\end{proof}

\section{The functions of form \(\Qm{k,\ell,0,m}=\Qm{k,\ell,\emptyset,m} \\= \Qmm{k,\ell,\emptyset,m}\)}

In this section, we shall show how we can compute generating functions 
of the form  \begin{math}\Qm{k,\ell,0,m}=\Qm{k,\ell,\emptyset,m}=\Qmm{k,\ell,\emptyset,m}\end{math}.  In 
this case, it is easier to compute generating functions of the form 
\begin{math}\Qmm{k,\ell,\emptyset,m}\end{math}. To do this, we will start by computing 
the marginal distributions \begin{math}\Qmm{k,0,\emptyset,0}\end{math}, \begin{math}\Qmm{0,\ell,\emptyset,0}\end{math}, and 
\begin{math}\Qmm{0,0,\emptyset,m}\end{math}. Then we can find expressions for 
\begin{math}\Qmm{k,\ell,\emptyset,0}\end{math}, \begin{math}\Qmm{0,\ell,\emptyset,m}\end{math}, 
and \begin{math}\Qmm{k,0,\emptyset,m}\end{math} in terms of the marginal distributions \begin{math}\Qmm{k,0,\emptyset,0}\end{math}, \begin{math}\Qmm{0,\ell,\emptyset,0}\end{math}, and \\
\begin{math}\Qmm{0,0,\emptyset,m}\end{math}. Finally we show how we can express 
\begin{math}\Qmm{k,\ell,\emptyset,m }\end{math} in terms of the distributions 
\begin{math}\Qmm{k,\ell,\emptyset,0}\end{math}, \begin{math}\Qmm{0,\ell,\emptyset,m}\end{math}, 
and \begin{math}\Qmm{k,0,\emptyset,m}\end{math}.

Recall that \cite{KRT1} proved that 
\begin{equation}
\Qmm{0,0,\emptyset,0}=\frac{(1+t-t x)-\sqrt{(1+t-t x)^2-4 t}}{2 t},
\end{equation}
and for \begin{math}k\geq 1\end{math},
\begin{equation}
\Qmm{k,0,\emptyset,0}=\frac{1}{1-t \Qmm{k-1,0,\emptyset,0}}.
\end{equation}

By Lemma \ref{sym}, we have that \begin{math}\Qmm{0,k,\emptyset,0} = \Qmm{0,0,\emptyset,k}\end{math}. 
Thus to complete our computations of the marginal distributions we need 
only compute \begin{math}\Qmm{0,k,\emptyset,0}\end{math} when \begin{math}k >  0\end{math}.

Let \begin{math}\Sn{n}^{(i)}(132)\end{math} be the set of \begin{math}\sigma=\sigma_1\cdots\sigma_n\in\Sn{n}(132)\end{math} such that \begin{math}\sigma_i=n\end{math}. Then the graph \begin{math}G(\sigma)\end{math} of each \begin{math}\sigma\in\Sn{n}^{(i)}(132)\end{math} has the structure showed in \fref{132a}. That is, in \begin{math}G(\sigma)\end{math}, the numbers to the left of \begin{math}n\end{math}, \begin{math}A_i(\sigma)\end{math}, have the structure of \begin{math}132\end{math}-avoiding permutation, the numbers to the right of \begin{math}n\end{math}, \begin{math}B_i(\sigma)\end{math}, have the structure of \begin{math}132\end{math}-avoiding permutation, and all the numbers in \begin{math}A_i(\sigma)\end{math} lie above all the numbers in \begin{math}B_i(\sigma)\end{math}. If we 
apply the map \begin{math}\Phi\end{math} to such permutations, then for \begin{math}\sigma\in\Sn{n}^{(i)}(132)\end{math}, 
\begin{math}\Phi(\sigma)\end{math} will  be a Dyck path of the form in \fref{132b} where the smaller Dyck path structures \begin{math}A_i\end{math} and \begin{math}B_i\end{math} correspond to \begin{math}132\end{math}-avoiding permutation structures \begin{math}A_i(\sigma)\end{math} and \begin{math}B_i(\sigma)\end{math}.

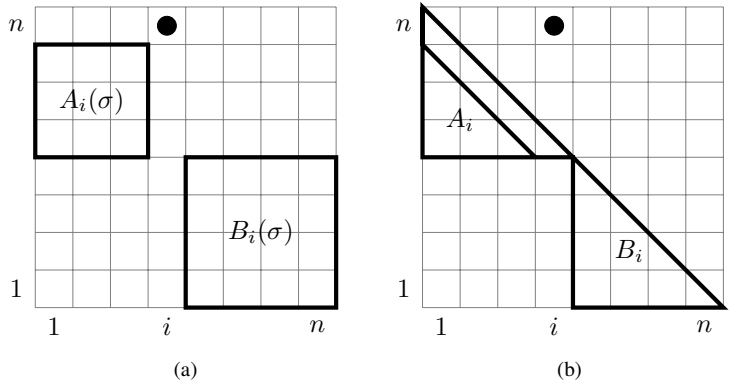
\begin{figure}[ht]
	\centering
	\subfigure[\label{fig:132a}]{\begin{tikzpicture}[scale =.5]
	\draw[help lines] (0,0) grid (8,8);
	\draw[ultra thick] (0,4) rectangle (3,7);
	\draw[ultra thick] (4,0) rectangle (8,4);
	\draw[fill] (3.5,7.5) circle (.25);
	\node at (1.5,5.5) {\begin{math}A_i(\sigma)\end{math}};
	\node at (3.5,-.5) {\begin{math}i\end{math}};
	\node at (6,2) {\begin{math}B_i(\sigma)\end{math}};
	\fillll{1}{0}{1};
	\fillll{0}{1}{1};
	\fillll{8}{0}{n};
	\fillll{0}{8}{n};
	\path (9,0);
	\end{tikzpicture}	}
	\subfigure[\label{fig:132b}]{\begin{tikzpicture}[scale =.5]
	\draw[help lines] (0,0) grid (8,8);
	\draw[ultra thick] (0,8) -- (0,4) -- (4,4) -- (4,0) -- (8,0) -- (0,8) -- (0,7) -- (3,4);
	\draw[fill] (3.5,7.5) circle (.25);
	\node at (1,5) {\begin{math}A_i\end{math}};
	\node at (3.5,-.5) {\begin{math}i\end{math}};
	\node at (5.5,1.5) {\begin{math}B_i\end{math}};
	\fillll{1}{0}{1};
	\fillll{0}{1}{1};
	\fillll{8}{0}{n};
	\fillll{0}{8}{n};
	\path (9,0);
	\end{tikzpicture}}
	\vspace{-3mm}
	\caption{Structures of \(\Sn{n}(132)\)  and \(\mathcal{D}_n\)}
	\label{fig:structure132}
\end{figure}

Now assume that \begin{math}k>0\end{math}. Then we can derive a simple recursion for \begin{math}\Sn{n}(132)\end{math}
based on the position of \begin{math}n\end{math} in a permutation \begin{math}\sigma = \sigma_1 \ldots \sigma_n \in \Sn{n}(132)\end{math}. 
That is, suppose \begin{math}\sigma_i =n\end{math} and \begin{math}A_i(\sigma)\end{math} and \begin{math}B_i(\sigma)\end{math} are as pictured in  
\fref{structure132}. Clearly \begin{math}\sigma_i =n\end{math} does not match  \begin{math}\MMP(0,k,\emptyset,0)\end{math} in \begin{math}\sigma\end{math}. 
Then we have two cases.
\begin{enumerate}[{\bf Case }\bf 1.]
\item \begin{math}i < k\end{math}.\\
Then elements in 
\begin{math}A_i(\sigma)\end{math} cannot match \begin{math}\MMP(0,k,\emptyset,0)\end{math} in \begin{math}\sigma\end{math} since no element 
of \begin{math}A_i(\sigma)\end{math} has \begin{math}k\end{math} elements to its right which are larger than it. However, an 
element \begin{math}\sigma_j\end{math} in \begin{math}B_i(\sigma)\end{math}  matches \begin{math}\MMP(0,k,\emptyset,0)\end{math} in \begin{math}\sigma\end{math} if and only 
if it matches \begin{math}\MMP(0,k-i,\emptyset,0)\end{math} in \begin{math}B_i(\sigma)\end{math}. Thus such permutations 
contribute \begin{math}C_{i-1}\Qmmn{0,k-i,\emptyset,0}{n-i}\end{math} to \begin{math}\Qmmn{0,k,\emptyset,0}{n}\end{math}. 
\item \begin{math}i> k\end{math}.\\
Then elements in 
\begin{math}A_i(\sigma)\end{math} match \begin{math}\MMP(0,k,\emptyset,0)\end{math} in \begin{math}\sigma\end{math} if and only if 
the corresponding element matches \begin{math}\MMP(0,k,\emptyset,0)\end{math} in the reduction 
of \begin{math}A_i(\sigma)\end{math}.  An element \begin{math}\sigma_j\end{math} in \begin{math}B_i(\sigma)\end{math} automatically has \begin{math}k\end{math} elements in the graph 
\begin{math}G(\sigma)\end{math} in the second quadrant relative to the coordinate system centered at \begin{math}(j,\sigma_j)\end{math}, 
namely, the elements in \begin{math}A_i(\sigma) \cup \{n\}\end{math} so that \begin{math}\sigma_j\end{math} matches 
\begin{math}\MMP(0,k,\emptyset,0)\end{math} in \begin{math}\sigma\end{math} if and only if \begin{math}\sigma_j\end{math} is a peak of \begin{math}\sigma\end{math}, or, equivalently, 
if and only if \begin{math}\sigma_j\end{math} matches \begin{math}\MMP(0,0,\emptyset,0)\end{math} in \begin{math}B_i(\sigma)\end{math}. Thus such permutations 
contribute \begin{math}\Qmmn{0,k,\emptyset,0}{i-1} \Qmmn{0,0,\emptyset,0}{n-i}\end{math} to \begin{math}\Qmmn{0,k,\emptyset,0}{n}\end{math}.
\end{enumerate}

It follows that for \begin{math}n \geq k+1\end{math}, 
\begin{equation}
\Qmmn{0,k,\emptyset,0}{n}=\sum_{i=1}^{k-1}C_{i-1}\Qmmn{0,k-i,\emptyset,0}{n-i}+\sum_{i=k}^{n}\Qmmn{0,k,\emptyset,0}{i-1}\Qmmn{0,0,\emptyset,0}{n-i}.
\end{equation}

Multiplying both sides of the equation by \begin{math}t^n\end{math} and summing for \begin{math}n\geq 1\end{math} gives that
\begin{equation}
\Qmm{0,k,\emptyset,0}=1+t\sum_{i=1}^{k-1}C_{i-1}t^{i-1}\Qmm{0,k-i,\emptyset,0}+t(\Qmm{0,k,\emptyset,0}-\sum_{i=0}^{k-2}C_i t^i)\Qmm{0,0,\emptyset,0}.
\end{equation}

Thus, we have the following theorem.
\begin{theorem}\label{theorem:6}
	\begin{equation}
	\Qmm{0,0,\emptyset,0}=\frac{1+t-t x -\sqrt{(1+t-t x)^2-4t}}{2t}.
	\end{equation}
	For \begin{math}k>0\end{math},
	\begin{equation}
	\Qmm{0,k,\emptyset,0}=\frac{1+t\sum_{i=1}^{k-1}C_{i-1}t^{i-1}(\Qmm{0,k-i,\emptyset,0}-\Qmm{0,0,\emptyset,0})}{1-t\Qmm{0,0,\emptyset,0})}.
	\end{equation}
\end{theorem}

We list the first \begin{math}10\end{math} terms of function \begin{math}\Qmm{0,k,\emptyset,0}\end{math} for \begin{math}k=1, \ldots, 5\end{math}.
\begin{align}
\Qmm{0,1,\emptyset,0}=&1+t+(1+x) t^2+\left(1+3 x+x^2\right) t^3+\left(1+6 x+6 x^2+x^3\right)
t^4\nonumber\\\nonumber
&+\left(1+10 x+20 x^2+10 x^3+x^4\right) t^5+\left(1+15 x+50 x^2+50 x^3+15
x^4+x^5\right) t^6\\\nonumber
&+\left(1+21 x+105 x^2+175 x^3+105 x^4+21 x^5+x^6\right)
t^7\\\nonumber
&+\left(1+28 x+196 x^2+490 x^3+490 x^4+196 x^5+28 x^6+x^7\right)
t^8\\
&+\left(1+36 x+336 x^2+1176 x^3+1764 x^4+1176 x^5+336 x^6+36
x^7+x^8\right) t^9+\cdots
\end{align}

We note that if \begin{math}\sigma_i\end{math} matches \begin{math}\MMP(0,1,\emptyset,0)\end{math} in 
\begin{math}\sigma = \sigma_1 \ldots \sigma_n \in \Sn{n}(132)\end{math}, then \begin{math}\sigma_i\end{math} must be a peak of \begin{math}\sigma\end{math} which has 
at least one element to its left which is larger than \begin{math}\sigma_i\end{math}.  However, it is easy to see 
from our description of \begin{math}\Phi^{-1}\end{math}, the every peak except the first one in \begin{math}\sigma\end{math} satisfies 
this condition. However such peaks are just the descents of \begin{math}\sigma\end{math} so that  
\begin{math}\Qmmn{0,1,\emptyset,0}{n} = \sum_{\sigma \in \Sn{n}(132)} x^{\mathrm{des}(\sigma)}\end{math}.

\begin{align}
\Qmm{0,2,\emptyset,0}=&1+t+2 t^2+(3+2 x) t^3+\left(4+8 x+2 x^2\right) t^4+\left(5+20 x+15 x^2+2
x^3\right) t^5\nonumber\\\nonumber
&+\left(6+40 x+60 x^2+24 x^3+2 x^4\right) t^6\\\nonumber
&+\left(7+70 x+175
x^2+140 x^3+35 x^4+2 x^5\right) t^7\\\nonumber
&+\left(8+112 x+420 x^2+560 x^3+280 x^4+48
x^5+2 x^6\right) t^8\\
&+\left(9+168 x+882 x^2+1764 x^3+1470 x^4+504 x^5+63
x^6+2 x^7\right) t^9+\cdots
\\
\Qmm{0,3,\emptyset,0}=&1+t+2 t^2+5 t^3+(9+5 x) t^4+\left(14+23 x+5 x^2\right) t^5\nonumber\\\nonumber
&+\left(20+65 x+42
x^2+5 x^3\right) t^6+\left(27+145 x+186 x^2+66 x^3+5 x^4\right)t^7\\\nonumber
&+\left(35+280 x+595 x^2+420 x^3+95 x^4+5 x^5\right) t^8\\
&+\left(44+490
x+1554 x^2+1820 x^3+820 x^4+129 x^5+5 x^6\right) t^9+\cdots
\\
\Qmm{0,4,\emptyset,0}=&1+t+2 t^2+5 t^3+14 t^4+(28+14 x) t^5+\left(48+70 x+14 x^2\right)
t^6\nonumber\\\nonumber
&+\left(75+214 x+126 x^2+14 x^3\right) t^7+\left(110+514 x+596 x^2+196
x^3+14 x^4\right) t^8\\
&+\left(154+1064 x+2030 x^2+1320 x^3+280 x^4+14
x^5\right) t^9+\cdots
\\
\Qmm{0,5,\emptyset,0}=&1+t+2 t^2+5 t^3+14 t^4+42 t^5+(90+42 x) t^6+\left(165+222 x+42 x^2\right)
t^7\nonumber\\\nonumber
&+\left(275+717 x+396 x^2+42 x^3\right) t^8\\
&+\left(429+1817 x+1962 x^2+612
x^3+42 x^4\right) t^9+\cdots
\end{align}

Next we consider \begin{math}\Qmm{k,\ell,\emptyset,0}\end{math} where both \begin{math}k\end{math} and \begin{math}\ell\end{math} are nonzero. 
Again we will develop a simple recursion for \begin{math}\Qmmn{k,\ell,\emptyset,0}{n}\end{math} based on the position of \begin{math}n\end{math} 
in \begin{math}\sigma\end{math}.  That is, let \begin{math}\sigma = \sigma_1 \ldots \sigma_n \in\Sn{n}(132)\end{math} and \begin{math}\sigma_i=n\end{math}. 
Again \begin{math}\sigma_i =n\end{math} does not match \begin{math}\MMP(k,\ell,\emptyset,0)\end{math} in \begin{math}\sigma\end{math}.  Then we have two cases. 
\begin{enumerate}[{\bf Case }\bf 1.]
\item \begin{math}i <  \ell\end{math}. \\
Then no element in \begin{math}A_i(\sigma)\end{math} cannot match \begin{math}\MMP(k,\ell,\emptyset,0)\end{math} 
in \begin{math}\sigma\end{math} since no element of \begin{math}A_i(\sigma)\end{math} has \begin{math}\ell\end{math} elements to its left which are larger than it. 
An element \begin{math}\sigma_j\end{math} in \begin{math}B_i(\sigma)\end{math} matches 
\begin{math}\MMP(k,\ell,\emptyset,0)\end{math} in \begin{math}\sigma\end{math} if and only if it matches \begin{math}\MMP(k,\ell-i,\emptyset,0)\end{math} in 
\begin{math}B_i(\sigma)\end{math}.  Thus such permutations contribute 
\begin{math}\sum_{i=1}^{\ell-1}C_{i-1}\Qmmn{k,\ell-i,\emptyset,0}{n-i}\end{math} to  \begin{math}\Qmmn{k,\ell,\emptyset,0}{n}\end{math}.

\item \begin{math}i > \ell\end{math}. \\
For an element \begin{math}\sigma_j\end{math} 
in \begin{math}A_i(\sigma)\end{math}, since the number \begin{math}n\end{math} is to its right and larger than it, \begin{math}\sigma_j\end{math} matches \begin{math}\MMP(k,\ell,\emptyset,0)\end{math} in \begin{math}\sigma\end{math} 
if and only if, in the reduction of \begin{math}A_i(\sigma)\end{math}, the corresponding element matches \begin{math}\MMP(k-1,\ell,\emptyset,0)\end{math}. An element \begin{math}\sigma_j\end{math} in \begin{math}B_i(\sigma)\end{math} automatically has \begin{math}\ell\end{math} elements to its left which are larger than 
it so that \begin{math}\sigma_j\end{math} matches  
\begin{math}\MMP(k,\ell,\emptyset,0)\end{math} in \begin{math}\sigma\end{math} if and only if it matches \begin{math}\MMP(k,0,\emptyset,0)\end{math} in 
\begin{math}B_i(\sigma)\end{math}. Thus such permutations contribute \begin{math}\sum_{i=\ell}^{n}\Qmmn{k-1,\ell,\emptyset,0}{i-1}\Qmmn{k,0,\emptyset,0}{n-i}\end{math} to \begin{math}\Qmmn{k,\ell,\emptyset,0}{n}\end{math}.
\end{enumerate}

It follows that for \begin{math}n \geq k+ \ell +1\end{math}, 
\begin{equation}
\Qmmn{k,\ell,\emptyset,0}{n}=\sum_{i=1}^{\ell-1}C_{i-1}\Qmmn{k,\ell-i,\emptyset,0}{n-i}+\sum_{i=\ell}^{n}\Qmmn{k-1,\ell,\emptyset,0}{i-1}\Qmmn{k,0,\emptyset,0}{n-i}.
\end{equation}
Multiplying both sides of the equation by \begin{math}t^n\end{math} and summing for \begin{math}n\geq 1\end{math} gives that
\begin{equation}
\Qmm{k,\ell,\emptyset,0}=1+t\sum_{i=1}^{\ell-1}C_{i-1}t^{i-1}\Qmm{k,\ell-i,\emptyset,0}+(\Qmm{k-1,\ell,\emptyset,0}-\sum_{i=0}^{\ell-2}C_it^i)\Qmm{k,0,\emptyset,0}.
\end{equation}
Thus, we have the following theorem.
\begin{theorem}\label{theorem:7}
	For all \begin{math}k,\ell>0\end{math},
	\begin{equation}
	\Qmm{k,\ell,\emptyset,0}=1+t\sum_{i=1}^{\ell-1}C_{i-1}t^{i-1}\Qmm{k,\ell-i,\emptyset,0}+(\Qmm{k-1,\ell,\emptyset,0}-\sum_{i=0}^{\ell-2}C_it^i)\Qmm{k,0,\emptyset,0}.
	\end{equation}	
\end{theorem}

We list the first \begin{math}10\end{math} terms of function \begin{math}\Qmm{k,\ell,\emptyset,0}\end{math} for \begin{math}1\leq k,\ell\leq 3\end{math}.
\begin{align}
\Qmm{1,1,\emptyset,0}=&1+t+2 t^2+(3+2 x) t^3+\left(4+8 x+2 x^2\right) t^4+\left(5+20 x+15 x^2+2
x^3\right) t^5\nonumber\\\nonumber
&+\left(6+40 x+60 x^2+24 x^3+2 x^4\right) t^6\\\nonumber
&+\left(7+70 x+175
x^2+140 x^3+35 x^4+2 x^5\right) t^7\\\nonumber
&+\left(8+112 x+420 x^2+560 x^3+280 x^4+48
x^5+2 x^6\right) t^8\\
&+\left(9+168 x+882 x^2+1764 x^3+1470 x^4+504 x^5+63
x^6+2 x^7\right) t^9+\cdots
\\
\Qmm{1,2,\emptyset,0}=&1+t+2 t^2+5 t^3+(9+5 x) t^4+\left(14+23 x+5 x^2\right) t^5\nonumber\\\nonumber
&+\left(20+65 x+42
x^2+5 x^3\right) t^6+\left(27+145 x+186 x^2+66 x^3+5 x^4\right)
t^7\\\nonumber
&+\left(35+280 x+595 x^2+420 x^3+95 x^4+5 x^5\right) t^8\\
&+\left(44+490
x+1554 x^2+1820 x^3+820 x^4+129 x^5+5 x^6\right) t^9+\cdots
\\
\Qmm{1,3,\emptyset,0}=&1+t+2 t^2+5 t^3+14 t^4+(28+14 x) t^5+\left(48+70 x+14 x^2\right)
t^6\nonumber\\\nonumber
&+\left(75+214 x+126 x^2+14 x^3\right) t^7+\left(110+514 x+596 x^2+196
x^3+14 x^4\right) t^8\\
&+\left(154+1064 x+2030 x^2+1320 x^3+280 x^4+14
x^5\right) t^9+\cdots
\\
\Qmm{2,1,\emptyset,0}=&1+t+2 t^2+5 t^3+(11+3 x) t^4+\left(23+16 x+3 x^2\right) t^5\nonumber\\\nonumber
&+\left(47+56 x+26
x^2+3 x^3\right) t^6+\left(95+163 x+129 x^2+39 x^3+3 x^4\right)
t^7\\\nonumber
&+\left(191+429 x+489 x^2+263 x^3+55 x^4+3 x^5\right) t^8\\
&+\left(383+1062
x+1583 x^2+1270 x^3+487 x^4+74 x^5+3 x^6\right) t^9+\cdots
\end{align}
\begin{align}
\Qmm{2,2,\emptyset,0}=&1+t+2 t^2+5 t^3+14 t^4+(33+9 x) t^5+\left(72+51 x+9 x^2\right)
t^6\nonumber\\\nonumber
&+\left(151+186 x+83 x^2+9 x^3\right) t^7+\left(310+556 x+431 x^2+124
x^3+9 x^4\right) t^8\\
&+\left(629+1487 x+1688 x^2+875 x^3+174 x^4+9 x^5\right)
t^9+\cdots
\\
\Qmm{2,3,\emptyset,0}=&1+t+2 t^2+5 t^3+14 t^4+42 t^5+(104+28 x) t^6+\left(235+166 x+28 x^2\right)
t^7\nonumber\\\nonumber
&+\left(505+627 x+270 x^2+28 x^3\right) t^8\\
&+\left(1054+1924 x+1454 x^2+402
x^3+28 x^4\right) t^9+\cdots
\\
\Qmm{3,1,\emptyset,0}=&1+t+2 t^2+5 t^3+14 t^4+(38+4 x) t^5+\left(101+27 x+4 x^2\right)
t^6\nonumber\\\nonumber
&+\left(266+119 x+40 x^2+4 x^3\right) t^7+\left(698+439 x+232 x^2+57 x^3+4
x^4\right) t^8\\
&+\left(1829+1477 x+1044 x^2+430 x^3+78 x^4+4 x^5\right)
t^9+\cdots
\\
\Qmm{3,2,\emptyset,0}=&1+t+2 t^2+5 t^3+14 t^4+42 t^5+(118+14 x) t^6+\left(319+96 x+14 x^2\right)
t^7\nonumber\\\nonumber
&+\left(847+425 x+144 x^2+14 x^3\right) t^8\\
&+\left(2231+1563 x+848 x^2+206
x^3+14 x^4\right) t^9+\cdots
\\
\Qmm{3,3,\emptyset,0}=&1+t+2 t^2+5 t^3+14 t^4+42 t^5+132 t^6+(381+48 x) t^7\nonumber\\
&+\left(1046+336 x+48
x^2\right) t^8+\left(2801+1506 x+507 x^2+48 x^3\right)
t^9+\cdots
\end{align}

If one compares the polynomials \begin{math}\Qmmn{0,k,\emptyset,0}{n}\end{math} 
and \begin{math}\Qmmn{1,k-1,\emptyset,0}{n}\end{math}, one observes that they are equal for \begin{math}k \geq 1\end{math}. Thus 
we make the following conjecture.
\begin{conjecture}
	For all \begin{math}k\geq 1\end{math}, we have
	\begin{equation}
		\Qmm{0,k,\emptyset,0}=\Qmm{1,k-1,\emptyset,0}.
	\end{equation}
\end{conjecture}

We have verified the conjecture for \begin{math}k=1,2,3\end{math} by directly computing the generating functions. 
That is, we can prove that 

\begin{eqnarray}
	\Qmm{0,1,\emptyset,0}&=&\Qmm{1,0,\emptyset,0},\\
	\Qmm{0,2,\emptyset,0}&=&\Qmm{1,1,\emptyset,0},\\
	\Qmm{0,3,\emptyset,0}&=&\Qmm{1,2,\emptyset,0}.
\end{eqnarray}

However, it is not obvious from the corresponding recursions for 
\begin{math}\Qmmn{0,k,\emptyset,0}{n}\end{math} and \begin{math}\Qmmn{1,k-1,\emptyset,0}{n}\end{math} that these two polynomials 
are equal.

Next we consider the generating functions \begin{math}\Qmm{0,k,\emptyset,\ell}=\Qm{0,k,\emptyset,\ell}\end{math}  
where \begin{math}k,\ell > 0\end{math}.

When \begin{math}n \leq k+\ell\end{math}, there is no element of a \begin{math}\sigma \in \Sn{n}(132)\end{math} that 
can match \begin{math}\MMP(0,k,\emptyset,\ell)\end{math} in \begin{math}\sigma\end{math}. Thus \begin{math}\Qmmn{0,k,\emptyset,\ell}{n} = C_n\end{math} 
in such cases. Thus assume that \begin{math}n \geq k+\ell +1\end{math} and \begin{math}\sigma =\sigma_1 \ldots \sigma_n \in 
\Sn{n}(132)\end{math} is such that \begin{math}\sigma_i =n\end{math}. Clearly \begin{math}\sigma_i\end{math} cannot match 
\begin{math}\MMP(0,k,\emptyset,\ell)\end{math} in \begin{math}\sigma\end{math}. We then have 3 cases. 
\begin{enumerate}[{\bf Case }\bf 1.]
	\item \begin{math}i < k\end{math}. \\
	Clearly no \begin{math}\sigma_j\end{math} in \begin{math}A_i(\sigma) \end{math} can match \begin{math}\MMP(0,k,\emptyset,\ell)\end{math} since it cannot 
	have \begin{math}k\end{math} elements to its left which are larger than it. A \begin{math}\sigma_j \in B_i(\sigma)\end{math} 
	matches  \begin{math}\MMP(0,k,\emptyset,\ell)\end{math} in \begin{math}\sigma\end{math} if and only if it matches 
	\begin{math}\MMP(0,k-i,\emptyset,\ell)\end{math} in \begin{math}B_i(\sigma)\end{math}.  Thus such 
	permutations contribute \begin{math}\sum_{i=1}^{k-1}C_{i-1}\Qmmn{0,k-i,\emptyset,\ell}{n-i}\end{math} to 
	\begin{math}\Qmmn{0,k,\emptyset,\ell}{n}\end{math}. 
	
	\item \begin{math}k\leq i \leq n-\ell\end{math}.\\
	For each peak \begin{math}\sigma_j \in A_i(\sigma)\end{math},   
	there are \begin{math}n-i\geq \ell\end{math} numbers in \begin{math}B_i(\sigma)\end{math} which are to its right and smaller than it 
	so that \begin{math}\sigma_j\end{math} matches \begin{math}\MMP(0,k,\emptyset,\ell)\end{math} in \begin{math}\sigma\end{math} if and only if, 
	in the reduction of \begin{math}A_i(\sigma)\end{math}, its corresponding element matches 
	\begin{math}\MMP(0,k,\emptyset,0)\end{math}. For each peak \begin{math}\sigma_j \in B_i(\sigma)\end{math},   
	there are \begin{math}\geq k\end{math} numbers in \begin{math}A_i(\sigma) \cup \{n\}\end{math} which are to its left  and larger  than it 
	so that \begin{math}\sigma_j\end{math} matches \begin{math}\MMP(0,k,\emptyset,\ell)\end{math} in \begin{math}\sigma\end{math} if and only if 
	\begin{math}\sigma_j\end{math} matches \begin{math}\MMP(0,0,\emptyset,\ell)\end{math} in \begin{math}B_i(\sigma)\end{math}.  
	Thus such 
	permutations contribute \begin{math}\sum_{i=k}^{n-\ell}\Qmmn{0,k,\emptyset,0}{i-1}\Qmmn{0,0,\emptyset,\ell}{n-i}\end{math} to 
	\begin{math}\Qmmn{0,k,\emptyset,\ell}{n}\end{math}.
	
	\item \begin{math}i\geq n-\ell+1\end{math}. \\
	For each peak \begin{math}\sigma_j \in A_i(\sigma)\end{math},   
	there are \begin{math}n-i\geq \ell\end{math} numbers in \begin{math}B_i(\sigma)\end{math} which are to its right and smaller than it 
	so that \begin{math}\sigma_j\end{math} matches \begin{math}\MMP(0,k,\emptyset,\ell)\end{math} in \begin{math}\sigma\end{math} if and only if, 
	in the reduction of \begin{math}A_i(\sigma)\end{math}, its corresponding element matches 
	\begin{math}\MMP(0,k,\emptyset,0)\end{math}. Clearly no element of \begin{math}B_i(\sigma) \end{math} can match \begin{math}\MMP(0,k,\emptyset,\ell)\end{math} since 
	it cannot have \begin{math}\ell\end{math} elements to its right which are smaller than it. 
	Thus such 
	permutations contribute \begin{math}\sum_{i=n-\ell+1}^{n} \Qmmn{0,k,\emptyset,\ell-(n-i)}{i-1} C_{n-i}\end{math} to 
	\begin{math}\Qmmn{0,k,\emptyset,\ell}{n}\end{math}.
	
\end{enumerate}

It follows that for \begin{math}n \geq k+\ell+1\end{math}, 
\begin{eqnarray}
	\Qmmn{0,k,\emptyset,\ell}{n}&=&\sum_{i=1}^{k-1}C_{i-1}\Qmmn{0,k-i,\emptyset,\ell}{n-i}+\sum_{i=k}^{n-\ell}\Qmmn{0,k,\emptyset,0}{i-1}\Qmmn{0,0,\emptyset,\ell}{n-i}\nonumber\\
	&&+\sum_{i=n-\ell+1}^{n} \Qmmn{0,k,\emptyset,\ell-(n-i)}{i-1} C_{n-i}.
\end{eqnarray}
Multiplying both sides of the equation by \begin{math}t^n\end{math} and summing for \begin{math}n\geq k+\ell+1\end{math} gives that
\begin{eqnarray}
	\Qmm{0,k,\emptyset,\ell}-\sum_{i=0}^{k+\ell}C_i t^i&=&t\sum_{i=1}^{k-1}C_{i-1}t^{i-1}(\Qmm{0,k-i,\emptyset,\ell}-\sum_{j=0}^{k+\ell-i-1}C_j t^j)\nonumber\\\nonumber
	&&+t(\Qmm{0,k,\emptyset,0}-\sum_{i=0}^{k-2}C_i t^i)(\Qmm{0,0,\emptyset,\ell}-\sum_{i=0}^{\ell-1}C_i t^i)\\
	&&+t\sum_{i=0}^{\ell-1}C_it^i(\Qmm{0,k,\emptyset,\ell-i}-\sum_{j=0}^{k+\ell-i-2}C_j t^j).
\end{eqnarray}

Note the first term of the last term \begin{math}t\sum_{i=0}^{\ell-1}C_it^i(\Qmm{0,k,\emptyset,\ell-i}-\sum_{j=0}^{k+\ell-i-2}C_j t^j)\end{math} on the right-hand side of the equation above is \begin{math}t(\Qmm{0,k,\emptyset,\ell}-\sum_{j=0}^{k+\ell-2}C_j t^j)\end{math}, so we can bring \begin{math}t\Qmm{0,k,\emptyset,\ell}\end{math} to the other side and solve \begin{math}\Qmm{0,k,\emptyset,\ell}\end{math} to obtain the following theorem.
\begin{theorem}\label{theorem:8}
	For all \begin{math}k,\ell>0\end{math},
	\begin{equation}
	\Qmm{0,k,\emptyset,\ell}=\frac{\Gamma_{k,\ell}(t,x)}{1-t},
	\end{equation}	
	where
	\begin{eqnarray}
		\Gamma_{k,\ell}(t,x)&=&\nonumber
		\sum_{i=0}^{k+\ell}C_i t^i - \sum_{i=0}^{k+\ell-2}C_i t^{i+1}
		+t\sum_{i=1}^{k-1}C_{i-1}t^{i-1}(\Qmm{0,k-i,\emptyset,\ell}-\sum_{j=0}^{k+\ell-i-1}C_j t^j)\\\nonumber
		&&+t(\Qmm{0,k,\emptyset,0}-\sum_{i=0}^{k-2}C_i t^i)(\Qmm{0,0,\emptyset,\ell}-\sum_{i=0}^{\ell-1}C_i t^i)\\
		&&+t\sum_{i=1}^{\ell-1}C_it^i(\Qmm{0,k,\emptyset,\ell-i}-\sum_{j=0}^{k+\ell-i-2}C_j t^j).
	\end{eqnarray}
\end{theorem}

We list the first \begin{math}10\end{math} terms of function \begin{math}\Qmm{0,k,\emptyset,\ell}\end{math} for \begin{math}1\leq k\leq \ell\leq 3\end{math}.
\begin{align}
\Qmm{0,1,\emptyset,1}=&1+t+2 t^2+(4+x) t^3+\left(7+6 x+x^2\right) t^4+\left(11+20 x+10 x^2+x^3\right)
t^5\nonumber\\\nonumber
&+\left(16+50 x+50 x^2+15 x^3+x^4\right) t^6\\\nonumber
&+\left(22+105 x+175 x^2+105
x^3+21 x^4+x^5\right) t^7\\\nonumber
&+\left(29+196 x+490 x^2+490 x^3+196 x^4+28
x^5+x^6\right) t^8\\
&+\left(37+336 x+1176 x^2+1764 x^3+1176 x^4+336 x^5+36
x^6+x^7\right) t^9+\cdots
\\
\Qmm{0,1,\emptyset,2}=&1+t+2 t^2+5 t^3+(12+2 x) t^4+\left(25+15 x+2 x^2\right) t^5\nonumber\\\nonumber
&+\left(46+60 x+24
x^2+2 x^3\right) t^6+\left(77+175 x+140 x^2+35 x^3+2 x^4\right)t^7\\\nonumber
&+\left(120+420 x+560 x^2+280 x^3+48 x^4+2 x^5\right) t^8\\
&+\left(177+882
x+1764 x^2+1470 x^3+504 x^4+63 x^5+2 x^6\right) t^9+\cdots
\\
\Qmm{0,1,\emptyset,3}=&1+t+2 t^2+5 t^3+14 t^4+(37+5 x) t^5+\left(85+42 x+5 x^2\right)
t^6\nonumber\\\nonumber
&+\left(172+186 x+66 x^2+5 x^3\right) t^7+\left(315+595 x+420 x^2+95 x^3+5
x^4\right) t^8\\
&+\left(534+1554 x+1820 x^2+820 x^3+129 x^4+5 x^5\right)
t^9+\cdots
\\
\Qmm{0,2,\emptyset,2}=&1+t+2 t^2+5 t^3+14 t^4+(38+4 x) t^5+\left(91+37 x+4 x^2\right)
t^6\nonumber\\\nonumber
&+\left(192+176 x+57 x^2+4 x^3\right) t^7+\left(365+595 x+385 x^2+81 x^3+4
x^4\right) t^8\\
&+\left(639+1624 x+1750 x^2+736 x^3+109 x^4+4 x^5\right)
t^9+\cdots
\\
\Qmm{0,2,\emptyset,3}=&1+t+2 t^2+5 t^3+14 t^4+42 t^5+(122+10 x) t^6+\left(316+103 x+10 x^2\right)
t^7\nonumber\\\nonumber
&+\left(724+540 x+156 x^2+10 x^3\right) t^8\\
&+\left(1493+1995 x+1145 x^2+219
x^3+10 x^4\right) t^9+\cdots
\\
\Qmm{0,3,\emptyset,3}=&1+t+2 t^2+5 t^3+14 t^4+42 t^5+132 t^6+(404+25 x) t^7\nonumber\\
&+\left(1119+286 x+25
x^2\right) t^8+\left(2762+1649 x+426 x^2+25 x^3\right)
t^9+\cdots
\end{align}

We are now in position to compute the generating functions 
\begin{math}\Qmm{a,k,\emptyset,\ell}=\Qm{a,k,0,\ell}=\Qm{a,k,\emptyset,\ell}\end{math} in the case where 
\begin{math}a,k, \ell > 0\end{math}. 
Again, we shall show that the polynomials \begin{math}\Qmmn{a,k,\emptyset,\ell}{n}\end{math} satisfy simple recursions.

When \begin{math}n \leq a+k+\ell\end{math}, there is no element of a \begin{math}\sigma \in \Sn{n}(132)\end{math} that 
can match \begin{math}\MMP(a,k,\emptyset,\ell)\end{math} in \begin{math}\sigma\end{math}. Thus \begin{math}\Qmmn{a,k,\emptyset,\ell}{n} = C_n\end{math} 
in such cases. Thus assume that \begin{math}n \geq a+k+\ell +1\end{math} and \begin{math}\sigma =\sigma_1 \ldots \sigma_n \in 
\Sn{n}(132)\end{math} is such that \begin{math}\sigma_i =n\end{math}. Clearly \begin{math}\sigma_i\end{math} cannot match 
\begin{math}\MMP(a,k,\emptyset,\ell)\end{math} in \begin{math}\sigma\end{math}. We then have 3 cases. 
\begin{enumerate}[{\bf Case }\bf 1.]
	\item \begin{math}i < k\end{math}. \\
	Clearly no \begin{math}\sigma_j\end{math} in \begin{math}A_i(\sigma) \end{math} can match \begin{math}\MMP(a,k,\emptyset,\ell)\end{math} since it cannot 
	have \begin{math}k\end{math} elements to its left which are larger than it. A \begin{math}\sigma_j \in B_i(\sigma)\end{math} 
	matches  \begin{math}\MMP(a,k,\emptyset,\ell)\end{math} in \begin{math}\sigma\end{math} if and only if it matches 
	\begin{math}\MMP(a,k-i,\emptyset,\ell)\end{math} in \begin{math}B_i(\sigma)\end{math}.  Thus such 
	permutations contribute \begin{math}\sum_{i=1}^{k-1}C_{i-1}\Qmmn{a,k-i,\emptyset,\ell}{n-i}\end{math} to 
	\begin{math}\Qmmn{a,k,\emptyset,\ell}{n}\end{math}.
	
	\item \begin{math}k\leq i \leq n-\ell\end{math}.\\
	For each peak \begin{math}\sigma_j \in A_i(\sigma)\end{math},   
	there are \begin{math}n-i\geq \ell\end{math} numbers in \begin{math}B_i(\sigma)\end{math} which are to its right and smaller than it. 
	Moreover, the number \begin{math}n\end{math} is to its right and is larger than it.  
	Thus \begin{math}\sigma_j\end{math} matches \begin{math}\MMP(a,k,\emptyset,\ell)\end{math} in \begin{math}\sigma\end{math} if and only if, 
	in the reduction of \begin{math}A_i(\sigma)\end{math}, its corresponding element matches 
	\begin{math}\MMP(a-1,k,\emptyset,0)\end{math}. For each peak \begin{math}\sigma_j \in B_i(\sigma)\end{math},   
	there are \begin{math}\geq k\end{math} numbers in \begin{math}A_i(\sigma) \cup \{n\}\end{math} which are to its left  and larger  than it 
	so that \begin{math}\sigma_j\end{math} matches \begin{math}\MMP(a,k,\emptyset,\ell)\end{math} in \begin{math}\sigma\end{math} if and only if 
	\begin{math}\sigma_j\end{math} matches \begin{math}\MMP(a,0,\emptyset,\ell)\end{math} in \begin{math}B_i(\sigma)\end{math}.  
	Thus such 
	permutations contribute \begin{math}\sum_{i=k}^{n-\ell}\Qmmn{a-1,k,\emptyset,0}{i-1}\Qmmn{a,0,\emptyset,\ell}{n-i}\end{math} to 
	\begin{math}\Qmmn{a,k,\emptyset,\ell}{n}\end{math}.
	
	\item \begin{math}i\geq n-\ell+1\end{math}. \\
	For each peak \begin{math}\sigma_j \in A_i(\sigma)\end{math},   
	there are \begin{math}n-i\geq \ell\end{math} numbers in \begin{math}B_i(\sigma)\end{math} which are to its right and smaller than it and 
	the number \begin{math}n\end{math} is to its right. 
	Thus \begin{math}\sigma_j\end{math} matches \begin{math}\MMP(a,k,\emptyset,\ell)\end{math} in \begin{math}\sigma\end{math} if and only if, 
	in the reduction of \begin{math}A_i(\sigma)\end{math}, its corresponding element matches 
	\begin{math}\MMP(a-1,k,\emptyset,\ell-(n-i))\end{math}. Clearly no element of \begin{math}B_i(\sigma) \end{math} can match \begin{math}\MMP(0,k,\emptyset,\ell)\end{math} since 
	it cannot have \begin{math}\ell\end{math} elements to its right which are smaller than it. 
	Thus such 
	permutations contribute \begin{math}\sum_{i=n-\ell+1}^{n} \Qmmn{a-1,k,\emptyset,\ell-(n-i)}{i-1} C_{n-i}\end{math} to 
	\begin{math}\Qmmn{a,k,\emptyset,\ell}{n}\end{math}. 
	
\end{enumerate}

It follows that for \begin{math}n \geq a+k+\ell+1\end{math},
\begin{eqnarray}
	\Qmmn{a,k,\emptyset,\ell}{n}&=&\sum_{i=1}^{k-1}C_{i-1}\Qmmn{a,k-i,\emptyset,\ell}{n-i}+\sum_{i=k}^{n-\ell}\Qmmn{a-1,k,\emptyset,0}{i-1}\Qmmn{a,0,\emptyset,\ell}{n-i}\nonumber\\
	&&+\sum_{i=n-\ell+1}^{n} \Qmmn{a-1,k,\emptyset,\ell-(n-i)}{i-1} C_{n-i}.
\end{eqnarray}

Multiplying both sides of the equation by \begin{math}t^n\end{math} and summing for \begin{math}n\geq 1\end{math} gives that
\begin{eqnarray}
	\Qmm{a,k,\emptyset,\ell}&=&\sum_{i=0}^{k+\ell-1}C_i t^i+t\sum_{i=1}^{k-1}C_{i-1}t^{i-1}(\Qmm{a,k-i,\emptyset,\ell}-\sum_{j=0}^{k+\ell-i-1}C_j t^j)\nonumber\\\nonumber
	&&+t(\Qmm{a-1,k,\emptyset,0}-\sum_{i=0}^{k-2}C_i t^i)(\Qmm{a,0,\emptyset,\ell}-\sum_{i=0}^{\ell-1}C_i t^i)\\
	&&+t\sum_{i=0}^{\ell-1}C_it^i(\Qmm{a-1,k,\emptyset,\ell-i}-\sum_{j=0}^{k+\ell-i-2}C_j t^j),
\end{eqnarray}
and we have the following theorem.

\begin{theorem}\label{theorem:10}
	For all \begin{math}a,k,\ell>0\end{math},
	\begin{eqnarray}
		\Qmm{a,k,\emptyset,\ell}&=&\sum_{i=0}^{k+\ell-1}C_i t^i+t\sum_{i=1}^{k-1}C_{i-1}t^{i-1}(\Qmm{a,k-i,\emptyset,\ell}-\sum_{j=0}^{k+\ell-i-1}C_j t^j)\nonumber\\\nonumber
		&&+t(\Qmm{a-1,k,\emptyset,0}-\sum_{i=0}^{k-2}C_i t^i)(\Qmm{a,0,\emptyset,\ell}-\sum_{i=0}^{\ell-1}C_i t^i)\\
		&&+t\sum_{i=0}^{\ell-1}C_it^i(\Qmm{a-1,k,\emptyset,\ell-i}-\sum_{j=0}^{k+\ell-i-2}C_j t^j).
	\end{eqnarray}
\end{theorem}

We list the first few terms of function \begin{math}\Qmm{a,k,\emptyset,\ell}\end{math} for \begin{math}1\leq a\leq 3\end{math} and \begin{math}1\leq k\leq \ell\leq 3\end{math}.

\begin{align}
\Qmm{1,1,\emptyset,1}=&1+t+2 t^2+5 t^3+(10+4 x) t^4+\left(17+21 x+4 x^2\right) t^5\nonumber\\\nonumber
&+\left(26+65 x+37
x^2+4 x^3\right) t^6+\left(37+155 x+176 x^2+57 x^3+4 x^4\right)
t^7\\\nonumber
&+\left(50+315 x+595 x^2+385 x^3+81 x^4+4 x^5\right) t^8\\
&+\left(65+574
x+1624 x^2+1750 x^3+736 x^4+109 x^5+4 x^6\right) t^9+\cdots
\\
\Qmm{1,1,\emptyset,2}=&1+t+2 t^2+5 t^3+14 t^4+(32+10 x) t^5+\left(62+60 x+10 x^2\right)
t^6\nonumber\\\nonumber
&+\left(107+209 x+103 x^2+10 x^3\right) t^7+\left(170+554 x+540 x^2+156
x^3+10 x^4\right) t^8\\
&+\left(254+1239 x+1995 x^2+1145 x^3+219 x^4+10
x^5\right) t^9+\cdots
\\
\Qmm{1,1,\emptyset,3}=&1+t+2 t^2+5 t^3+14 t^4+42 t^5+(104+28 x) t^6+\left(219+182 x+28 x^2\right)
t^7\nonumber\\\nonumber
&+\left(410+684 x+308 x^2+28 x^3\right) t^8\\
&+\left(704+1948 x+1720 x^2+462
x^3+28 x^4\right) t^9+\cdots
\\
\Qmm{1,2,\emptyset,2}=&1+t+2 t^2+5 t^3+14 t^4+42 t^5+(107+25 x) t^6+\left(233+171 x+25 x^2\right)
t^7\nonumber\\\nonumber
&+\left(450+669 x+286 x^2+25 x^3\right) t^8\\
&+\left(794+1968 x+1649 x^2+426
x^3+25 x^4\right) t^9+\cdots
\\
\Qmm{1,2,\emptyset,3}=&1+t+2 t^2+5 t^3+14 t^4+42 t^5+132 t^6+(359+70 x) t^7\nonumber\\\nonumber
&+\left(842+518 x+70
x^2\right) t^8+\left(1754+2184 x+854 x^2+70 x^3\right) t^9\\
&+\left(3332+6896
x+5238 x^2+1260 x^3+70 x^4\right) t^{10}+\cdots
\\
\Qmm{1,3,\emptyset,3}=&1+t+2 t^2+5 t^3+14 t^4+42 t^5+132 t^6+429 t^7+(1234+196 x) t^8\nonumber\\\nonumber
&+\left(3098+1568
x+196 x^2\right) t^9+\left(6932+7120 x+2548 x^2+196 x^3\right)
t^{10}\\
&+\left(14137+24117 x+16612 x^2+3724 x^3+196 x^4\right)
t^{11}+\cdots
\end{align}
\begin{align}
\Qmm{2,1,\emptyset,1}=&1+t+2 t^2+5 t^3+14 t^4+(33+9 x) t^5+\left(71+52 x+9 x^2\right)
t^6\nonumber\\\nonumber
&+\left(146+189 x+85 x^2+9 x^3\right) t^7+\left(294+557 x+443 x^2+127
x^3+9 x^4\right) t^8\\
&+\left(587+1463 x+1722 x^2+903 x^3+178 x^4+9 x^5\right)
t^9+\cdots
\\
\Qmm{2,1,\emptyset,2}=&1+t+2 t^2+5 t^3+14 t^4+42 t^5+(105+27 x) t^6+\left(235+167 x+27 x^2\right)
t^7\nonumber\\\nonumber
&+\left(494+637 x+272 x^2+27 x^3\right) t^8\\
&+\left(1004+1938 x+1489 x^2+404
x^3+27 x^4\right) t^9+\cdots
\\
\Qmm{2,1,\emptyset,3}=&1+t+2 t^2+5 t^3+14 t^4+42 t^5+132 t^6+(345+84 x) t^7\nonumber\\\nonumber
&+\left(800+546 x+84
x^2\right) t^8+\left(1724+2168 x+886 x^2+84 x^3\right) t^9\\
&+\left(3557+6803
x+5042 x^2+1310 x^3+84 x^4\right) t^{10}+\cdots
\\
\Qmm{2,2,\emptyset,2}=&1+t+2 t^2+5 t^3+14 t^4+42 t^5+132 t^6+(348+81 x) t^7\nonumber\\\nonumber
&+\left(811+538 x+81
x^2\right) t^8+\left(1747+2163 x+871 x^2+81 x^3\right) t^9\\
&+\left(3587+6826
x+5017 x^2+1285 x^3+81 x^4\right) t^{10}+\cdots
\\
\Qmm{2,2,\emptyset,3}=&1+t+2 t^2+5 t^3+14 t^4+42 t^5+132 t^6+429 t^7+(1178+252 x) t^8\nonumber\\\nonumber
&+\left(2848+1762
x+252 x^2\right) t^9+\left(6311+7395 x+2838 x^2+252 x^3\right)
t^{10}\\
&+\left(13201+24156 x+17011 x^2+4166 x^3+252 x^4\right)
t^{11}+\cdots
\\
\Qmm{2,3,\emptyset,3}=&1+t+2 t^2+5 t^3+14 t^4+42 t^5+132 t^6+429 t^7+1430 t^8+(4078+784 x)
t^9\nonumber\\\nonumber
&+\left(10236+5776 x+784 x^2\right) t^{10}+\left(23405+25349 x+9248
x^2+784 x^3\right) t^{11}\\
&+\left(50086+85921 x+57717 x^2+13504 x^3+784
x^4\right) t^{12}+\cdots
\\
\Qmm{3,1,\emptyset,1}=&1+t+2 t^2+5 t^3+14 t^4+42 t^5+(116+16 x) t^6+\left(308+105 x+16 x^2\right)
t^7\nonumber\\\nonumber
&+\left(807+446 x+161 x^2+16 x^3\right) t^8+\left(2108+1586 x+919 x^2+233
x^3+16 x^4\right) t^9\\
&+\left(5507+5169 x+4029 x^2+1754 x^3+321 x^4+16
x^5\right) t^{10}+\cdots
\\
\Qmm{3,1,\emptyset,2}=&1+t+2 t^2+5 t^3+14 t^4+42 t^5+132 t^6+(373+56 x) t^7\nonumber\\\nonumber
&+\left(998+376 x+56
x^2\right) t^8+\left(2615+1609 x+582 x^2+56 x^3\right) t^9\\
&+\left(6813+5701
x+3382 x^2+844 x^3+56 x^4\right) t^{10}+\cdots
\\
\Qmm{3,1,\emptyset,3}=&1+t+2 t^2+5 t^3+14 t^4+42 t^5+132 t^6+429 t^7+(1238+192 x) t^8\nonumber\\\nonumber
&+\left(3347+1323
x+192 x^2\right) t^9+\left(8798+5751 x+2055 x^2+192 x^3\right)
t^{10}\\
&+\left(22909+20509 x+12197 x^2+2979 x^3+192 x^4\right)
t^{11}+\cdots
\\
\Qmm{3,2,\emptyset,2}=&1+t+2 t^2+5 t^3+14 t^4+42 t^5+132 t^6+429 t^7+(1234+196 x) t^8\nonumber\\\nonumber
&+\left(3314+1352
x+196 x^2\right) t^9+\left(8643+5849 x+2108 x^2+196 x^3\right)
t^{10}\\
&+\left(22345+20688 x+12497 x^2+3060 x^3+196 x^4\right)
t^{11}+\cdots
\\
\Qmm{3,2,\emptyset,3}=&1+t+2 t^2+5 t^3+14 t^4+42 t^5+132 t^6+429 t^7+1430 t^8+(4190+672 x)
t^9\nonumber\\\nonumber
&+\left(11354+4770 x+672 x^2\right) t^{10}+\left(29639+21023 x+7452
x^2+672 x^3\right) t^{11}\\
&+\left(76326+75014 x+45194 x^2+10806 x^3+672
x^4\right) t^{12}+\cdots
\end{align}
\begin{align}
\Qmm{3,3,\emptyset,3}=&1+t+2 t^2+5 t^3+14 t^4+42 t^5+132 t^6+429 t^7+1430 t^8+4862 t^9\nonumber\\\nonumber
&+(14492+2304 x)
t^{10}+\left(39625+16857 x+2304 x^2\right) t^{11}\\\nonumber
&+\left(103494+75853 x+26361
x^2+2304 x^3\right) t^{12}\\
&+\left(265047+273660 x+163720 x^2+38169 x^3+2304
x^4\right) t^{13}+\cdots
\end{align}

From the functions list above, we see the coefficient of the biggest power of \begin{math}x\end{math}, \begin{math}x^{n-a-k-\ell}\end{math}, satisfies that \begin{math}\Qmmn{a,k,\emptyset,\ell}{n}\big\vert_{x^{n-a-\ell}}=\frac{(a+1)^2}{(a+k+1)(k+\ell+1)}\binom{a+2k}{k}\binom{a+2\ell}{\ell}\end{math} as predicted by \tref{0004}.

\section{Quadrant marked mesh patterns and hills of \(\Sn{n}(132)\)}

In this section, we want to study the generating function \begin{math}\Qmm{\emptyset,k,\emptyset,\ell}\end{math} 
where \begin{math}k, \ell \geq 0\end{math}. Note that by part (c) of Lemma \ref{p1}, a \begin{math}\sigma_j\end{math} can match 
\begin{math}\MMP(\emptyset,k,\emptyset,\ell)\end{math} in \begin{math}\sigma = \sigma_1 \ldots \sigma_n \in \Sn{n}(132)\end{math} 
if and only if \begin{math}\sigma_j\end{math} is a peak of \begin{math}\sigma\end{math} which is on the \begin{math}0^{\mathrm{th}}\end{math} diagonal (main diagonal). In terms of the Dyck path 
\begin{math}\Phi(\sigma)\end{math},  this is the number of steps \begin{math}DR\end{math} which start and end on the main diagonal 
which are called hills of the Dyck path by \cite{D,DS}. We will 
call such \begin{math}\sigma_i\end{math}, the hills of \begin{math}\sigma\end{math}.
Moreover, if \begin{math}\sigma_i =n\end{math} where 
\begin{math}i > 1\end{math}, then there can be no hills in \begin{math}A_i(\sigma)\end{math} as one can see from  
\fref{structure132}. 

First we shall show that \begin{math}\Qmmn{\emptyset,0,\emptyset,0}{n}\end{math} satisfies a simple recursion. 
We have two cases. 
\begin{enumerate}[{\bf Case }\bf 1.]
	\item \begin{math}\sigma_1 =n\end{math}.  \\
	In this case \begin{math}\sigma_1\end{math} matches \begin{math}\MMP(\emptyset,0,\emptyset,0)\end{math} which 
	contributes an \begin{math}x\end{math} and a \begin{math}\sigma_j\end{math} where \begin{math}j > 1\end{math} matches \begin{math}\MMP(\emptyset,0,\emptyset,0)\end{math} 
	in \begin{math}\sigma\end{math} if and only if \begin{math}\sigma_j\end{math}  matches \begin{math}\MMP(\emptyset,0,\emptyset,0)\end{math} 
	in \begin{math}B_1(\sigma)\end{math}. Thus such permutations contribute \begin{math}x\Qmmn{\emptyset,0,\emptyset,0}{n-1}\end{math} to 
	\begin{math}\Qmmn{\emptyset,0,\emptyset,0}{n}\end{math}.
	
	\item \begin{math}\sigma_i =n\end{math} where \begin{math}i \geq 2\end{math}.\\
	In this case no element of \begin{math}A_i(\sigma) \cup \{n\}\end{math} matches \begin{math}\MMP(\emptyset,0,\emptyset,0)\end{math} and 
	a \begin{math}\sigma_j\end{math} in \begin{math}B_i(\sigma)\end{math} matches \\\begin{math}\MMP(\emptyset,0,\emptyset,0)\end{math} 
	in \begin{math}\sigma\end{math} if and only if \begin{math}\sigma_j\end{math}  matches \begin{math}\MMP(\emptyset,0,\emptyset,0)\end{math} 
	in \begin{math}B_i(\sigma)\end{math}. Thus such permutations contribute \begin{math}C_{i-1}\Qmmn{\emptyset,0,\emptyset,0}{n-i}\end{math} to 
	\begin{math}\Qmmn{\emptyset,0,\emptyset,0}{n}\end{math}.
	
\end{enumerate}

It follows that for \begin{math}n \geq 1\end{math}, 
\begin{equation}
\Qmmn{\emptyset,0,\emptyset,0}{n}=x\Qmmn{\emptyset,0,\emptyset,0}{n-1}+\sum_{i=2}^{n}C_{i-1}\Qmmn{\emptyset,0,\emptyset,0}{n-i}.
\end{equation}
Multiplying both sides of the equation by \begin{math}t^n\end{math} and summing for \begin{math}n\geq 1\end{math} gives that
\begin{equation}
\Qmm{\emptyset,0,\emptyset,0}=1+t(C(t)+x-1)\Qmm{\emptyset,0,\emptyset,0}.
\end{equation}
Thus,
\begin{equation}
\Qmm{\emptyset,0,\emptyset,0}=\frac{1}{1-t(C(t)+x-1)}.
\end{equation}

Now we calculate \begin{math}\Qmm{\emptyset,k,\emptyset,\ell}\end{math} for the case when \begin{math}k > 0\end{math} and \begin{math}\ell \geq 0\end{math}. 
Notice by Lemma \ref{sym}, \\\begin{math}\Qmm{\emptyset,k,\emptyset,\ell} =\Qmm{\emptyset,\ell,\emptyset,k}\end{math}. 
Thus \begin{math}\Qmm{\emptyset,k,\emptyset,0} =\Qmm{\emptyset,0,\emptyset,k}\end{math}.

First we shall show that \begin{math}\Qmmn{\emptyset,k,\emptyset,\ell}{n}\end{math} satisfies a simple recursion. 
Clearly, if \begin{math}n \leq k+ \ell\end{math}, no element in a \begin{math}\sigma \in \Sn{n}(132)\end{math} can match 
\begin{math}\MMP(\emptyset,k,\emptyset,\ell)\end{math}. If \begin{math}n \geq k + \ell +1\end{math}, then we 
have two cases. 
\begin{enumerate}[{\bf Case }\bf 1.]
	\item \begin{math}\sigma_i =n\end{math} where \begin{math}i < k\end{math}.  \\
	In this case, even in the case where \begin{math}i =1\end{math}, \begin{math}\sigma_i =n\end{math} cannot match
	\begin{math}\MMP(\emptyset,k,\emptyset,\ell)\end{math}.  Moreover if 
	\begin{math}i > 1\end{math}, then no element in \begin{math}A_i(\sigma)\end{math} can match \begin{math}\MMP(\emptyset,k,\emptyset,\ell)\end{math}. 
	For any \begin{math}\sigma_j\end{math} in \begin{math}B_i(\sigma)\end{math}, all the elements in \begin{math}A_i(\sigma) \cup \{n\}\end{math} are to its 
	left and are greater than or equal to \begin{math}\sigma_j\end{math}. 
	Thus, a \begin{math}\sigma_j\end{math} in \begin{math}B_i(\sigma)\end{math} matches \begin{math}\MMP(\emptyset,0,\emptyset,0)\end{math} 
	in \begin{math}\sigma\end{math} if and only if \begin{math}\sigma_j\end{math}  matches \begin{math}\MMP(\emptyset,k-i,\emptyset,\ell)\end{math} 
	in \begin{math}B_1(\sigma)\end{math}. Thus such permutations contribute \begin{math}C_{i-1}\Qmmn{\emptyset,k-i,\emptyset,\ell}{n-1}\end{math} to 
	\begin{math}\Qmmn{\emptyset,k,\emptyset,\ell}{n}\end{math}.
	
	\item \begin{math}\sigma_i =n\end{math} where \begin{math}i \geq k\end{math}.\\
	In this case no element of \begin{math}A_i(\sigma) \cup \{n\}\end{math} matches \begin{math}\MMP(\emptyset,k,\emptyset,\ell)\end{math}.
	For any \begin{math}\sigma_j\end{math} in \begin{math}B_i(\sigma)\end{math}, all the elements in \begin{math}A_i(\sigma) \cup \{n\}\end{math} are to its 
	left and are greater than or equal to \begin{math}\sigma_j\end{math} so that such a \begin{math}\sigma_j\end{math} automatically has 
	\begin{math}k\end{math} elements to its left which are larger than \begin{math}\sigma_j\end{math}. Thus, 
	a \begin{math}\sigma_j\end{math} in \begin{math}B_i(\sigma)\end{math} matches \begin{math}\MMP(\emptyset,k,\emptyset,\ell)\end{math} 
	in \begin{math}\sigma\end{math} if and only if \begin{math}\sigma_j\end{math}  matches \begin{math}\MMP(\emptyset,0,\emptyset,\ell)\end{math} 
	in \begin{math}B_i(\sigma)\end{math}. Thus such permutations contribute \begin{math}C_{i-1}\Qmmn{\emptyset,0,\emptyset,\ell}{n-i}\end{math} to 
	\begin{math}\Qmmn{\emptyset,k,\emptyset,\ell}{n}\end{math}.
	
\end{enumerate}

It follows that for \begin{math}n \geq k + \ell+1\end{math}, 
\begin{equation}
\Qmmn{\emptyset,k,\emptyset,\ell}{n}=\sum_{i=1}^{k-1}C_{i-1}\Qmmn{\emptyset,k-i,\emptyset,\ell}{n-i}+\sum_{i=k}^{n}C_{i-1}\Qmmn{\emptyset,0,\emptyset,\ell}{n-i}.
\end{equation}
Multiplying both sides of the equation by \begin{math}t^n\end{math} and summing for \begin{math}n\geq 1\end{math} gives that
\begin{equation}
\Qmm{\emptyset,k,\emptyset,\ell}=1+t\sum_{i=1}^{k-1}C_{i-1}t^{i-1}\Qmm{\emptyset,k-i,\emptyset,\ell}+t(C(t)-\sum_{i=0}^{k-2}C_i t^i)\Qmm{\emptyset,0,\emptyset,\ell}.
\end{equation}
Thus, we have the following theorem.
\begin{theorem}\label{theorem:9}
	\begin{equation}
	\Qmm{\emptyset,0,\emptyset,0}=\frac{1}{1-t(C(t)+x-1)}.
	\end{equation}
	For \begin{math}k >0\end{math},
	\begin{equation}
	\Qmm{\emptyset,0,\emptyset,k}=\Qmm{\emptyset,k,\emptyset,0},
	\end{equation}
	and
	\begin{equation}
	\Qmm{\emptyset,k,\emptyset,0}=1+t\sum_{i=1}^{k-1}C_{i-1}t^{i-1}\Qmm{\emptyset,k-i,\emptyset,0}+t(C(t)-\sum_{i=0}^{k-2}C_i t^i)\Qmm{\emptyset,0,\emptyset,0}.
	\end{equation}
	For \begin{math}k, \ell > 0\end{math}, 
	\begin{equation}
	\Qmm{\emptyset,k,\emptyset,\ell}=1+t\sum_{i=1}^{k-1}C_{i-1}t^{i-1}\Qmm{\emptyset,k-i,\emptyset,\ell}+t(C(t)-\sum_{i=0}^{k-2}C_i t^i)\Qmm{\emptyset,0,\emptyset,\ell}.
	\end{equation}
\end{theorem}

By Corollary \ref{corollary:05}, we know that 
the highest power of \begin{math}x\end{math} that appears in \begin{math}\Qmn{\emptyset,k,\emptyset,\ell}{n}\end{math} is \begin{math}x^{n-k-\ell}\end{math} and 
that 
\begin{equation}
\Qmn{\emptyset,k,\emptyset,\ell}{n}\big\vert_{x^{n-k-\ell}}=C_k C_\ell.
\end{equation}

We start out by listing the first 10 terms in \begin{math}\Qmm{\emptyset,k,\emptyset,0}\end{math} for 
\begin{math}k = 0, \ldots, 5\end{math}. 

\begin{align}
\Qmm{\emptyset,0,\emptyset,0}=&
1+x t+\left(1+x^2\right) t^2+\left(2+2 x+x^3\right) t^3+\left(6+4 x+3 x^2+x^4\right) t^4\nonumber\\\nonumber
&+\left(18+13 x+6 x^2+4 x^3+x^5\right) t^5+\left(57+40
x+21 x^2+8 x^3+5 x^4+x^6\right) t^6\\\nonumber
&+\left(186+130 x+66 x^2+30 x^3+10 x^4+6 x^5+x^7\right) t^7\\\nonumber
&+\left(622+432 x+220 x^2+96 x^3+40 x^4+12 x^5+7 x^6+x^8\right)
t^8\\
& +\left(2120+1466 x+744 x^2+328 x^3+130 x^4+51 x^5+14 x^6+8 x^7+x^9\right) t^9+ \cdots 
\\
\Qmm{\emptyset,1,\emptyset,0}=&1+t+(1+x) t^2+\left(3+x+x^2\right) t^3+\left(8+4 x+x^2+x^3\right) t^4 \nonumber\\\nonumber
&+\left(24+11 x+5 x^2+x^3+x^4\right) t^5 +\left(75+35 x+14 x^2+6 x^3+x^4+x^5\right)
t^6\\\nonumber
&+\left(243+113 x+47 x^2+17 x^3+7 x^4+x^5+x^6\right) t^7\\\nonumber
&+\left(808+376 x+156 x^2+60 x^3+20 x^4+8 x^5+x^6+x^7\right) t^8+\\
& \left(2742+1276 x+532 x^2+204
x^3+74 x^4+23 x^5+9 x^6+x^7+x^8\right) t^9+ \cdots 
\\
\Qmm{\emptyset,2,\emptyset,0}=& 1+t+2 t^2+(3+2 x) t^3+\left(9+3 x+2 x^2\right) t^4+\left(26+11 x+3 x^2+2 x^3\right) t^5\nonumber\\\nonumber
&+\left(81+33 x+13 x^2+3 x^3+2 x^4\right) t^6\\\nonumber
&+\left(261+108
x+40 x^2+15 x^3+3 x^4+2 x^5\right) t^7\\\nonumber
&+\left(865+359 x+137 x^2+47 x^3+17 x^4+3 x^5+2 x^6\right) t^8\\
&+\left(2928+1220 x+468 x^2+168 x^3+54 x^4+19 x^5+3
x^6+2 x^7\right) t^9+\cdots
\\
\Qmm{\emptyset,3,\emptyset,0}=& 1+t+2 t^2+5 t^3+(9+5 x) t^4+\left(28+9 x+5 x^2\right) t^5+\nonumber\\\nonumber
&\left(85+33 x+9 x^2+5 x^3\right) t^6+\left(273+104 x+38 x^2+9 x^3+5 x^4\right)
t^7\\\nonumber
&+\left(901+349 x+123 x^2+43 x^3+9 x^4+5 x^5\right) t^8\\
&+\left(3042+1186 x+430 x^2+142 x^3+48 x^4+9 x^5+5 x^6\right) t^9+\cdots 
\\
\Qmm{\emptyset,4,\emptyset,0}=& 1+t+2 t^2+5 t^3+14 t^4+(28+14 x) t^5\nonumber\\\nonumber
&+\left(90+28 x+14 x^2\right) t^6+\left(283+104 x+28 x^2+14 x^3\right) t^7\\\nonumber
&+\left(931+339 x+118 x^2+28
x^3+14 x^4\right) t^8\\
&+\left(3132+1161 x+395 x^2+132 x^3+28 x^4+14 x^5\right) t^9+\cdots 
\\
\Qmm{\emptyset,5,\emptyset,0}=&1+t+2 t^2+5 t^3+14 t^4+42 t^5+(90+42 x) t^6\nonumber\\\nonumber
&+\left(297+90 x+42 x^2\right) t^7+\left(959+339 x+90 x^2+42 x^3\right) t^8\\
&+\left(3216+1133
x+381 x^2+90 x^3+42 x^4\right) t^9+\ldots
\end{align}

It is known that the sequence \begin{math}\{\Qmmn{\emptyset,0,\emptyset,0}{n}\big\vert_{x^0}\}_{n \geq 1}\end{math} 
is the Fine 
numbers which is sequence A000957 in the On-line Encyclopedia of Integer Sequences (OEIS) of \cite{oeis}.
Similarly, \begin{math}\{\Qmmn{\emptyset,0,\emptyset,0}{n}\big\vert_{x^1}\}_{n \geq 1}\end{math} is sequence 
A065601 in the OEIS.  However the sequence 
\begin{math}\{\Qmmn{\emptyset,0,\emptyset,0}{n}\big\vert{x^2}\}_{n \geq 2}\end{math} which starts out \\
\begin{math}1,0,3,6,21,66,220,744, \ldots\end{math} does not appear in the OEIS. This counts the number of 
Dyck paths with exactly \begin{math}2\end{math} hills. Nevertheless, it is easy to compute the generating 
function for the sequence by taking the second derivative of \begin{math}\Qmm{\emptyset,0,\emptyset,0}\end{math} 
with respect to \begin{math}x\end{math}, dividing it by \begin{math}2\end{math}, and setting \begin{math}x=0\end{math}.  In this case, the 
generating function is \begin{math}\frac{16t^2}{2(1+\sqrt{1-4t}+2t)^3}\end{math}.

The sequence \begin{math}\{\Qmmn{\emptyset,1,\emptyset,0}{n}\big\vert{x^0}\}_{n \geq 1}\end{math} which 
starts \begin{math}1,1,3,8,24,75,243,808, \ldots\end{math} is sequence A000958 in the OEIS and 
counts the number of ordered rooted trees with \begin{math}n\end{math} edges having the root of odd degree. 
None of  sequences \begin{math}\{\Qmmn{\emptyset,k,\emptyset,0}{n}\big\vert{x^0}\}_{n \geq 1}\end{math} where 
\begin{math}2 \leq k \leq 5\end{math} appear in the OEIS. 
None of  sequences \begin{math}\{\Qmmn{\emptyset,k,\emptyset,0}{n}\big\vert{x^1}\}_{n \geq 1}\end{math} where 
\begin{math}1 \leq k \leq 5\end{math} appear in the OEIS. 
In both cases, we can easily compute the generating functions of these sequences.

We list the first \begin{math}10\end{math} terms of function \begin{math}\Qmm{\emptyset,k,\emptyset,\ell}\end{math} for \begin{math}1\leq k\leq \ell\leq 3\end{math}.
\begin{align}
\Qmm{\emptyset,1,\emptyset,1}=&1+t+2 t^2+(4+x) t^3+\left(11+2 x+x^2\right) t^4+\left(32+7 x+2 x^2+x^3\right)
t^5\nonumber\\\nonumber
&+\left(99+22 x+8 x^2+2 x^3+x^4\right) t^6+\left(318+73 x+26 x^2+9 x^3+2
x^4+x^5\right) t^7\\\nonumber
&+\left(1051+246 x+90 x^2+30 x^3+10 x^4+2 x^5+x^6\right)
t^8\\
&+\left(3550+844 x+312 x^2+108 x^3+34 x^4+11 x^5+2 x^6+x^7\right)
t^9+\cdots
\\
\Qmm{\emptyset,1,\emptyset,2}=&1+t+2 t^2+5 t^3+(12+2 x) t^4+\left(35+5 x+2 x^2\right) t^5\nonumber\\\nonumber
&+\left(107+18 x+5
x^2+2 x^3\right) t^6+\left(342+60 x+20 x^2+5 x^3+2 x^4\right)
t^7\\\nonumber
&+\left(1126+206 x+69 x^2+22 x^3+5 x^4+2 x^5\right) t^8\\
&+\left(3793+714
x+246 x^2+78 x^3+24 x^4+5 x^5+2 x^6\right) t^9+\cdots
\\
\Qmm{\emptyset,1,\emptyset,3}=&1+t+2 t^2+5 t^3+14 t^4+(37+5 x) t^5+\left(113+14 x+5 x^2\right)
t^6\nonumber\\\nonumber
&+\left(358+52 x+14 x^2+5 x^3\right) t^7+\left(1174+180 x+57 x^2+14 x^3+5
x^4\right) t^8\\
&+\left(3943+634 x+204 x^2+62 x^3+14 x^4+5 x^5\right)
t^9+\cdots
\\
\Qmm{\emptyset,2,\emptyset,2}=&1+t+2 t^2+5 t^3+14 t^4+(38+4 x) t^5+\left(116+12 x+4 x^2\right)
t^6\nonumber\\\nonumber
&+\left(368+45 x+12 x^2+4 x^3\right) t^7+\left(1207+158 x+49 x^2+12 x^3+4
x^4\right) t^8\\
&+\left(4054+561 x+178 x^2+53 x^3+12 x^4+4 x^5\right)
t^9+\cdots
\\
\Qmm{\emptyset,2,\emptyset,3}=&1+t+2 t^2+5 t^3+14 t^4+42 t^5+(122+10 x) t^6+\left(386+33 x+10 x^2\right)
t^7\nonumber\\
&+\left(1259+128 x+33 x^2+10 x^3\right) t^8\\\nonumber
&+\left(4216+465 x+138 x^2+33
x^3+10 x^4\right) t^9+\cdots
\\
\Qmm{\emptyset,3,\emptyset,3}=&1+t+2 t^2+5 t^3+14 t^4+42 t^5+132 t^6+(404+25 x) t^7\nonumber\\
&+\left(1315+90 x+25
x^2\right) t^8+\left(4386+361 x+90 x^2+25 x^3\right)
t^9+\cdots
\end{align}

From the functions list above, we see the coefficient of the biggest power of \begin{math}x\end{math} satisfies that\\ \begin{math}\Qmn{\emptyset,k,\emptyset,\ell}{n}\big\vert_{x^{n-k-\ell}}=C_k C_\ell\end{math} as predicated by \coref{05}.

\section{The functions \(\Qm{0,k,0,0}\) and \(\Qm{0,k,0,\ell}\) for \(\ell\geq 1\)}
In this section, we will discuss how to compute the generating functions 
\begin{math}\Qm{0,k,0,0}\end{math} and\\
\begin{math}\Qm{0,k,0,\ell}\end{math} for \begin{math}\ell\geq 1\end{math}. These generating functions 
cannot be reduced to \begin{math}\Qmm{a,b,c,d}\end{math} so that we will use the \begin{math}\Psi\end{math} map to 
develop recursions for such functions. Since we are considering 
quadrant marked mesh patterns where neither the first nor third quadrants need to be empty, 
this means both peaks and non-peaks can match such patterns.

We start by considering generating functions of the form  \begin{math}\Qm{0,k,0,0}\end{math}. 
In this case, it will be useful to separately track peaks and non-peaks. 
Thus if \begin{math}\sigma = \sigma_1 \ldots \sigma_n \in \Sn{n}(123)\end{math}, then we will say 
that \begin{math}\sigma_i\end{math} matches the pattern \begin{math}\MMP(0,\binom{k_1}{k_2},0,0)\end{math} if \begin{math}\sigma_i\end{math} is a 
peak of \begin{math}\sigma\end{math} and it matches the pattern \begin{math}\MMP(0,k_1,0,0)\end{math} 
or \begin{math}\sigma_i\end{math} is a non-peak of \begin{math}\sigma\end{math} and it matches the pattern \begin{math}\MMP(0,k_2,0,0)\end{math}. 
We write \begin{math}\MMP(a,b,c,d)\end{math}-mch for short of \begin{math}\MMP(a,b,c,d)\end{math} pattern match, then we define 
\begin{equation}
\Qmz{k_1}{k_2}=\sum_{n=0}^{\infty}t^n \sum_{\sigma\in\Sn{n}(123)}x_0^{\#\ \MMP(0,k_1,0,0)\textnormal{-mch of peaks}}x_1^{\#\ \MMP(0,k_2,0,0)\textnormal{-mch of non-peaks}}
\end{equation}
and
\begin{equation}
\Qmzn{k_1}{k_2}{n}=\sum_{\sigma\in\Sn{n}(123)}x_0^{\#\ \MMP(0,k_1,0,0)\textnormal{-mch of peaks}}x_1^{\#\ \MMP(0,k_2,0,0)\textnormal{-mch of non-peaks}}.
\end{equation}
Clearly, \begin{math}\Qm{0,k,0,0}=\Qmzx{k}{k}{x,x}\end{math}. 

First we will compute  \begin{math}\Qmz{0}{0}\end{math}. 
When \begin{math}k_1=k_2=0\end{math}, in the generating function\\ \begin{math}\Qmz{0}{0}\end{math}, 
the variable \begin{math}x_0\end{math} is used to keep track of the number of peaks  in \begin{math}\sigma\end{math} and 
the variable \begin{math}x_1\end{math} is used to keep track of the number of non-peaks of \begin{math}\sigma\end{math}.  
Since the number of peaks and non-peaks in any \begin{math}\sigma \in \Sn{n}(123)\end{math} 
add up to \begin{math}n\end{math}, we can write \begin{math}\Qmz{0}{0}\end{math} in terms of \begin{math}\Qm{0,0,\emptyset,0}\end{math} which tracks the number of peaks.  That is, 
\begin{eqnarray}
	\Qmz{0}{0}&=&\Qmx{0,0,\emptyset,0}{t x_1,\frac{x_0}{x_1}}\nonumber\\
	&=&\frac{1-t x_0+t x_1 -\sqrt{(1-t x_0+t x_1)^2-4 t x_1}}{2 t x_1}.
\end{eqnarray}

When \begin{math}k_1\end{math} and \begin{math}k_2\end{math} are not both nonzero, we need to analyze the difference 
between \begin{math}\Psi^{-1}(P)\end{math} where \begin{math}P\end{math}  a Dyck path in \begin{math}\mathcal{D}_n\end{math} and \begin{math}\Psi^{-1}\end{math} on 
the \emph{lift of the path \begin{math}P\end{math}}, \begin{math}\lift(P)\end{math}, which is the Dyck path \begin{math}DPR \in 
\mathcal{D}_{n+1}\end{math}. The lifting operation is pictured in \fref{Dn123lift}. 
It is easy to see that the peaks of \begin{math}P\end{math} and \begin{math}\lift(P)\end{math} are labeled with the same 
numbers under \begin{math}\Psi^{-1}\end{math}. Since we label the rows and columns that do not contain peaks from left to right with the numbers of non-peaks 
in decreasing order under the map \begin{math}\Psi^{-1}\end{math}, it is easy to see that \begin{math}n+1\end{math} will be in the column of the first non-peak 
and that all the remaining shifts over one to the next column that does not contain a peak. 
This is illustrated in \fref{Dn123lift}.

The change in the labeling of the non-peaks is as follows. 
It is easy to see from \fref{Dn123lift} in the red cells in the case where 
\begin{math}\Psi^{-1}(P) = \sigma=(8,6,9,7,4,3,2,5,1)\in\Sn{9}(123)\end{math} and \begin{math}
\Psi^{-1}(\lift(P))=\sigma'=(8,6,10,9,4,3,2,7,1,5)\end{math}. It is easy to see that 
the action of lift does not change the number of elements in the second quadrant of the peak numbers; but increases the number of elements in the second quadrant of the non-peak numbers by \begin{math}1\end{math} since the number \begin{math}n+1\end{math} is in the second quadrant of the non-peaks.  In addition, 
the action of lift creates a new non-peak, namely, \begin{math}n+1\end{math}.  For our convenience, we write 
\begin{math}\mathrm{lift}(\sigma)\end{math} for the permutation \begin{math}\Psi^{-1}(\lift(P))\end{math}. 

\begin{figure}[ht]
	\centering
	\vspace{-1mm}
	\begin{tikzpicture}[scale =.5]
	\draw[help lines] (0,9)--(9,0);
	\draw[ultra thick] (0,9)--(0,7) -- (1,7)--(1,5)-- (4,5)--(4,3)--(5,3)--(5,2)--(6,2)--(6,1)--(8,1)--(8,0)--(9,0);
	\draw[help lines] (0,0) grid (9,9);
	\filllll{1}{8};\filllll{2}{6};
	\filllll{3}{9};\filllll{4}{7};
	\filllll{5}{4};\filllll{6}{3};
	\filllll{7}{2};\filllll{8}{5};
	\filllll{9}{1};
	\end{tikzpicture}	
	\begin{tikzpicture}[scale =.5]
	\path[fill=blue!20!white] (0,10)--(0,9) -- (9,0)--(10,0);
	\path[fill=red!50!white] (2,10)--(2,8)--(3,8)--(3,6)--(7,6)--(7,4)--(9,4)--(10,4)--
	(10,5)--(8,5)--(8,7)--(4,7)--(4,9)--(3,9)--(3,10);
	\draw[help lines] (0,10)--(10,0);
	\draw (-2,4.5) node {\begin{math}\Longrightarrow\end{math}};
	\path (-4,4.5);
	\draw[ultra thick] (0,10)--(0,7) -- (1,7)--(1,5)-- (4,5)--(4,3)--(5,3)--(5,2)--(6,2)--(6,1)--(8,1)--(8,0)--(10,0);
	\draw[help lines] (0,0) grid (10,10);
	\filllll{1}{8};\filllll{2}{6};
	\filllll{5}{4};\filllll{6}{3};
	\filllll{7}{2};
	\filllll{9}{1};
	
	\filllllg{3}{9};\filllllg{4}{7};\filllllg{8}{5};
	\filllll{3}{10};\filllll{4}{9};\filllll{8}{7};\filllll{10}{5};
	\end{tikzpicture}
	
	\caption{\(\sigma=(8,6,9,7,4,3,2,5,1)\) and \(\mathrm{lift}(\sigma)\)}
	\label{fig:Dn123lift}
\end{figure}
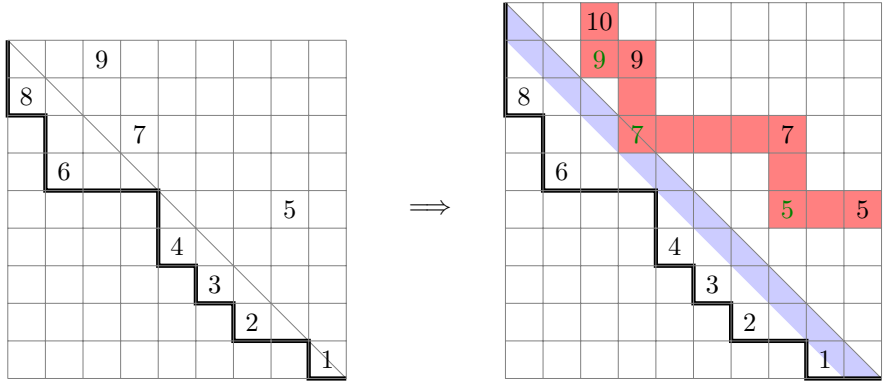

With the lift action, we can apply the Dyck path recursion for permutations in \begin{math}\Sn{n}(123)\end{math}. For any permutation \begin{math}\sigma\in\Sn{123}\end{math}, we suppose that the first return of the Dyck path \begin{math}\Psi(\sigma)\end{math} 
of \begin{math}\sigma\end{math} is located after the \thn{i} column. 
Then we can partition \begin{math}\sigma\end{math} according to the structure \begin{math}A_i(\sigma)\end{math} before the return on a height \begin{math}1\end{math} trapezoid and a Dyck path structure \begin{math}B_i(\sigma)\end{math} after the return as illustrated in \fref{123a}. 
Note that if \begin{math}\sigma_j\end{math} is in \begin{math}B(\sigma_i)\end{math}, then \begin{math}\sigma_j\end{math} is a peak of \begin{math}\sigma\end{math} if and only if it is 
peak of \begin{math}B_i(\sigma)\end{math}.

\begin{figure}[ht]
	\centering
	\subfigure[\label{fig:123a}]{\begin{tikzpicture}[scale =.5]
	\draw[help lines] (0,0) grid (8,8);
	\draw[ultra thick] (0,8) -- (0,4) -- (4,4) -- (4,0) -- (8,0) -- (0,8) -- (0,7) -- (3,4);
	\node at (1,5) {\begin{math}A_i\end{math}};
	\node at (3.5,-.5) {\begin{math}i\end{math}};
	\node at (5.5,1.5) {\begin{math}B_i\end{math}};
	\fillll{1}{0}{1};
	\fillll{0}{1}{1};
	\fillll{8}{0}{n};
	\fillll{0}{8}{n};
	\path (9,0);
	\end{tikzpicture}}
	\subfigure[\label{fig:123b}]{\begin{tikzpicture}[scale =.5]
	\draw[help lines] (0,0) grid (8,8);
	\draw[ultra thick] (0,4) rectangle (4,8);
	\draw[ultra thick] (4,0) rectangle (8,4);
	\node at (2,6) {\begin{math}\mathrm{lift}(A_i(\sigma))\end{math}};
	\node at (3.5,-.5) {\begin{math}i\end{math}};
	\node at (6,2) {\begin{math}B_i(\sigma)\end{math}};
	\fillll{1}{0}{1};
	\fillll{0}{1}{1};
	\fillll{8}{0}{n};
	\fillll{0}{8}{n};
	\path (9,0);
	\end{tikzpicture}	}
	\vspace{-3mm}
	\caption{Dyck path recursion of \(\Sn{n}(123)\)}
	\label{fig:structure123}
\end{figure}
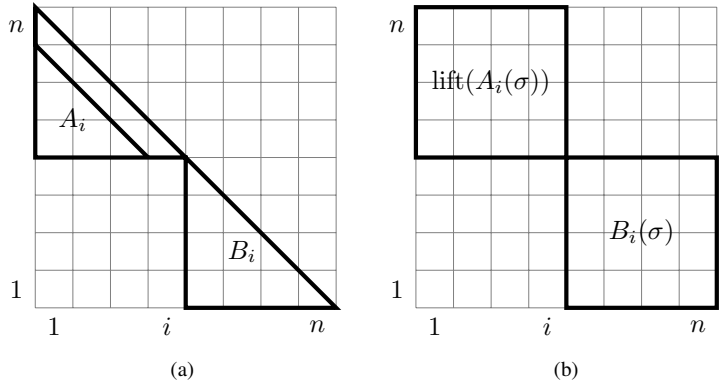

We first calculate the function \begin{math}\Qmz{k_1}{0}\end{math} for \begin{math}k_1>0\end{math}. We can develop 
simple recursions for \begin{math}\Qmzn{k_1}{0}{n}\end{math}. Note that when \begin{math}n \leq k_1\end{math}, 
then no peak in a \begin{math}\sigma \in \Sn{n}(123)\end{math} can match \begin{math}\MMP(0,k_1,0,0)\end{math} so 
that \begin{math}\Qmzn{k_1}{0}{n}=Q_{n,123}^{(0,\binom{0}{0},0,0)}(1,x_1)\end{math}.

Next assume that \begin{math}n \geq k_1 +1\end{math}. 
We are tracking the number of peaks matching \begin{math}\MMP(0,k_1,0,0)\end{math} by \begin{math}x_0\end{math} and tracking the number of non-peaks by \begin{math}x_1\end{math} in the polynomial \begin{math}\Qmzn{k_1}{0}{n}\end{math}. We will classify the permutations \begin{math}\sigma\in\Sn{n}(123)\end{math} 
according to the column \begin{math}i\end{math} of the first return of \begin{math}\Psi(\sigma)\end{math}. If 
the first return of \begin{math}\Psi(\sigma)\end{math} occurs in \thn{i} column of \begin{math}\sigma\end{math}, then 
we shall partition \begin{math}\sigma\end{math} into \begin{math}\mathrm{lift}(A_i(\sigma))\end{math} and \begin{math}B_i(\sigma)\end{math} 
as pictured in \fref{structure123}. We then have three cases. 
\begin{enumerate}[{\bf Case }\bf 1.]
	\item \begin{math}i=1\end{math}.\\
	In this case, \begin{math}\sigma_1=n\end{math} is a peak and the path \begin{math}\Psi(\sigma)\end{math} starts out \begin{math}DR \ldots \end{math}. Thus \begin{math}\sigma_1\end{math} does 
	not match \begin{math}\MMP(0,k_1,0,0)\end{math} in \begin{math}\sigma\end{math} in this case.  For any \begin{math}\sigma_j\end{math} in \begin{math}B_1(\sigma)\end{math}, \begin{math}n\end{math} is 
	always an element which is to the left of \begin{math}\sigma_j\end{math} which is larger than \begin{math}\sigma_j\end{math} so 
	that \begin{math}\sigma_j\end{math} matches  \begin{math}\MMP(0,k_1,0,0)\end{math} in \begin{math}\sigma\end{math} if and only if 
	\begin{math}\sigma_j\end{math} matches  \begin{math}\MMP(0,k_1-1,0,0)\end{math} in \begin{math}B_1(\sigma)\end{math}. Thus such permutations 
	contribute \\\begin{math}\Qmzn{k_1-1}{0}{n-1}\end{math} to \begin{math}\Qmzn{k_1}{0}{n}\end{math}.
	
	\item \begin{math}1 < i \leq k_1.\end{math}\\
	In this case, the only thing that has changed with respect to matches of \begin{math}\MMP(0,k_1,0,0)\end{math} 
	for peaks and the matches of \begin{math}\MMP(0,0,0,0)\end{math} for non-peaks in moving  
	to \begin{math}\mathrm{lift}(A_i(\sigma))\end{math} from \begin{math}A_i(\sigma)\end{math} is that we have one more non-peak. Clearly, 
	no peak of \begin{math}\sigma\end{math} that is in \begin{math}\lift(A_i(\sigma))\end{math} can match \begin{math}\MMP(0,k_1,0,0)\end{math} because it 
	will automatically have less than \begin{math}k_1\end{math} elements to the left which is larger than it. 
	Moreover, for  any \begin{math}\sigma_j\end{math} in \begin{math}B_i(\sigma)\end{math}, the elements in 
	the \begin{math}\mathrm{lift}(A_i(\sigma))\end{math} are elements to 
	the left of \begin{math}\sigma_j\end{math} which are larger than \begin{math}\sigma_j\end{math} so 
	that a peak \begin{math}\sigma_j\end{math} of \begin{math}\sigma\end{math} matches  \begin{math}\MMP(0,k_1,0,0)\end{math} in \begin{math}\sigma\end{math} if and only if 
	\begin{math}\sigma_j\end{math} matches  \begin{math}\MMP(0,k_1-i,0,0)\end{math} in \begin{math}B_i(\sigma)\end{math}. Thus such permutations 
	contribute \begin{math}x_1Q_{i-1,123}^{(0,\binom{0}{0},0,0)}(1,x_1)\Qmzn{k_1-i}{0}{n-i}\end{math} to \begin{math}\Qmzn{k_1}{0}{n}\end{math}.
	
	\item \begin{math}i >  k_1\end{math}.\\
	Again, the only thing that has changed with respect to matches of \begin{math}\MMP(0,k_1,0,0)\end{math} 
	for peaks and the matches of \begin{math}\MMP(0,0,0,0)\end{math} for non-peaks in moving  
	to \begin{math}\mathrm{lift}(A_i(\sigma))\end{math} from \begin{math}A_i(\sigma)\end{math} is that we have one more non-peak.
	A peak \begin{math}\sigma_j\end{math} of \begin{math}\sigma\end{math} that is in \begin{math}B_i(\sigma)\end{math} automatically matches 
	\begin{math}\MMP(0,k_1,0,0)\end{math} since all the elements in \begin{math}\lift(A_i(\sigma))\end{math} are to the left of \begin{math}\sigma_j\end{math} 
	and greater than \begin{math}\sigma_j\end{math}. Thus such permutations 
	contribute \begin{math}x_1\Qmzn{k_1}{0}{i-1}\Qmzn{0}{0}{n-i}\end{math}  to\\ \begin{math}\Qmzn{k_1}{0}{n}\end{math}.
	
\end{enumerate}

It follows that for \begin{math}n \geq k_1 +1\end{math},

\begin{eqnarray}
	\Qmzn{k_1}{0}{n}&=&\Qmzn{k_1-1}{0}{n-1}+x_1\sum_{i=2}^{k_1}
	Q_{i-1,123}^{(0,\binom{0}{0},0,0)}(1,x_1)\Qmzn{k_1-i}{0}{n-i}\nonumber\\
	&&+x_1\sum_{i=k_1+1}^{n}\Qmzn{k_1}{0}{i-1}\Qmzn{0}{0}{n-i}. 
\end{eqnarray}

Multiplying both sides of the equation by \begin{math}t^n\end{math} and summing for \begin{math}n\geq k_1+1\end{math} gives that
\begin{align}
&\Qmz{k_1}{0}-\sum_{j=0}^{k_1}t^j Q_{j,123}^{(0,\binom{0}{0},0,0)}(1,x_1)  \nonumber\\\nonumber
=\ &t(\Qmz{k_1-1}{0}-\sum_{j=0}^{k_1-1}t^j Q_{j,123}^{(0,\binom{0}{0},0,0)}(1,x_1))\ \\\nonumber
& +t x_1\sum_{i=2}^{k_1}t^{i-1}Q_{i-1,123}^{(0,\binom{0}{0},0,0)}(1,x_1) 
(\Qmz{k_1-i}{0}-\sum_{j=0}^{k_1-i}t^j Q_{j,123}^{(0,\binom{0}{0},0,0)}(1,x_1))\ \\
&+t x_1\Qmz{0}{0}(\Qmz{k_1}{0}-\sum_{j=0}^{k_1-1}t^j Q_{j,123}^{(0,\binom{0}{0},0,0)}(1,x_1))). 
\end{align}
Simplifying the equation gives
\begin{equation}
	\Qmz{k_1}{0}=\frac{\Delta_{k_1}(x_0,x_1,t)}{1-t x_1\Qmz{0}{0}},
\end{equation}
where 
\begin{eqnarray}
	\Delta_{k_1}(x_0,x_1,t) &=& K_{k_1}(x_1) + t \Qmz{k_1-1}{0} \nonumber\\\nonumber
	&&+ t x_1 \sum_{i=2}^{k_1} t^{i-1} Q_{i-1,123}^{(0,\binom{0}{0},0,0)}(1,x_1)\Qmz{k_1-i}{0}\\
	&&-  t x_1\Qmz{0}{0}(\sum_{j=0}^{k_1-1}t^j Q_{j,123}^{(0,\binom{0}{0},0,0)}(1,x_1))
\end{eqnarray}
and
\begin{eqnarray}
	K_{k_1}(x_1) &=& \sum_{j=0}^{k_1}t^j Q_{j,123}^{(0,\binom{0}{0},0,0)}(1,x_1) -t 
	\sum_{j=0}^{k_1-1}t^j Q_{j,123}^{(0,\binom{0}{0},0,0)}(1,x_1)\nonumber\\
	&&- t x_1\sum_{i=2}^{k_1}t^{i-1}Q_{i-1,123}^{(0,\binom{0}{0},0,0)}(1,x_1)
	(\sum_{j=0}^{k_1-i}t^j Q_{j,123}^{(0,\binom{0}{0},0,0)}(1,x_1)).
\end{eqnarray}

However, it is easy to see using our recursions for \begin{math}\Qmzn{k_1}{0}{n}\end{math} that 

\begin{eqnarray}
	0&=& \sum_{j=1}^{k_1}t^j Q_{j,123}^{(0,\binom{0}{0},0,0)}(1,x_1) -t 
	\sum_{j=0}^{k_1-1}t^j Q_{j,123}^{(0,\binom{0}{0},0,0)}(1,x_1)\nonumber\\
	&&- t x_1\sum_{i=2}^{k_1}t^{i-1}Q_{i-1,123}^{(0,\binom{0}{0},0,0)}(1,x_1)
	(\sum_{j=0}^{k_1-i}t^j Q_{j,123}^{(0,\binom{0}{0},0,0)}(1,x_1))
\end{eqnarray}
so that 
\begin{eqnarray}
	\Qmz{k_1}{0}&=&\frac{1}{1-t x_1\Qmz{0}{0}}\biggl(1+ t\Qmz{k_1-1}{0} \biggr. \nonumber\\\nonumber
	&& +t x_1\sum_{i=2}^{k_1}t^{i-1}Q_{i-1,123}^{(0,\binom{0}{0},0,0)}(1,x_1)\Qmz{k_1-i}{0}\\
	&&-\biggl. t x_1\Qmz{0}{0}(\sum_{j=0}^{k_1-1}t^j Q_{j,123}^{(0,\binom{0}{0},0,0)}(1,x_1))\biggr).
\end{eqnarray}

Next we will calculate the function \begin{math}\Qmz{0}{k_2}\end{math} for \begin{math}k_2>0\end{math}. 
In this case, we are tracking the number of non-peaks matching \begin{math}\MMP(0,k_2,0,0)\end{math} by \begin{math}x_1\end{math} and tracking the number of peaks by \begin{math}x_0\end{math}. We will classify the permutations \begin{math}\sigma\in\Sn{n}(123)\end{math} 
according to the column \begin{math}i\end{math} of the first return of \begin{math}\Psi(\sigma)\end{math}. If 
the first return of \begin{math}\Psi(\sigma)\end{math} occurs in \thn{i} column of \begin{math}\sigma\end{math}, then 
we shall partition \begin{math}\sigma\end{math} into \begin{math}\mathrm{lift}(A_i(\sigma))\end{math} and \begin{math}B_i(\sigma)\end{math} 
as pictured in \fref{structure123}. We then have three cases. 
\begin{enumerate}[{\bf Case }\bf 1.]
	\item \begin{math}i=1\end{math}.\\
	In this case \begin{math}\sigma_1=n\end{math} is a peak and the path \begin{math}\Psi(\sigma)\end{math} starts out \begin{math}DR \cdots \end{math}. 
	For any \begin{math}\sigma_j\end{math} in \begin{math}B_1(\sigma)\end{math}, \begin{math}n\end{math} is 
	always an element which is to the left of \begin{math}\sigma_j\end{math} which is larger than \begin{math}\sigma_j\end{math} so 
	that \begin{math}\sigma_j\end{math} matches  \begin{math}\MMP(0,k_2,0,0)\end{math} in \begin{math}\sigma\end{math} if and only if 
	\begin{math}\sigma_j\end{math} matches  \begin{math}\MMP(0,k_2-1,0,0)\end{math} in \begin{math}B_1(\sigma)\end{math}. Thus such permutations 
	contribute \begin{math}x_0\Qmzn{0}{k_2-1}{n-1}\end{math} to \begin{math}\Qmzn{0}{k_2}{n}\end{math}.
	
	\item \begin{math}1 < i \leq k_2\end{math}.\\
	In this case, the only thing that has changed with respect to matches of \begin{math}\MMP(0,0,0,0)\end{math} 
	for peaks and the matches of \begin{math}\MMP(0,k_2,0,0)\end{math} for non-peaks in moving  
	to \begin{math}\mathrm{lift}(A_i(\sigma))\end{math} from \begin{math}A_i(\sigma)\end{math} is that we have one more non-peak 
	which is in the first row. This new non-peak will be to the left of and larger than 
	any non-peak in \begin{math}\sigma\end{math}. None of the non-peaks in \begin{math}\lift(A_i(\sigma))\end{math} 
	match \begin{math}\MMP(0,k_2,0,0)\end{math} in \begin{math}\sigma\end{math} since no element in \begin{math}\lift(A_i(\sigma))\end{math} has \begin{math}k_2\end{math} elements 
	to its left. For  any \begin{math}\sigma_j\end{math} in \begin{math}B_i(\sigma)\end{math}, all the elements in 
	\begin{math}\mathrm{lift}(A_i(\sigma))\end{math} are elements to 
	the left of and larger than \begin{math}\sigma_j\end{math} so 
	that a non-peak \begin{math}\sigma_j\end{math} of \begin{math}\sigma\end{math} matches  \begin{math}\MMP(0,k_2,0,\ell)\end{math} in \begin{math}\sigma\end{math} if and only if 
	\begin{math}\sigma_j\end{math} matches  \begin{math}\MMP(0,k_2-i,0,\ell)\end{math} in \begin{math}B_i(\sigma)\end{math}. Thus such permutations 
	contribute \begin{math}Q_{i-1,123}^{(0,\binom{0}{0},0,0)}(x_0,1)\Qmzn{0}{k_2-i}{n-i}\end{math} to \begin{math}\Qmzn{0}{k_2}{n}\end{math}.
	
	\item \begin{math}i >  k_2\end{math}.\\
	Again, the only thing that has changed with respect to matches of \begin{math}\MMP(0,0,0,0)\end{math} 
	for peaks and the matches of \begin{math}\MMP(0,k_2,0,0)\end{math} for non-peaks in moving  
	to \begin{math}\mathrm{lift}(A_i(\sigma))\end{math} from \begin{math}A_i(\sigma)\end{math} is that we have one more non-peak which 
	is in the first row. This new non-peak will be to the left of and larger than 
	any non-peak in \begin{math}\sigma\end{math}. For any remaining non-peak \begin{math}\sigma_j\end{math} in \begin{math}\lift(A_i(\sigma))\end{math}, it will  
	match \begin{math}\MMP(0,k_2,0,0)\end{math} in \begin{math}\sigma\end{math} if and only if its corresponding non-peak 
	matches \begin{math}\MMP(0,k_2-1,0,0)\end{math} in \begin{math}A_i(\sigma)\end{math}. 
	A non-peak \begin{math}\sigma_j\end{math} of \begin{math}\sigma\end{math} that is in \begin{math}B_i(\sigma)\end{math} automatically matches 
	\begin{math}\MMP(0,k_2,0,0)\end{math} since all the elements in \begin{math}\lift(A_i(\sigma))\end{math} are to the left of \begin{math}\sigma_j\end{math} 
	and greater than \begin{math}\sigma_j\end{math}. Thus such permutations 
	contribute \begin{math}\Qmzn{0}{k_2-1}{i-1}\Qmzn{0}{0}{n-i}\end{math} to \begin{math}\Qmzn{0}{k_2}{n}\end{math}.
	
\end{enumerate}

It follows that for \begin{math}n \geq k_2+1\end{math}, 
\begin{eqnarray}
	\Qmzn{0}{k_2}{n}&=&x_0 \Qmzn{0}{k_2-1}{n-1}+\sum_{i=2}^{k_2-1}\Qmznx{0}{0}{i-1}{x_0,1}\Qmzn{0}{k_2-i}{n-i}\nonumber\\
	&&+\sum_{i=k_2}^{n}\Qmzn{0}{k_2-1}{i-1}\Qmzn{0}{0}{n-i}. 
\end{eqnarray}

From this recursion, one can compute in essentially the same way that we computed \\
\begin{math}\Qmz{k_1}{0}\end{math} that 
\begin{eqnarray}
	\Qmz{0}{k_2}&=&1+t x_0 \Qmz{0}{k_2-1}\nonumber\\\nonumber
	&&+t \sum_{i=2}^{k_2-1}t^{i-1} \Qmznx{0}{0}{i-1}{x_0,1}\Qmz{0}{k_2-i}\\\nonumber
	&&+t \Qmz{0}{0} (\Qmz{0}{k_2-1}\\
	&&-\sum_{i=0}^{k_2-2}t^i \Qmznx{0}{0}{i}{x_0,1}). 
\end{eqnarray}

Next we will show that the polynomials \begin{math}\Qmzn{k_1}{k_2}{n}\end{math} satisfy a simple 
recursion for any \begin{math}k_1,k_2>0\end{math} that involve the polynomials 
\begin{math}\Qmzn{a}{0}{n}\end{math} and \begin{math}\Qmzn{0}{b}{n}\end{math}. We first consider the case when \begin{math}k_1\geq k_2 \geq 1\end{math}.   
We will classify the permutations \begin{math}\sigma\in\Sn{n}(123)\end{math} 
according to the column \begin{math}i\end{math} of the first return of \begin{math}\Psi(\sigma)\end{math}. If 
the first return of \begin{math}\Psi(\sigma)\end{math} occurs in \thn{i} column of \begin{math}\sigma\end{math}, then 
we shall partition \begin{math}\sigma\end{math} into \begin{math}\mathrm{lift}(A_i(\sigma))\end{math} and \begin{math}B_i(\sigma)\end{math} 
as pictured in \fref{structure123}. We then have two cases. 
\begin{enumerate}[{\bf Case }\bf 1.]
	\item \begin{math}i < k_1\end{math}.\\
	In this case no  peak in \begin{math}\mathrm{lift}(A_i(\sigma))\end{math} can match 
	\begin{math}\MMP(0,k_1,0,0)\end{math}.  Thus  in \begin{math}\mathrm{lift}(A_i(\sigma))\end{math}, we need 
	only track the number of non-peaks which match \begin{math}\MMP(0,k_2,0,0)\end{math}. The new non-peak 
	that is created in going from \begin{math}A_i(\sigma)\end{math} to \begin{math}\mathrm{lift}(A_i(\sigma))\end{math} has no 
	elements to its left which are greater than it so it cannot match 
	\begin{math}\MMP(0,k_2,0,0)\end{math} since \begin{math}k_2 \geq 1\end{math}.  However the new non-peak is larger than 
	and to the left of any other non-peak in \begin{math}\mathrm{lift}(A_i(\sigma))\end{math}.  Thus for 
	all the remaining non-peaks in \begin{math}\mathrm{lift}(A_i(\sigma))\end{math}, they match 
	\begin{math}\MMP(0,k_2,0,0)\end{math} in \begin{math}\sigma\end{math} if and only if they match \begin{math}\MMP(0,k_2-1,0,0)\end{math} in 
	\begin{math}A_i(\sigma)\end{math}. Since all the elements of \begin{math}\mathrm{lift}(A_i(\sigma))\end{math} are larger than 
	and to the left of all the elements in \begin{math}B_i(\sigma)\end{math}, a peak in \begin{math}B_i(\sigma)\end{math} matches 
	\begin{math}\MMP(0,k_1,0,0)\end{math} in \begin{math}\sigma\end{math} if and only it matches \begin{math}\MMP(0,k_1-i,0,0)\end{math} in \begin{math}B_i(\sigma)\end{math} 
	and a non-peak in \begin{math}B_i(\sigma)\end{math} matches 
	\begin{math}\MMP(0,k_2,0,0)\end{math} in \begin{math}\sigma\end{math} if and only it matches \begin{math}\MMP(0,\max(k_2-i,0),0,0)\end{math} in \begin{math}B_i(\sigma)\end{math}.
	It follows that such permutations contribute 
	\begin{math}\Qmznx{0}{k_2-1}{i-1}{1,x_1}\Qmzn{k_1-i}{\max(k_2-i,0)}{n-i}\end{math} to 
	\begin{math}\Qmzn{k_1}{k_2}{n}\end{math}.
	
	\item \begin{math}i \geq  k_1\end{math}. \\
	By our analysis in Case 1, each non-peak in \begin{math}\lift(A_i(\sigma))\end{math}, except the new non-peak created 
	in going from \begin{math}A_i(\sigma)\end{math} to \begin{math}\mathrm{lift}(A_i(\sigma))\end{math}, matches 
	\begin{math}\MMP(0,k_2,0,0)\end{math} in \begin{math}\sigma\end{math} if and only if it matches  \\
	\begin{math}\MMP(0,k_2-1,0,0)\end{math} in \begin{math}\lift(A_i(\sigma))\end{math}.
	Each peak in \begin{math}\lift(A_i(\sigma))\end{math} matches \begin{math}\MMP(0,k_1,0,0)\end{math} in \begin{math}\sigma\end{math} if and only if it 
	matches \begin{math}\MMP(0,k_1,0,0)\end{math} in \begin{math}A_i(\sigma)\end{math}. Every peak in \begin{math}B_i(\sigma)\end{math} matches 
	\begin{math}\MMP(0,k_1,0,0)\end{math} in \begin{math}\sigma\end{math} and every non-peak matches \begin{math}\MMP(0,k_2,0,0)\end{math} in 
	\begin{math}\sigma\end{math}. It follows that such permutations contribute 
	\begin{math}\Qmznx{k_1}{k_2-1}{i-1}{x_0,x_1}\Qmzn{0}{0}{n-i}\end{math} to 
	\begin{math}\Qmzn{k_1}{k_2}{n}\end{math}.
	
\end{enumerate}

It follows that

\begin{eqnarray}
	\Qmzn{k_1}{k_2}{n}&=&\sum_{i=1}^{k_1-1}\Qmznx{0}{k_2-1}{i-1}{1,x_1}\Qmzn{k_1-i}{\max(k_2-i,0)}{n-i}\nonumber\\
	&&+\sum_{i=k_1}^{n}\Qmzn{k_1}{k_2-1}{i-1}\Qmzn{0}{0}{n-i}. 
\end{eqnarray}
Multiplying both sides of the equation by \begin{math}t^n\end{math} and summing for \begin{math}n\geq 1\end{math} gives that
\begin{eqnarray}
	\Qmz{k_1}{k_2}&=&1+t\sum_{i=1}^{k_1-1}\Qmznx{0}{k_2-1}{i-1}{1,x_1}t^{i-1}\Qmz{k_1-i}{\max(k_2-i,0)}\nonumber\\\nonumber
	&&+t \Qmz{0}{0}(\Qmz{k_1}{k_2-1}\\
	&&-\sum_{i=0}^{k_1-2}\Qmzn{0}{k_2-1}{i}t^i). 
\end{eqnarray}

Similarly, for \begin{math}k_2>k_1\geq 1\end{math}, we can do similar analysis and obtain that
\begin{eqnarray}
	\Qmz{k_1}{k_2}&=&1+t\sum_{i=1}^{k_2-1}\Qmznx{k_1}{0}{i-1}{x_0,1}t^{i-1}\Qmz{\max(k_1-i,0)}{k_2-i}\nonumber\\\nonumber
	&&+t \Qmz{0}{0}(\Qmz{k_1}{k_2-1}\\
	&&-\sum_{i=0}^{k_2-2}\Qmzn{k_1}{0}{i}t^i). 
\end{eqnarray}

\begin{theorem}\label{theorem:11}
	For all \begin{math}k_1,k_2>0\end{math}, we have
	\begin{eqnarray}
		\Qmz{k_1}{0}&=&\frac{1}{1-t x_1\Qmz{0}{0}}\left(1+t\Qmz{k_1-1}{0}\right.\nonumber\\\nonumber
		&&+t x_1\sum_{i=2}^{k_1-1}t^{i-1}\Qmznx{0}{0}{i-1}{1,x_1}\Qmz{k_1-i}{0}\\
		&&\left.-t x_1\Qmz{0}{0}\sum_{i=0}^{k_1-2}t^i\Qmznx{0}{0}{i-1}{1,x_1}\right),
	\end{eqnarray}
	\begin{eqnarray}
		\Qmz{0}{k_2}&=&1+t x_0 \Qmz{0}{k_2-1}\nonumber\\\nonumber
		&&+t \sum_{i=2}^{k_2-1}\Qmznx{0}{0}{i-1}{x_0,1}t^{i-1}\Qmz{0}{k_2-i}\\\nonumber
		&&+t \Qmz{0}{0} (\Qmz{0}{k_2-1}\\
		&&-\sum_{i=0}^{k_2-2}\Qmznx{0}{0}{i}{x_0,1}t^i). 
	\end{eqnarray}
	When \begin{math}k_1\geq k_2 \geq 1\end{math},
	\begin{eqnarray}
		\Qmz{k_1}{k_2}&=&1+t\sum_{i=1}^{k_1-1}\Qmznx{0}{k_2-1}{i-1}{1,x_1}t^{i-1}\Qmz{k_1-i}{\max(k_2-i,0)}\nonumber\\\nonumber
		&&+t \Qmz{0}{0}(\Qmz{k_1}{k_2-1}\\
		&&-\sum_{i=0}^{k_1-2}\Qmzn{0}{k_2-1}{i}t^i);
	\end{eqnarray}
	for \begin{math}k_2>k_1\geq 1\end{math},
	\begin{eqnarray}
		\Qmz{k_1}{k_2}&=&1+t\sum_{i=1}^{k_2-1}\Qmznx{k_1}{0}{i-1}{x_0,1}t^{i-1}\Qmz{\max(k_1-i,0)}{k_2-i}\nonumber\\\nonumber
		&&+t \Qmz{0}{0}(\Qmz{k_1}{k_2-1}\\
		&&-\sum_{i=0}^{k_2-2}\Qmzn{k_1}{0}{i}t^i). 
	\end{eqnarray}
	Finally, we have
	\begin{equation}
	\Qm{0,k,0,0}=\Qmzx{k}{k}{x,x}.
	\end{equation}
\end{theorem}

We list the first few terms of function \begin{math}\Qmm{0,k,0,0}\end{math} for \begin{math}k=1,\ldots,5\end{math}.
\begin{align}
\Qm{0,1,0,0}=&1+t+(1+x) t^2+\left(3 x+2 x^2\right) t^3+\left(9 x^2+5 x^3\right) t^4+\left(28
x^3+14 x^4\right) t^5\nonumber\\\nonumber
&+\left(90 x^4+42 x^5\right) t^6+\left(297 x^5+132
x^6\right) t^7\\
&+\left(1001 x^6+429 x^7\right) t^8+\left(3432 x^7+1430
x^8\right) t^9+\cdots
\\
\Qm{0,2,0,0}=&1+t+2 t^2+(3+2 x) t^3+\left(1+9 x+4 x^2\right) t^4+\left(5 x+27 x^2+10
x^3\right) t^5\nonumber\\\nonumber
&+\left(20 x^2+84 x^3+28 x^4\right) t^6+\left(75 x^3+270 x^4+84
x^5\right) t^7\\
&+\left(275 x^4+891 x^5+264 x^6\right) t^8+\left(1001 x^5+3003
x^6+858 x^7\right) t^9+\cdots
\end{align}
\begin{align}
\Qm{0,3,0,0}=&1+t+2 t^2+5 t^3+(9+5 x) t^4+\left(5+27 x+10 x^2\right) t^5\nonumber\\\nonumber
&+\left(1+25 x+81
x^2+25 x^3\right) t^6+\left(7 x+100 x^2+252 x^3+70 x^4\right) t^7\\\nonumber
&+\left(35
x^2+375 x^3+810 x^4+210 x^5\right) t^8\\
&+\left(154 x^3+1375 x^4+2673 x^5+660
x^6\right) t^9+\cdots
\\
\Qm{0,4,0,0}=&1+t+2 t^2+5 t^3+14 t^4+(28+14 x) t^5+\left(20+84 x+28 x^2\right)
t^6\nonumber\\\nonumber
&+\left(7+100 x+252 x^2+70 x^3\right) t^7+\left(1+49 x+400 x^2+784 x^3+196
x^4\right) t^8\\\nonumber
&+\left(9 x+245 x^2+1500 x^3+2520 x^4+588 x^5\right)
t^9\\\nonumber
&+\left(54 x^2+1078 x^3+5500 x^4+8316 x^5+1848 x^6\right) t^{10}\\
&+\left(273
x^3+4459 x^4+20020 x^5+28028 x^6+6006 x^7\right) t^{11}+\cdots
\\
\Qm{0,5,0,0}=&1+t+2 t^2+5 t^3+14 t^4+42 t^5+(90+42 x) t^6+\left(75+270 x+84 x^2\right)
t^7\nonumber\\\nonumber
&+\left(35+375 x+810 x^2+210 x^3\right) t^8+\left(9+245 x+1500 x^2+2520
x^3+588 x^4\right) t^9\\\nonumber
&+\left(1+81 x+1225 x^2+5625 x^3+8100 x^4+1764
x^5\right) t^{10}\\\nonumber
&+\left(11 x+486 x^2+5390 x^3+20625 x^4+26730 x^5+5544
x^6\right) t^{11}\\\nonumber
&+\left(77 x^2+2457 x^3+22295 x^4+75075 x^5+90090 x^6+18018
x^7\right) t^{12}\\
&+\left(440 x^3+11340 x^4+89180 x^5+273000 x^6+308880
x^7+60060 x^8\right) t^{13}+\cdots
\end{align}

\subsection{The function \(\Qm{0,k,0,\ell}\)}
In this section, we will show how to compute \begin{math}\Qm{0,k,0,\ell}\end{math} for small values 
of \begin{math}k\end{math} and \begin{math}\ell\end{math}.  In this case, we have not been able to 
obtain simple recursions for the polynomials \begin{math}\Qmn{0,k,0,\ell}{n}\end{math} because the process of 
going from \begin{math}A_i(\sigma)\end{math} to \begin{math}\lift(A_i(\sigma))\end{math} is not nicely behaved with respect 
to elements in the fourth quadrant of the graph of \begin{math}\sigma\end{math} centered at an 
element \begin{math}(j,\sigma_j)\end{math} when \begin{math}j \leq i\end{math}. However, in this case, we establish formulas 
for the coefficients of  \begin{math}\Qm{0,1,0,1}\end{math}, \begin{math}\Qm{0,2,0,1}\end{math} and \begin{math}\Qm{0,2,0,2}\end{math} by direct 
counting arguments.

Suppose that \begin{math}\sigma \in \Sn{n}(123)\end{math}. 
It is easy to see that no number in the top \begin{math}k\end{math} rows or the left-most \begin{math}k\end{math} columns in the graph of \begin{math}\sigma\end{math} can match \begin{math}\MMP(0,k,0,0)\end{math} 
in \begin{math}\sigma\end{math}. Similarly, it  is easy to see that no number in the bottom \begin{math}\ell\end{math} rows or right-most \begin{math}\ell\end{math} columns 
in the graph of \begin{math}\sigma\end{math} can match \begin{math}\MMP(0,0,0,\ell)\end{math} in \begin{math}\sigma\end{math}. 
Given \begin{math}\sigma_j\end{math}  in \begin{math}\sigma\end{math}, consider the graph of \begin{math}G(\sigma)\end{math} of \begin{math}\sigma\end{math} relative to the coordinate 
system centered at the point \begin{math}(j,\sigma_j)\end{math}.  Since \begin{math}\sigma\end{math} is 123-avoiding, \begin{math}\sigma_j\end{math} cannot 
have elements in both its first and third quadrant. \begin{math}\sigma_j\end{math} is a peak if and only if 
it has no elements in its third quadrant and \begin{math}\sigma_j\end{math} is non-peak if and only if it 
has at least one element in its third quadrant and no element in its first quadrant. 
Now suppose that \begin{math}\sigma_j\end{math} is a peak that is not in the top \begin{math}k\end{math}-rows or the left-most 
\begin{math}k\end{math} columns and is not in bottom \begin{math}\ell\end{math} rows or right-most \begin{math}\ell\end{math} columns. 
The elements in its first quadrant are the elements to the north-east 
of \begin{math}(j,\sigma_j)\end{math}. Since \begin{math}\sigma_j\end{math} has no elements in its third quadrant, it follows that 
the elements of \begin{math}\sigma\end{math} in the first \begin{math}k\end{math} columns must all be in the second quadrant for \begin{math}\sigma_j\end{math} 
and the elements in bottom \begin{math}\ell\end{math} rows of \begin{math}\sigma\end{math} must all be in the fourth quadrant for \begin{math}\sigma_j\end{math}. 
Thus \begin{math}\sigma_j\end{math} matches \begin{math}\MMP(0,k,0,\ell)\end{math}. Next suppose that \begin{math}\sigma_j\end{math} is a 
non-peak that is not in the top \begin{math}k\end{math}-rows or the left-most 
\begin{math}k\end{math} columns and is not in bottom \begin{math}\ell\end{math} rows or right-most \begin{math}\ell\end{math} columns. Then  
\begin{math}\sigma_j\end{math} has no elements in its first quadrant and the elements in its third quadrant 
are the elements south-west of \begin{math}(j,\sigma_j)\end{math}.  Again it follows that 
the elements of \begin{math}\sigma\end{math} in the top \begin{math}k\end{math} rows must all be in the second quadrant for \begin{math}\sigma_j\end{math} 
and the elements in right-most \begin{math}\ell\end{math} columns of \begin{math}\sigma\end{math} must all be in the fourth quadrant for \begin{math}\sigma_j\end{math}. 
Thus \begin{math}\sigma_j\end{math} matches \begin{math}\MMP(0,k,0,\ell)\end{math}. For example, in Figure \ref{fig:MMP0k0l}, we have pictured 
this situation in the case where \begin{math}k=2\end{math} and \begin{math}\ell =1\end{math} where the red cells represent 
the cells that are not in the top \begin{math}k\end{math}-rows or the left-most 
\begin{math}k\end{math} columns and are not in bottom \begin{math}\ell\end{math} rows or right-most \begin{math}\ell\end{math} columns. Thus we 
have the following theorem.
\begin{theorem}\label{theorem:12}
	For any \begin{math}123\end{math}-avoiding permutation \begin{math}\sigma = \sigma_1 \ldots \sigma_n\end{math}, \begin{math}\sigma_j\end{math} matches \begin{math}\MMP(0,k,0,\ell)\end{math} 
	in \begin{math}\sigma\end{math} if and only if, in the graph \begin{math}G(\sigma)\end{math} of \begin{math}\sigma\end{math}, \begin{math}(j,\sigma_j)\end{math} does not lie 
	in the top \begin{math}k\end{math} rows or the bottom \begin{math}\ell\end{math} rows and it does not lie in the left-most 
	\begin{math}k\end{math} columns or the right-most \begin{math}\ell\end{math} columns. Thus 
	\begin{equation}
	\mmp^{(0,k,0,\ell)}(\sigma)=\bigg|\{j|k<j\leq n-\ell\textnormal{ and }k<\sigma_j\leq n-\ell\}\bigg|.
	\end{equation}
\end{theorem}

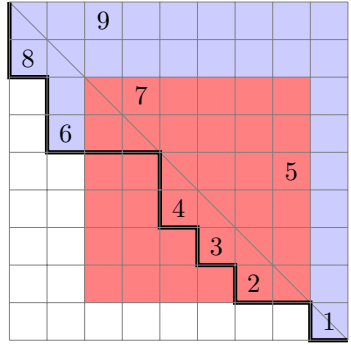
\begin{figure}[ht]
	\centering
	\vspace{-1mm}
	\begin{tikzpicture}[scale =.5]
	\path[fill,blue!20!white] (0,7) -- (1,7)--(1,5)-- (4,5)--(4,3)--(5,3)--(5,2)--(6,2)--(6,1)--(8,1)--(8,0)--(9,0)--(9,9)--(0,9);
	\fill[red!50!white] (2,1) rectangle (8,7);
	\draw[help lines] (0,9)--(9,0);
	\draw[ultra thick] (0,9)--(0,7) -- (1,7)--(1,5)-- (4,5)--(4,3)--(5,3)--(5,2)--(6,2)--(6,1)--(8,1)--(8,0)--(9,0);
	\draw[help lines] (0,0) grid (9,9);
	\filllll{1}{8};\filllll{2}{6};
	\filllll{3}{9};\filllll{4}{7};
	\filllll{5}{4};\filllll{6}{3};
	\filllll{7}{2};\filllll{8}{5};
	\filllll{9}{1};
	\end{tikzpicture}
	\caption{\(\MMP(0,2,0,1)\) matches of the permutation \(\sigma=869743251\)}
	\label{fig:MMP0k0l}
\end{figure}

Thus, for any permutation \begin{math}\sigma\in\Sn{n}(123)\end{math}, \tref{12} tells that we need to count the numbers in the 
rectangle that are obtained by deleting the top \begin{math}k\end{math} rows and bottom \begin{math}\ell\end{math} rows and deleting 
the left-most \begin{math}k\end{math} columns and the right-most \begin{math}\ell\end{math} columns. We have pictured this region 
in red and its complement in blue in Figure \ref{fig:corners}. We shall call 
the blue area the \begin{math}k,\ell\end{math}-frame area and the corners \begin{math}A \cup B \cup C \cup D\end{math} 
the \begin{math}k,\ell\end{math}-corner area. Now suppose that \begin{math}\sigma \in \Sn{n}(123)\end{math} and in the graph 
of \begin{math}\sigma\end{math}, there are \begin{math}r\end{math} elements in the \begin{math}k,\ell\end{math}-corner area and a total of \begin{math}s\end{math} numbers 
in the \begin{math}k,\ell\end{math}-frame area. In Figure \ref{fig:corners}, we have labeled the rectangles in 
the \begin{math}k,\ell\end{math}-frame area that are not part of the \begin{math}k,\ell\end{math}-corner area as \begin{math}E,F,G,H\end{math} starting at the 
top and proceeding clockwise. Suppose that in \begin{math}\sigma\end{math} there are \begin{math}a\end{math} elements in region \begin{math}A\end{math}, 
\begin{math}b\end{math} elements in region \begin{math}B\end{math}, \begin{math}c\end{math} elements in region \begin{math}C\end{math}, \begin{math}d\end{math} elements in region \begin{math}D\end{math}, 
\begin{math}e\end{math} elements in region \begin{math}E\end{math}, \begin{math}f\end{math} elements in region \begin{math}F\end{math}, \begin{math}g\end{math} elements in region \begin{math}G\end{math}, and \begin{math}h\end{math} elements 
region \begin{math}H\end{math}.
Then \begin{math}a+e+b =k\end{math}  and \begin{math}c+g+d = \ell\end{math} since there are \begin{math}k\end{math} elements of \begin{math}\sigma\end{math} in the top \begin{math}k\end{math} rows and \begin{math}\ell\end{math} elements of \begin{math}\sigma\end{math} in the bottom \begin{math}\ell\end{math} rows. Similarly, \begin{math}a+h+c = k\end{math} and \begin{math}b+f+d =\ell\end{math} since 
there are \begin{math}k\end{math} elements in the left-most \begin{math}k\end{math} columns and \begin{math}\ell\end{math} 
elements in the right-most \begin{math}\ell\end{math} columns. 
Adding these equation together we see that 
\begin{equation}
2(k+\ell) = 2a+2b+2c+2d+e+f+g+h = r +s.
\end{equation}

Thus we have the following theorem.
\begin{theorem}\label{theorem:13}
	For any \begin{math}k,\ell\geq 0\end{math}, \begin{math}n>k+\ell\end{math} and \begin{math}\sigma\in\Sn{n}(123)\end{math}, suppose there are \begin{math}r\end{math} numbers in the \begin{math}k,\ell\end{math}-corner area and \begin{math}s\end{math} numbers in the \begin{math}k,\ell\end{math}-frame area the graph of \begin{math}\sigma\end{math}. Then 
	\begin{equation}
	0\leq r\leq k+\ell, \ s=2(k+\ell)-r,\textnormal{ \ and \ }\mmp^{(0,k,0,\ell)}(\sigma)=n-s=n-2(k+\ell)+r.
	\end{equation}
	When \begin{math}n\leq k+\ell\end{math}, \begin{math}\mmp^{(0,k,0,\ell)}(\sigma)=0\end{math}.
\end{theorem}

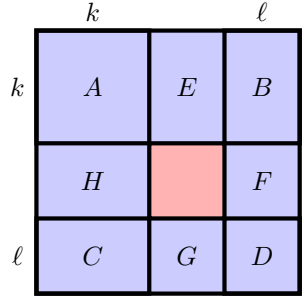
\begin{figure}[ht]
	\centering
	\vspace{-1mm}
	\begin{tikzpicture}[scale =.5]
	\draw[ultra thick, fill=blue!20!white] (0,0) rectangle (7,7);
	\draw[ultra thick, fill=red!30!white] (3,2) rectangle (5,4);
	\draw[ultra thick] (0,0)rectangle(3,2);
	\draw[ultra thick] (0,4)rectangle(3,7);
	\draw[ultra thick] (5,0)rectangle(7,2);
	\draw[ultra thick] (5,4)rectangle(7,7);
	\node at (1.5,5.5) {\begin{math}A\end{math}};
	\node at (1.5,1) {\begin{math}C\end{math}};
	\node at (6,5.5) {\begin{math}B\end{math}};
	\node at (6,1) {\begin{math}D\end{math}};
	\node at (4,5.5) {\begin{math}E\end{math}};
	\node at (6,3) {\begin{math}F\end{math}};
	\node at (4,1) {\begin{math}G\end{math}};
	\node at (1.5,3) {\begin{math}H\end{math}};
	\node at (-.5,5.5) {\begin{math}k\end{math}};
	\node at (1.5,7.5) {\begin{math}k\end{math}};
	\node at (-.5,1) {\begin{math}\ell\end{math}};
	\node at (6,7.5) {\begin{math}\ell\end{math}};
	\end{tikzpicture}
	\caption{The division of permutations in \(\Sn{n}(123)\) to count pattern \(\MMP(0,k,0,\ell)\) matches}
	\label{fig:corners}
\end{figure}

\tref{13} tells us that for each \begin{math}n> k+\ell\end{math}, the coefficients \begin{math}\Qm{0,k,0,\ell}\big|_{t^n}\end{math} have at most \begin{math}k+\ell+1\end{math} terms since the numbers in the \begin{math}k,\ell\end{math}-corner area can only range from \begin{math}0\end{math} to \begin{math}k+\ell\end{math}. 
In particular, the coefficient \begin{math}\Qm{0,k,0,\ell}\big|_{t^n x^{n-2(k+\ell)+r}}\end{math} equals the number 
of permutations in \begin{math}\sigma \in \Sn{n}(123)\end{math} with \begin{math}r\end{math} numbers in  the \begin{math}k,\ell\end{math}-corner area in the graph of \begin{math}\sigma\end{math}. \fref{cornerss} 
shows the squares in the \begin{math}k,\ell\end{math}-corner regions that we must consider for 
the generating functions \begin{math}\Qm{0,1,0,0}\end{math}, \begin{math}\Qm{0,2,0,0}\end{math}, \begin{math}\Qm{0,1,0,1}\end{math}, \begin{math}\Qm{0,2,0,1}\end{math}, 
and \begin{math}\Qm{0,2,0,2}\end{math}, respectively. In the next few subsections, we shall present and 
analyze the coefficients in such generating functions based on these observations.

\begin{figure}[ht]
	\centering
	\vspace{-1mm}
	\subfigure[\label{fig:0k0la}]{\begin{tikzpicture}[scale =.5]
	\draw[ultra thick] (0,0) rectangle (4,4);
	\draw[ultra thick, fill=blue!30!white] (0,3) rectangle (1,4);
	\draw[ultra thick] (0,3)--(4,3);
	\draw[ultra thick] (1,0)--(1,4);
	\node at (.5,3.5) {\begin{math}A\end{math}};
	\path (4.5,0);
	\path (-.5,0);
	\end{tikzpicture}}
	\subfigure[\label{fig:0k0lb}]{\begin{tikzpicture}[scale =.5]
	\draw[ultra thick] (0,0) rectangle (4,4);
	\draw[ultra thick, fill=blue!30!white] (0,2) rectangle (2,4);
	\draw[ultra thick] (0,2)--(4,2);
	\draw[ultra thick] (2,0)--(2,4);
	\draw[ultra thick] (0,3)--(4,3);
	\draw[ultra thick] (1,0)--(1,4);
	\node at (.5,3.5) {\begin{math}A\end{math}};
	\node at (1.5,3.5) {\begin{math}B\end{math}};
	\node at (.5,2.5) {\begin{math}C\end{math}};
	\node at (1.5,2.5) {\begin{math}D\end{math}};
	\path (-.5,0);
	\path (4.5,0);
	\end{tikzpicture}}
	\subfigure[\label{fig:0k0lc}]{\begin{tikzpicture}[scale =.5]
	\draw[ultra thick] (0,0) rectangle (4,4);
	\draw[ultra thick, fill=blue!30!white] (0,3) rectangle (1,4);
	\draw[ultra thick, fill=blue!30!white] (0,0) rectangle (1,1);
	\draw[ultra thick, fill=blue!30!white] (3,3) rectangle (4,4);
	\draw[ultra thick, fill=blue!30!white] (3,0) rectangle (4,1);
	\draw[ultra thick] (1,1) rectangle (3,3);
	\node at (.5,3.5) {\begin{math}A\end{math}};
	\node at (3.5,3.5) {\begin{math}B\end{math}};
	\node at (.5,.5) {\begin{math}C\end{math}};
	\node at (3.5,.5) {\begin{math}D\end{math}};
	\path (-.5,0);
	\path (4.5,0);
	\end{tikzpicture}}
	\subfigure[\label{fig:0k0ld}]{\begin{tikzpicture}[scale =.5]
	\draw[ultra thick] (0,0) rectangle (5,5);
	\draw[ultra thick, fill=blue!30!white] (0,3) rectangle (2,5);
	\draw[ultra thick, fill=blue!30!white] (0,0) rectangle (2,1);
	\draw[ultra thick, fill=blue!30!white] (4,3) rectangle (5,5);
	\draw[ultra thick, fill=blue!30!white] (4,0) rectangle (5,1);
	\draw[ultra thick] (0,1)--(5,1);
	\draw[ultra thick] (0,3)--(5,3);
	\draw[ultra thick] (0,4)--(5,4);
	\draw[ultra thick] (1,0)--(1,5);
	\draw[ultra thick] (2,0)--(2,5);
	\draw[ultra thick] (4,0)--(4,5);
	\node at (.5,4.5) {\begin{math}A\end{math}};
	\node at (1.5,4.5) {\begin{math}B\end{math}};
	\node at (4.5,4.5) {\begin{math}C\end{math}};
	\node at (.5,3.5) {\begin{math}D\end{math}};
	\node at (1.5,3.5) {\begin{math}E\end{math}};
	\node at (4.5,3.5) {\begin{math}F\end{math}};
	\node at (.5,.5) {\begin{math}G\end{math}};
	\node at (1.5,.5) {\begin{math}H\end{math}};
	\node at (4.5,.5) {\begin{math}I\end{math}};
	\path (-.5,0);
	\path (5.5,0);
	\end{tikzpicture}}
	\subfigure[\label{fig:0k0le}]{\begin{tikzpicture}[scale =.45]
	\draw[ultra thick] (0,0) rectangle (6,6);
	\draw[ultra thick, fill=blue!30!white] (0,4) rectangle (2,6);
	\draw[ultra thick, fill=blue!30!white] (0,0) rectangle (2,2);
	\draw[ultra thick, fill=blue!30!white] (4,4) rectangle (6,6);
	\draw[ultra thick, fill=blue!30!white] (4,0) rectangle (6,2);
	\draw[ultra thick] (0,1)--(6,1);
	\draw[ultra thick] (0,2)--(6,2);
	\draw[ultra thick] (0,4)--(6,4);
	\draw[ultra thick] (0,5)--(6,5);
	\draw[ultra thick] (1,0)--(1,6);
	\draw[ultra thick] (2,0)--(2,6);
	\draw[ultra thick] (4,0)--(4,6);
	\draw[ultra thick] (5,0)--(5,6);
	\node at (3,7.5) {columns};
	\node at (.5,6.5) {\begin{math}1\end{math}};
	\node at (1.5,6.5) {\begin{math}2\end{math}};
	\node at (4.5,6.5) {\begin{math}3\end{math}};
	\node at (5.5,6.5) {\begin{math}4\end{math}};
	\node at (-.5,5.5) {\begin{math}1\end{math}};
	\node at (-.5,4.5) {\begin{math}2\end{math}};
	\node at (-1.2,3) {rows};
	\node at (-.5,1.5) {\begin{math}3\end{math}};
	\node at (-.5,.5) {\begin{math}4\end{math}};
	\node at (.5,5.5) {\begin{math}A\end{math}};
	\node at (1.5,5.5) {\begin{math}B\end{math}};
	\node at (4.5,5.5) {\begin{math}C\end{math}};
	\node at (5.5,5.5) {\begin{math}D\end{math}};
	\node at (.5,4.5) {\begin{math}E\end{math}};
	\node at (1.5,4.5) {\begin{math}F\end{math}};
	\node at (4.5,4.5) {\begin{math}G\end{math}};
	\node at (5.5,4.5) {\begin{math}H\end{math}};
	\node at (.5,1.5) {\begin{math}I\end{math}};
	\node at (1.5,1.5) {\begin{math}J\end{math}};
	\node at (4.5,1.5) {\begin{math}K\end{math}};
	\node at (5.5,1.5) {\begin{math}L\end{math}};
	\node at (.5,.5) {\begin{math}M\end{math}};
	\node at (1.5,.5) {\begin{math}N\end{math}};
	\node at (4.5,.5) {\begin{math}O\end{math}};
	\node at (5.5,.5) {\begin{math}P\end{math}};
	\path (-.5,0);
	\path (6.5,0);
	\end{tikzpicture}}
	\caption{\(\Qm{0,1,0,0}\), \(\Qm{0,2,0,0}\), \(\Qm{0,1,0,1}\), \(\Qm{0,2,0,1}\) and \(\Qm{0,2,0,2}\)}
	\label{fig:cornerss}
\end{figure}
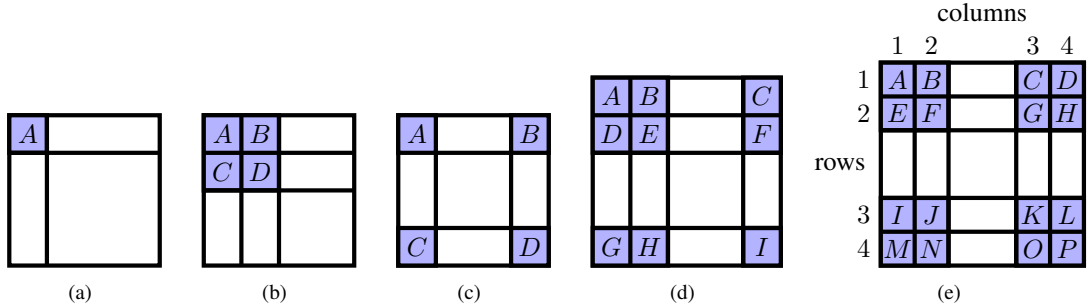

\subsubsection{\(\Qm{0,1,0,0}\big|_{t^n x^{n-2}}\) and \(\Qm{0,1,0,0}\big|_{t^n x^{n-1}}\)}
A formula for the generating function 
\begin{math}\Qm{0,1,0,0}\end{math} was calculated in Section \begin{math}5.1\end{math}. It follows from \tref{13} that there 
are exactly two terms in the polynomial \begin{math}\Qmn{0,1,0,0}{n}\end{math} for any \begin{math}n \geq 2\end{math}. 
Our next theorem shows that we can explicitly calculate these two terms.

\begin{theorem}\label{theorem:14}
	For \begin{math}n\geq 2\end{math}, \begin{math}\Qm{0,1,0,0}\big|_{t^n x^{n-2}}=C_n-C_{n-1}\end{math} and \begin{math}\Qm{0,1,0,0}\big|_{t^n x^{n-1}}=C_{n-1}\end{math}. Hence, 
	\begin{equation}
	\Qm{0,1,0,0}=(1+t-\frac{2t}{x}-\frac{1}{x^2})+(\frac{1}{x^2}+\frac{t}{x}+1)C(tx).
	\end{equation}
\end{theorem}

\begin{proof}
	By  \tref{13}, to calculate the coefficients of function \begin{math}\Qm{0,1,0,0}\end{math}, we only need to enumerate the 
	\begin{math}123\end{math}-avoiding permutations based on how many elements in the graph of \begin{math}\sigma\end{math} lie in 
	\begin{math}1,0\end{math}-corner area. In other words, referring to \fref{0k0la}, the permutations in \begin{math}\Sn{n}(123)\end{math} whose graphs have a number in square  \begin{math}A\end{math} contribute to the coefficient of \begin{math}t^n x^{n-1}\end{math} in \begin{math}\Qm{0,1,0,0}\end{math} and the permutations in \begin{math}\Sn{n}(123)\end{math} whose graphs have no element in square \begin{math}A\end{math} contribute to the coefficient of \begin{math}t^n x^{n-2}\end{math} in \begin{math}\Qm{0,1,0,0}\end{math}. Let  \begin{math}N_A(n)\end{math} be the number of permutations in \begin{math}\Sn{n}(123)\end{math} whose graph has a number  in square  \begin{math}A\end{math}. Then \begin{math}N_A(n)=C_{n-1}\end{math} since \begin{math}N_A(n)\end{math} counts those 
	\begin{math}\sigma\end{math} such that \begin{math}\sigma_1 =n\end{math} which means that the corresponding Dyck path \begin{math}\Psi(\sigma)\end{math} 
	has a peak at position \begin{math}A\end{math}. All such paths start out with \begin{math}DR\end{math}. Thus, \begin{math}\Qm{0,1,0,0}\big|_{t^n x^{n-1}}=C_{n-1}\end{math}. This means that the number of permutations in \begin{math}\Sn{n}(123)\end{math} which do not 
	have an element in square \begin{math}A\end{math} in its graphs is \begin{math}C_n-C_{n-1}\end{math}. Thus \begin{math}\Qm{0,1,0,0}\big|_{t^n x^{n-2}}=C_n-C_{n-1}\end{math}. It follows that 
	\begin{eqnarray}
		\Qm{0,1,0,0}&=&1+t+\sum_{n=2}^{\infty}t^n((C_n-C_{n-1})x^{n-2}+C_{n-1}x^{n-1})\nonumber\\\nonumber
		&=&1+t+\frac{C(tx)-1-xt}{x^2}+\frac{tC(tx)-t}{x}+C(tx)\\
		&=&(1+t-\frac{2t}{x}-\frac{1}{x^2})+(\frac{1}{x^2}+\frac{t}{x}+1)C(tx). \qedhere
	\end{eqnarray}
\end{proof}

\subsubsection{\(\Qm{0,2,0,0}\big|_{t^n x^{n-4}}\), \(\Qm{0,2,0,0}\big|_{t^n x^{n-3}}\) and \(\Qm{0,2,0,0}\big|_{t^n x^{n-2}}\)}

It follows from \tref{13} that there 
are exactly three terms in the polynomial \begin{math}\Qmn{0,1,0,0}{n}\end{math} for any \begin{math}n \geq 2\end{math}. 
Our next theorem shows that we can explicitly calculate these three terms.  

\begin{theorem}\label{theorem:15}
	For \begin{math}n\geq 4\end{math}, 
	\begin{eqnarray}
		\Qm{0,2,0,0}\big|_{t^n x^{n-4}}&=&C_n-3C_{n-1}+C_{n-2},\\
		\Qm{0,2,0,0}\big|_{t^n x^{n-3}}&=&3(C_{n-1}-C_{n-2}), \ \mbox{and} \\
		\Qm{0,2,0,0}\big|_{t^n x^{n-2}}&=&2C_{n-2}.
	\end{eqnarray}
\end{theorem}

\begin{proof}
	To find the coefficients of function \begin{math}\Qm{0,2,0,0}\end{math}, we need to enumerate the \begin{math}123\end{math}-avoiding permutations that have \begin{math}0\end{math}, \begin{math}1\end{math} or \begin{math}2\end{math} numbers in the \begin{math}2,0\end{math}-corner area as pictured in \fref{0k0lb}. 
	Let \begin{math}\phi_i(n)\end{math} be the number of permutations in \begin{math}\Sn{n}(123)\end{math} whose graphs have \begin{math}i\end{math} numbers in \begin{math}2,0\end{math}-corner area, colored blue in the picture, then in \begin{math}\Qm{0,2,0,0}\end{math}, \begin{math}\phi_0(n)\end{math} is the coefficient of \begin{math}t^n x^{n-4}\end{math}, \begin{math}\phi_1(n)\end{math} is the coefficient of \begin{math}t^n x^{n-3}\end{math} and \begin{math}\phi_2(n)\end{math} is the coefficient of \begin{math}t^n x^{n-2}\end{math}. 
	
	In this case, we can  use inclusion-exclusion to count the number of permutations 
	\begin{math}\sigma \in \Sn{n}(123)\end{math} whose graph has exactly \begin{math}r\end{math} elements in the \begin{math}2,0\end{math}-corner area. 
	We will labels the cells in \begin{math}2,0\end{math}-corner area as pictured in \fref{0k0lb}. For 
	\begin{math}S \subseteq  \{A,B,C,D\}\end{math}, we let \begin{math}N_S(n)\end{math} be the number of permutations \begin{math}\sigma\end{math} in \begin{math}\Sn{n}(123)\end{math} such 
	that there is an element in each square of  \begin{math}S\end{math} in the graph of \begin{math}\sigma\end{math}. Then it is easy to see 
	by inclusion-exclusion that 
	\begin{eqnarray}
		\phi_2(n)&=&N_{A,D}(n)+N_{B,C}(n),\\
		\phi_1(n)&=&N_A(n)+N_{B}(n)+N_{C}(n)+N_{D}(n)-2(N_{A,D}(n)+N_{B,C}(n)),\\
		\phi_0(n)&=&C_n-\phi_1(n)-\phi_2(n).
	\end{eqnarray}
	The problem is reduced to computing \begin{math}N_A(n)\end{math}, \begin{math}N_{B}(n)\end{math}, \begin{math}N_{C}(n)\end{math}, \begin{math}N_{D}(n)\end{math}, \begin{math}N_{A,D}(n)\end{math} and \begin{math}N_{B,C}(n)\end{math}. From the proof of \tref{14}, we have \begin{math}N_A(n)=C_{n-1}\end{math}. For \begin{math}N_{C}(n)\end{math}, we are counting 
	the number of permutations \begin{math}\sigma = \sigma_1 \ldots \sigma_n \in \Sn{n}(123)\end{math} such that 
	\begin{math}\sigma_1 =n-1\end{math} which means that \begin{math}P=\Psi(\sigma)\end{math} has a peak at position \begin{math}C\end{math}. Any such path 
	\begin{math}P\end{math} must start with \begin{math}DDR\end{math} and then we can remove the \begin{math}DR\end{math} at steps 2 and 3 and 
	obtain a Dyck path of length \begin{math}2n-2\end{math}. Thus \begin{math}N_{C}(n)=C_{n-1}\end{math}. For \begin{math}N_B(n)\end{math}, we are counting 
	the number of \begin{math}\sigma = \sigma_1 \ldots \sigma_n \in \Sn{n}(123)\end{math} such that \begin{math}\sigma_2 =n\end{math}.  It is easy to 
	see for for such \begin{math}\sigma\end{math}, \begin{math}\sigma\end{math} is 123-avoiding if and only if \begin{math}\sigma_1\sigma_3 \ldots \sigma_n\end{math} is 
	123-avoiding so that \begin{math}N_B(n) = C_{n-1}\end{math}. For \begin{math}N_D(n)\end{math}, we are counting the permutations 
	such that \begin{math}\sigma = \sigma_1 \ldots \sigma_n \in \Sn{n}(123)\end{math} such that \begin{math}\sigma_{2} = n-1\end{math}.  It follows 
	that \begin{math}\sigma_1 =n\end{math} since otherwise 123 would occur in \begin{math}\sigma\end{math}. Thus \begin{math}N_{D}(n)=N_{A,D}(n)=C_{n-2}\end{math}.
	For \begin{math}N_{B,D}(n)\end{math}, we are counting the permutations 
	such that \begin{math}\sigma = \sigma_1 \ldots \sigma_n \in \Sn{n}(123)\end{math} such that \begin{math}\sigma_{1} = n-1\end{math} and \begin{math}\sigma_2 =n\end{math}.  
	Hence  \begin{math}N_{B,C}(n) =C_{n-2}\end{math}. It follows that  
	\begin{eqnarray}
		\Qm{0,2,0,0}\big|_{t^n x^{n-2}}=\phi_2(n)&=&2C_{n-2},\\
		\Qm{0,2,0,0}\big|_{t^n x^{n-3}}=\phi_1(n)&=&3(C_{n-1}-C_{n-2}),\\
		\Qm{0,2,0,0}\big|_{t^n x^{n-4}}=\phi_0(n)&=&C_n-3C_{n-1}+C_{n-2}. \qedhere
	\end{eqnarray}
\end{proof}
It is technically possible to write the generating function \begin{math}\Qm{0,2,0,0}\end{math} in terms of the generating function of the Catalan numbers, \begin{math}C(x)\end{math}, like we did in \tref{14}. However the formula is messy so that 
we will not write it down here. 

\subsubsection{\(\Qm{0,1,0,1}\big|_{t^n x^{n-4}},\ \Qm{0,1,0,1}\big|_{t^n x^{n-3}}\) and \(\Qm{0,1,0,1}\big|_{t^n x^{n-2}}\)}

To find the coefficients of function \begin{math}\Qm{0,1,0,1}\end{math}, we need to enumerate the \begin{math}123\end{math}-avoiding permutations that have \begin{math}0\end{math}, \begin{math}1\end{math} or \begin{math}2\end{math} numbers in the \begin{math}1,1\end{math}-corner area as pictured in \fref{0k0lc}. 
Let \begin{math}\phi_i(n)\end{math} be the number of permutations in \begin{math}\Sn{n}(123)\end{math} whose graphs have \begin{math}i\end{math} numbers in \begin{math}1,1\end{math}-corner area, colored blue in the picture, then in \begin{math}\Qm{0,1,0,1}\end{math}, \begin{math}\phi_0(n)\end{math} is the coefficient of \begin{math}t^n x^{n-4}\end{math}, \begin{math}\phi_1(n)\end{math} is the coefficient of \begin{math}t^n x^{n-3}\end{math} and \begin{math}\phi_2(n)\end{math} is the coefficient of \begin{math}t^n x^{n-2}\end{math}. 

\begin{theorem}\label{theorem:16}
	For \begin{math}n\geq 4\end{math}, 
	\begin{eqnarray}
		\Qm{0,1,0,1}\big|_{t^n x^{n-4}}&=& C_n-2C_{n-1}+C_{n-2}-2, \\
		\Qm{0,1,0,1}\big|_{t^n x^{n-3}}&=&2C_{n-1}-2C_{n-2}+2, \ \mbox{and} \\
		\Qm{0,1,0,1}\big|_{t^n x^{n-2}}&=& C_{n-2}.
	\end{eqnarray}
\end{theorem}

\begin{proof}

	The four cells in the blue area are still denoted by \begin{math}A\end{math}, \begin{math}B\end{math}, \begin{math}C\end{math} and \begin{math}D\end{math}, though the positions 
	these cells are different from \fref{0k0lb}. For 
	\begin{math}S \subseteq  \{A,B,C,D\}\end{math}, we let \begin{math}N_S(n)\end{math} be the number of permutations \begin{math}\sigma\end{math} in \begin{math}\Sn{n}(123)\end{math} such 
	that there is an element in each square of  \begin{math}S\end{math} in the graph of \begin{math}\sigma\end{math}. Then 
	\begin{eqnarray}
		\phi_2(n)&=&N_{A,D}(n)+N_{B,C}(n),\\
		\phi_1(n)&=&N_A(n)+N_{B}(n)+N_{C}(n)+N_{D}(n)-2(N_{A,D}(n)+N_{B,C}(n)),\\
		\phi_0(n)&=&C_n-\phi_1(n)-\phi_2(n).
	\end{eqnarray}
	
	Thus we must compute \begin{math}N_A(n)\end{math}, \begin{math}N_{B}(n)\end{math}, \begin{math}N_{C}(n)\end{math}, \begin{math}N_{D}(n)\end{math}, \begin{math}N_{A,D}(n)\end{math} and \begin{math}N_{B,C}(n)\end{math}, which are different from \tref{15}. Assume that \begin{math}n \geq 4\end{math}. 
	By our previous results, \begin{math}N_A(n)=C_{n-1}\end{math}. For \begin{math}N_{C}(n)\end{math}, we are counting the number of \begin{math}\sigma = \sigma_1 \ldots \sigma_n \in \Sn{n}(123)\end{math} such that \begin{math}\sigma_1 =1\end{math}. The only such 
	\begin{math}\sigma\end{math} is \begin{math}\sigma=1n(n-1)\cdots 2\end{math} so that \begin{math}N_{C}(n)=1\end{math}. For \begin{math}N_{B}(n)\end{math}, we are counting the number of \begin{math}\sigma = \sigma_1 \ldots \sigma_n \in \Sn{n}(123)\end{math} such that \begin{math}\sigma_n =n\end{math}. The only such 
	\begin{math}\sigma\end{math} is \begin{math}\sigma=(n-1)\cdots 21n\end{math} so that \begin{math}N_{B}(n)=1\end{math}. For \begin{math}N_{D}(n)\end{math}, we are counting the number of \begin{math}\sigma = \sigma_1 \ldots \sigma_n \in \Sn{n}(123)\end{math} such that \begin{math}\sigma_n =1\end{math}. Clearly if we remove 1 from such a permutation and reduce the remaining numbers of \begin{math}1\end{math}, we obtain a 123-avoiding permutation in \begin{math}\Sn{n-1}(123)\end{math}. Thus \begin{math}N_{D}(n)=C_{n-1}\end{math}. For \begin{math}N_{B,C}\end{math}, we are counting the number of \begin{math}\sigma = \sigma_1 \ldots \sigma_n \in \Sn{n}(123)\end{math} such that \begin{math}\sigma_1 =1\end{math} and \begin{math}\sigma_n =n\end{math} which is impossible for \begin{math}n \geq 3\end{math}.  For \begin{math}N_{A,D}\end{math}, we are counting the number of \begin{math}\sigma = \sigma_1 \ldots \sigma_n \in \Sn{n}(123)\end{math} such that \begin{math}\sigma_1 =n\end{math} and \begin{math}\sigma_n =n\end{math}. For such 
	\begin{math}\sigma\end{math}, we can remove \begin{math}1\end{math} and \begin{math}n\end{math} to and reduce the remaining numbers by \begin{math}1\end{math} to obtain 
	a 123-avoiding permutation in \begin{math}\Sn{n}(123)\end{math}. Thus \begin{math}N_{A,D}= C_{n-2}\end{math}. 
	
	It follows that for \begin{math}n \geq 4\end{math}, 
	
	\begin{eqnarray}
		\Qm{0,1,0,1}\big|_{t^n x^{n-2}}=\phi_2(n)&=&C_{n-2},\\
		\Qm{0,1,0,1}\big|_{t^n x^{n-3}}=\phi_1(n)&=&2C_{n-1}-2C_{n-2}+2,\\
		\Qm{0,1,0,1}\big|_{t^n x^{n-4}}=\phi_0(n)&=&C_n-2C_{n-1}+C_{n-2}-2.\ \ \  \qedhere
	\end{eqnarray}
\end{proof}
\tref{16} gives the coefficient of \begin{math}t^n\end{math} in \begin{math}\Qm{0,1,0,1}\end{math} for \begin{math}n\geq 4\end{math}. One can easily 
compute the required coefficients at \begin{math}n=1,2,3\end{math} to obtain that 
\begin{eqnarray}
	\Qm{0,1,0,1}&=&1+t+2t^2+(4+x)t^3+\nonumber\\\nonumber
	&&\sum_{n\geq 4}t^n\left((C_n-2C_{n-1}+C_{n-2}-2)x^{n-4}\right.\\\nonumber
	&&\left.+(2C_{n-1}-2C_{n-2}+2)x^{n-3}+C_{n-2}x^{n-2}\right)\\\nonumber
	&=&1+t+2 t^2+(4+x) t^3+\left(4+8 x+2 x^2\right) t^4+\left(17 x+20 x^2+5 x^3\right)
	t^5\\\nonumber
	&&+\left(60 x^2+58 x^3+14 x^4\right) t^6+\left(205 x^3+182 x^4+42
	x^5\right) t^7\\\nonumber
	&&+\left(702 x^4+596 x^5+132 x^6\right) t^8\\
	&&+\left(2429 x^5+2004
	x^6+429 x^7\right) t^9+\cdots .
\end{eqnarray}  

\subsubsection{\(\Qm{0,2,0,1}\big|_{t^n x^{n-6}},\ \Qm{0,2,0,1}\big|_{t^n x^{n-5}},\ \Qm{0,2,0,1}\big|_{t^n x^{n-4}}\) \\and \(\Qm{0,2,0,1}\big|_{t^n x^{n-3}}\)}
In this section, we shall sketch the proof of the following theorem. 

\begin{theorem}\label{theorem:17}
	For \begin{math}n\geq 5\end{math}, 
	\begin{eqnarray}
		\Qm{0,2,0,1}\big|_{t^n x^{n-6}}&=&C_n-4C_{n-1}+4C_{n-2}-C_{n-3}-2n+6, \\
		\Qm{0,2,0,1}\big|_{t^n x^{n-5}}&=& 4C_{n-1}-9C_{n-2}+4C_{n-3}+2n-12, \\
		\Qm{0,2,0,1}\big|_{t^n x^{n-4}}&=&5C_{n-2}-5C_{n-3}+6, \ \mbox{and} \\
		\Qm{0,1,0,1}\big|_{t^n x^{n-3}}&=&2C_{n-3}.
	\end{eqnarray}
\end{theorem}

\begin{proof}
	To count the coefficients of function \begin{math}\Qm{0,2,0,1}\end{math}, we need to enumerate the \begin{math}123\end{math}-avoiding permutations that have \begin{math}0\end{math}, \begin{math}1\end{math}, \begin{math}2\end{math} or \begin{math}3\end{math} numbers in the \begin{math}2,1\end{math}-corner area. Referring to \fref{0k0ld}, let \begin{math}\phi_i(n)\end{math} be the number of permutations in \begin{math}\Sn{n}(123)\end{math} whose graphs have \begin{math}i\end{math} numbers in \begin{math}2,1\end{math}-corner area, colored blue in the picture, then in \begin{math}\Qm{0,2,0,1}\end{math}, \begin{math}\phi_0(n)\end{math} is the coefficient of \begin{math}t^n x^{n-6}\end{math}, \begin{math}\phi_1(n)\end{math} is the coefficient of \begin{math}t^n x^{n-5}\end{math}, \begin{math}\phi_2(n)\end{math} is the coefficient of \begin{math}t^n x^{n-4}\end{math} and \begin{math}\phi_3(n)\end{math} is the coefficient of \begin{math}t^n x^{n-3}\end{math}. 
	
	There are \begin{math}9\end{math} cells in the blue area denoted by \begin{math}A\end{math}, \begin{math}B\end{math}, \begin{math}C\end{math}, \begin{math}D\end{math}, \begin{math}E\end{math}, \begin{math}F\end{math}, \begin{math}G\end{math}, \begin{math}H\end{math}, \begin{math}I\end{math} in \fref{0k0ld}. For any \begin{math}S \subseteq \{A,B,C,D,E,F,G,H,I\}\end{math}, we let 
	\begin{math}N_S(n)\end{math} denote the number of \begin{math}\sigma \in \Sn{n}(123)\end{math} such that there is an element in 
	each cell of \begin{math}S\end{math} in the graph of \begin{math}\sigma\end{math}.  
	Let \begin{math}N_i(n)=\sum_{S \subseteq \{A,B,C,D,E,F,G,H,I\},|S|=i}N_{S}(n)\end{math}. 
	Then it follows from inclusion-exclusion that 
	\begin{eqnarray}
		\phi_3(n)&=&N_3(n),\\
		\phi_2(n)&=&N_2(n)-3N_3(n),\\
		\phi_1(n)&=&N_1(n)-2N_2(n)+3N_3(n),\\
		\phi_0(n)&=&C_n-\phi_1(n)-\phi_2(n)-\phi_3(n).
	\end{eqnarray}

	To compute \begin{math}N_1(n)\end{math}, we must compute \begin{math}N_S(n)\end{math} for 9 sets of size 1. To compute 
	\begin{math}N_2(n)\end{math}, we must compute \begin{math}N_{S}\end{math} for 18 allowable sets of size 2. To compute 
	\begin{math}N_3(n)\end{math}, we must compute \begin{math}N_{S}\end{math} for 6 allowable sets of size 3. It is tedious, but 
	not difficult to carry out required calculations. For space reasons, 
	we will not provide explanations 
	for each \begin{math}N_S(n)\end{math}, but we will simply list the results of our calculations. 
	
	For \begin{math}n\geq 5\end{math},
	\begin{eqnarray}
		N_A(n)&=&N_B(n)=N_D(n)=N_I(n)=C_{n-1}, \ \ N_E(n)=C_{n-2}, \\
		N_C(n)&=&N_G(n)=1, \ \ N_F(n)=N_H(n)=n-1, \ \textnormal{so}\\
		N_1(n)&=&4C_{n-1}+C_{n-2}+2n.
	\end{eqnarray}
	\begin{eqnarray}
		N_{A,E}(n)&=&N_{A,I}(n)=N_{B,D}(n)=N_{B,I}(n)=N_{D,I}(n)=C_{n-2}, \ \ N_{E,I}(n)=C_{n-3}, \\
		N_{A,F}(n)&=&N_{A,H}(n)=N_{B,F}(n)=N_{B,G}(n)=N_{C,D}(n)=N_{D,H}(n)=1, \\
		N_{C,E}(n)&=&N_{C,G}(n)=N_{C,H}(n)=N_{E,G}(n)=N_{F,G}(n)=N_{F,H}(n)=0, \ \textnormal{so}\\
		N_2(n)&=&5C_{n-2}+C_{n-3}+6.
	\end{eqnarray}
	\begin{eqnarray}
		N_{A,E,I}(n)&=&N_{B,D,I}(n)=C_{n-3},\\
		N_{A,F,H}(n)&=&N_{B,F,G}(n)=N_{C,D,H}(n)=N_{C,E,G}(n)=0, \ \textnormal{so}\\
		N_2(n)&=&2C_{n-3}, \ \textnormal{and}
	\end{eqnarray}
	\begin{eqnarray}
		\Qm{0,2,0,1}\big|_{t^n x^{n-3}}=\phi_3(n)&=&2C_{n-2},\\
		\Qm{0,2,0,1}\big|_{t^n x^{n-4}}=\phi_2(n)&=&5C_{n-2}-5C_{n-3}+6,\\
		\Qm{0,2,0,1}\big|_{t^n x^{n-5}}=\phi_1(n)&=&4C_{n-1}-9C_{n-2}+4C_{n-3}+2n-12,\\
		\Qm{0,2,0,1}\big|_{t^n x^{n-6}}=\phi_0(n)&=&C_n-4C_{n-1}+4C_{n-2}-C_{n-3}-2n+6. \qedhere
	\end{eqnarray}
\end{proof}

\tref{17} gives the coefficient of \begin{math}t^n\end{math} in \begin{math}\Qm{0,2,0,1}\end{math} for \begin{math}n\geq 5\end{math}. One can easily 
compute \\\begin{math}\Qmn{0,2,0,1}{n}\end{math} for \begin{math}n \leq 4\end{math} to obtain the following:
\begin{eqnarray}
	\Qm{0,2,0,1}&=&1+t+2t^2+5t^3+(12+2x)t^4\nonumber\\\nonumber
	&&+\sum_{n\geq 5}t^n\left(
	(C_n-4C_{n-1}+4C_{n-2}-C_{n-3}-2n+6)x^{n-6}\right.\\\nonumber
	&&+(4C_{n-1}-9C_{n-2}+4C_{n-3}+2n-12)x^{n-5}\\\nonumber
	&&\left.+(5C_{n-2}-5C_{n-3}+6)x^{n-4}+2C_{n-3} x^{n-3}
	\right)\\\nonumber
	&=&1+t+2 t^2+5 t^3+(12+2 x) t^4+\left(17+21 x+4 x^2\right) t^5\\\nonumber
	&&+\left(9+62 x+51
	x^2+10 x^3\right) t^6+\left(47 x+208 x^2+146 x^3+28 x^4\right) t^7\\\nonumber
	&&+\left(190
	x^2+700 x^3+456 x^4+84 x^5\right) t^8\\
	&&+\left(714 x^3+2393 x^4+1491 x^5+264
	x^6\right) t^9+\cdots .
\end{eqnarray}  

\subsubsection{\(\Qm{0,2,0,2}\big|_{t^n x^{n-8}},\ \Qm{0,2,0,2}\big|_{t^n x^{n-7}},\ \Qm{0,2,0,2}\big|_{t^n x^{n-6}},\\\Qm{0,2,0,2}\big|_{t^n x^{n-5}}\) and \(\Qm{0,2,0,2}\big|_{t^n x^{n-4}}\)}

In this section, we will sketch the proof of the following theorem. 

\begin{theorem}\label{theorem:18}For \begin{math}n\geq 7\end{math}, 
	\begin{eqnarray}
		\Qm{0,2,0,2}\big|_{t^n x^{n-8}}&=&C_n-6C_{n-1}+11C_{n-2}-6C_{n-3}+C_{n-4}-2n^2+16n-34,\\
		\Qm{0,2,0,2}\big|_{t^n x^{n-7}}&=&6C_{n-1}-24C_{n-2}+24C_{n-3}-6C_{n-4}+2n^2-28n+80,\\
		\Qm{0,2,0,2}\big|_{t^n x^{n-6}}&=&13C_{n-2}-30C_{n-3}+13C_{n-4}+12n-64,\\
		\Qm{0,2,0,2}\big|_{t^n x^{n-5}}&=&12C_{n-3}-12C_{n-4}+18,\\
		\Qm{0,2,0,2}\big|_{t^n x^{n-4}}&=&4C_{n-4}.
	\end{eqnarray}
\end{theorem}
\begin{proof}
	To count the coefficients of function \begin{math}\Qm{0,1,0,1}\end{math}, we need to enumerate the \begin{math}123\end{math}-avoiding permutations that have \begin{math}0\end{math}, \begin{math}1\end{math}, \begin{math}2\end{math}, \begin{math}3\end{math} or \begin{math}4\end{math} numbers in the \begin{math}2,2\end{math}-corner area. Referring to \fref{0k0le}, let \begin{math}\phi_i(n)\end{math} be the number of permutations in \begin{math}\Sn{n}(123)\end{math} whose graphs have \begin{math}i\end{math} numbers in \begin{math}2,2\end{math}-corner area, colored blue in the picture, then in \begin{math}\Qm{0,2,0,2}\end{math}, \begin{math}\phi_0(n)\end{math} is the coefficient of \begin{math}t^n x^{n-8}\end{math}, \begin{math}\phi_1(n)\end{math} is the coefficient of \begin{math}t^n x^{n-7}\end{math}, \begin{math}\phi_2(n)\end{math} is the coefficient of \begin{math}t^n x^{n-6}\end{math}, \begin{math}\phi_3(n)\end{math} is the coefficient of \begin{math}t^n x^{n-5}\end{math} and \begin{math}\phi_4(n)\end{math} is the coefficient of \begin{math}t^n x^{n-4}\end{math}. 
	
	There are \begin{math}16\end{math} cells in the blue area denoted by letters \begin{math}A\sim P\end{math} in \fref{0k0le}. 
	For any \begin{math}S \subseteq \{A, \ldots, P\}\end{math}, we let \begin{math}N_S(n)\end{math} denote the number of 
	\begin{math}\sigma \in \Sn{n}(123)\end{math} such that in the graph of \begin{math}\sigma\end{math}, there is an element in each square of \begin{math}S\end{math}. 
	We let \begin{math}N_i(n)=\sum_{S \subseteq \{A, \ldots,P\},|S|=i}N_{S}(n)\end{math}, then by inclusion-exclusion,
	\begin{eqnarray}
		\phi_4(n)&=&N_4(n),\\
		\phi_3(n)&=&N_3(n)-4N_4(n),\\
		\phi_2(n)&=&N_2(n)-3N_3(n)+6N_4(n),\\
		\phi_1(n)&=&N_1(n)-2N_2(n)+3N_3(n)-4N_4(n),\\
		\phi_0(n)&=&C_n-\phi_1(n)-\phi_2(n)-\phi_3(n)-\phi_4(n).
	\end{eqnarray}
	There are huge number positions and combination of positions in the \begin{math}2,2\end{math}-corner area. Since the selected letters should be in different rows and columns, we need to consider \begin{math}\binom{4}{i}^2 i!\end{math} combinations for calculation each \begin{math}N_i(n)\end{math}, i.e. \begin{math}16\end{math} singletons to calculate \begin{math}N_1(n)\end{math}, \begin{math}72\end{math} pairs to calculate \begin{math}N_2(n)\end{math}, \begin{math}96\end{math} groups of size \begin{math}3\end{math} to calculate \begin{math}N_3(n)\end{math} and \begin{math}24\end{math} groups of size \begin{math}4\end{math} to calculate \begin{math}N_4(n)\end{math}, totally \begin{math}208\end{math} separate calculations. Again we shall simply list the results of the relevant calculations 
	that we carried out. 
	We use the results that we have calculated for the cases that were covered in \tref{17} and only calculate the new combinations in this proof. We use ``New" to represent the sum of the new computations.

	\begin{eqnarray}
		N_C(n)&=&N_I(n)=n-1, \ \ N_k(n)=C_{n-2}, \\
		N_O(n)&=&N_L(n)=C_{n-1}, \ \ N_J(n)=N_G(n)=(n-2)^2, \ \textnormal{so}\\\nonumber
		N_1(n)&=&4C_{n-1}+C_{n-2}+2n+\textnormal{New}\\
		&=&6C_{n-1}+2C_{n-2}+2n^2-4n+6.
	\end{eqnarray}
	\begin{math}
	N_{C,E}(n),N_{O,H}(n),N_{B,I}(n),N_{L,N}(n),N_{G,A}(n),N_{G,P}(n),N_{J,A}(n),N_{J,P}(n),N_{G,B}(n),N_{G,L}(n),\\
	{\color{white}N_{J,E}(n),}N_{J,E}(n),N_{J,O}(n)=k-2, \\
	N_{C,H}(n),N_{I,N}(n),N_{C,L}(n),N_{I,O}(n),N_{C,P}(n),N_{I,P}(n),N_{O,D}(n),N_{L,M}(n)=1,\\
	N_{K,P}(n),N_{O,A}(n),N_{L,A}(n),N_{O,B}(n),N_{L,E}(n),N_{O,E}(n),N_{B,L}(n),N_{O,L}(n)=C_{n-2},\\
	N_{K,A}(n),N_{K,B}(n),N_{K,E}(n),N_{F,O}(n),N_{F,L}(n)=C_{n-3}, \ \  N_{K,F}(n)=C_{n-4},\\
	N_{C,F}(n),N_{K,H}(n),N_{F,I}(n),N_{K,N}(n),N_{C,I}(n),N_{C,J}(n),N_{H,J}(n),N_{G,I}(n),N_{N,G}(n),\\N_{C,M}(n),
	N_{D,I}(n),N_{C,N}(n),N_{H,I}(n),N_{G,D}(n),N_{J,M}(n),N_{G,J}(n),N_{G,M}(n),N_{J,D}(n),\\N_{K,D}(n),N_{K,M}(n)=0,
	\end{math}\\
	so\begin{eqnarray}
		N_2(n)&=&5C_{n-2}+C_{n-3}+6+\textnormal{New}\nonumber\\
		&=&13C_{n-2}+6C_{n-3}+C_{n-4}+12n-10.
	\end{eqnarray}
	
	To calculate \begin{math}N_3(n)\end{math}, other than calculating the new combinations in the \begin{math}96\end{math} enumerations, we calculate the cases by symmetry. Notice that there are \begin{math}4\end{math} columns and rows, namely, column \begin{math}1,2,3,4\end{math} and row \begin{math}1,2,3,4\end{math} in the \begin{math}2,2\end{math}-corner area, marked in \fref{0k0le}. In any combination of three letters, we are taking \begin{math}3\end{math} columns and \begin{math}3\end{math} rows. We let \begin{math}N_{(c_1c_2c_3,r_1r_2r_3)}(n)\end{math} be the contribution that we are taking \begin{math}3\end{math} letters from the columns \begin{math}c_1c_2c_3\end{math} and rows \begin{math}c_1c_2c_3\end{math}, then by symmetry of \begin{math}123\end{math}-avoiding permutations,
	\begin{eqnarray}
		N_{(123,123)}(n)&=&N_{(234,234)}(n),\\
		N_{(134,134)}(n)&=&N_{(124,124)}(n),\\
		N_{(123,124)}(n)&=&N_{(134,234)}(n)=N_{(124,123)}(n)=N_{(234,134)}(n),\\
		N_{(123,134)}(n)&=&N_{(124,234)}(n)=N_{(134,123)}(n)=N_{(234,124)}(n),\\
		N_{(123,234)}(n)&=&N_{(234,123)}(n),\\
		N_{(134,124)}(n)&=&N_{(124,134)}(n).
	\end{eqnarray}
	Then we calculate the \begin{math}6\end{math} cases:
	\begin{eqnarray}
		N_{A,F,K}(n)&=&N_{E,B,K}(n)=C_{n-4},\\ N_{A,J,G}(n)&=&N_{E,J,C}(n)=N_{B,G,I}(n)=N_{I,F,C}(n)=0,\ \textnormal{so}\\
		N_{(123,123)}(n)&=&N_{(234,234)}(n)=2C_{n-4};\\\nonumber
		N_{(124,124)}(n)&\textnormal{is}&N_3(n)\textnormal{ in \tref{17}, so}\\
		N_{(134,134)}(n)&=&N_{(124,124)}(n)=2C_{n-3};\\
		N_{A,F,O}(n)&=&N_{E,B,O}(n)=C_{n-3},\\ N_{A,N,G}(n)&=&N_{E,N,C}(n)=N_{M,B,G}(n)=N_{M,F,C}(n)=0,\ \textnormal{so}\\
		N_{(123,124)}(n)&=&N_{(134,234)}(n)=N_{(124,123)}(n)=N_{(234,134)}(n)=2C_{n-3};\\
		N_{A,J,O}(n)&=&N_{I,B,O}(n)=1, \\N_{A,N,K}(n)&=&N_{I,N,C}(n)=N_{M,B,K}(n)=N_{M,J,C}(n)=0,\ \textnormal{so}\\
		N_{(123,134)}(n)&=&N_{(124,234)}(n)=N_{(134,123)}(n)=N_{(234,124)}(n)=2;\\
		N_{E,J,O}(n)&=&1, \\N_{I,F,O}(n)&=&N_{E,N,K}(n)=N_{I,N,G}(n)=N_{M,F,K}(n)=N_{M,J,G}(n)=0,\ \textnormal{so}\\
		N_{(123,234)}(n)&=&N_{(234,123)}(n)=1;\\
		N_{A,G,P}(n)&=&N_{A,O,H}(n)=N_{E,C,P}(n)=N_{E,O,D}(n)=1, \\N_{M,C,H}(n)&=&N_{M,J,D}(n)=0,\ \textnormal{so}\\
		N_{(134,124)}(n)&=&N_{(124,134)}(n)=4,\textnormal{ and}
	\end{eqnarray}
	\begin{equation}
	N_3(n)=12C_{n-3}+4C_{n-4}+18
	\end{equation}
	
	To calculate \begin{math}N_4(n)\end{math}, we need to use all the \begin{math}4\end{math} columns and rows in the \begin{math}2,2\end{math}-corner area. To make things easier, we only consider the \begin{math}14\end{math} collections of \begin{math}4\end{math}-letter groups that avoid \begin{math}123\end{math}. We have\\
	\begin{math}N_{A,F,K,P}(n)=N_{A,F,O,L}(n)=N_{E,B,K,P}(n)=N_{E,B,O,L}(n)=C_{n-4},\end{math} and\\
	\begin{math}N_{A,J,G,O}(n),N_{I,B,G,P}(n),N_{E,J,C,P}(n),N_{A,J,O,H}(n),N_{A,N,G,L}(n),N_{I,N,C,H}(n),N_{M,B,G,L}(n),\\
	{\color{white}N_{J,E}(n),}N_{E,J,O,D}(n),N_{I,B,O,H}(n),N_{E,N,C,L}(n)=0,\end{math} so
	\begin{equation}
		N_4(n)=4C_{n-4}.
	\end{equation}
	With all \begin{math}N_1(n)\end{math}, \begin{math}N_2(n)\end{math}, \begin{math}N_3(n)\end{math} and \begin{math}N_4(n)\end{math} calculated, one can apply inclusion-exclusion and obtain that for \begin{math}n\geq 7\end{math},
	\begin{eqnarray}
		\Qm{0,2,0,2}\big|_{t^n x^{n-8}}&=&\phi_0(n)\nonumber\\
		&=&C_n-6C_{n-1}+11C_{n-2}-6C_{n-3}+C_{n-4}-2n^2+16n-34,\\
		\Qm{0,2,0,2}\big|_{t^n x^{n-7}}&=&\phi_1(n)\nonumber\\
		&=&6C_{n-1}-24C_{n-2}+24C_{n-3}-6C_{n-4}+2n^2-28n+80,\\
		\Qm{0,2,0,2}\big|_{t^n x^{n-6}}&=&\phi_2(n)=13C_{n-2}-30C_{n-3}+13C_{n-4}+12n-64,\\
		\Qm{0,2,0,2}\big|_{t^n x^{n-5}}&=&\phi_3(n)=12C_{n-3}-12C_{n-4}+18,\\
		\Qm{0,2,0,2}\big|_{t^n x^{n-4}}&=&\phi_4(n)=4C_{n-4}. \qedhere
	\end{eqnarray}
\end{proof}
Note that we have a lower bound, \begin{math}n\geq 7\end{math} for these formulas, which is because when \begin{math}n\leq 6\end{math}, \begin{math}N_{G,J}\neq 0\end{math} since permutation \begin{math}321654\end{math} matches both the positions \begin{math}G\end{math} and \begin{math}J\end{math}.

\tref{18} gives the coefficient of \begin{math}t^n\end{math} in \begin{math}\Qm{0,2,0,2}\end{math} for \begin{math}n\geq 7\end{math}. We calculated the initial \begin{math}7\end{math} coefficients by a computer program to obtain the following:
\begin{eqnarray}
	\Qm{0,2,0,2}&=&1+t+2t^2+5t^3+14t^4+(38+4x)t^5+(70+54x+8x^2)t^6\nonumber\\\nonumber
	&&+\sum_{n\geq 7}t^n\left(
	(C_n-6C_{n-1}+11C_{n-2}-6C_{n-3}+C_{n-4}-2n^2+16n-34)x^{n-8}\right.\\\nonumber
	&&+(6C_{n-1}-24C_{n-2}+24C_{n-3}-6C_{n-4}+2n^2-28n+80)x^{n-7}\\\nonumber
	&&+(13C_{n-2}-30C_{n-3}+13C_{n-4}+12n-64)x^{n-6}\\\nonumber
	&&\left.+(12C_{n-3}-12C_{n-4}+18)x^{n-5}
	+(4C_{n-4})x^{n-4}
	\right)\\\nonumber
	&=&1+t+2 t^2+5 t^3+14 t^4+(38+4 x) t^5+\left(70+54x+8 x^2\right)
	t^6\\\nonumber
	&&+\left(72+211 x+126 x^2+20 x^3\right) t^7+\left(36+314 x+670 x^2+354
	x^3+56 x^4\right) t^8\\\nonumber
	&&+\left(199 x+1190 x^2+2207 x^3+1098 x^4+168 x^5\right)
	t^9\\\nonumber
	&&+\left(838 x^2+4356 x^3+7492 x^4+3582 x^5+528 x^6\right)
	t^{10}\\\nonumber
	&&+\left(3241 x^3+15848 x^4+25951 x^5+12030 x^6+1716 x^7\right)
	t^{11}\\
	&&+\left(12180 x^4+57752 x^5+91158 x^6+41202 x^7+5720 x^8\right)
	t^{12}+\cdots .
\end{eqnarray}

\nocite{*}
\bibliographystyle{abbrvnat}
\bibliography{QQ123bib}
\label{sec:biblio}

\end{document}